\documentclass[11pt,reqno]{amsproc}

\usepackage[margin=1in]{geometry}

\usepackage[usenames,dvipsnames]{color}
\usepackage[colorlinks=true, pdfstartview=FitV, linkcolor=RoyalBlue,
            citecolor=ForestGreen, urlcolor=blue]{hyperref}

\usepackage{phaistos}

\usepackage{marvosym, fourier}

\title[Local and global existence for the stochastic Euler equations]
{Local and global existence of smooth solutions for the stochastic Euler equations with multiplicative noise \footnote{{\em Date:} \today} }

\author{Nathan E.~Glatt-Holtz}
\author{Vlad C.~Vicol}

\address{Department of Mathematics, Indiana University, 300 Swain East, Bloomington, IN 47405}
\email{negh@indiana.edu}

\address{Department of Mathematics, University of Chicago, 5734 University Avenue, Chicago, IL 60637}
\email{vicol@math.uchicago.edu}

\usepackage{amsmath, amsthm, amssymb,amsfonts,latexsym,verbatim,amsbsy}

\usepackage[usenames,dvipsnames]{color}

\usepackage{times}

\theoremstyle{plain}
\newtheorem{theorem}{Theorem}[section]
\newtheorem{definition}[theorem]{Definition}
\newtheorem{lemma}[theorem]{Lemma}
\newtheorem{proposition}[theorem]{Proposition}

\theoremstyle{definition}
\newtheorem{remark}[theorem]{Remark}

\renewcommand{\tilde}{\widetilde}
\def\DD{\mathcal D}
\def\RR{\mathbb R}
\def\RSZ{\mathcal R}
\def \UU{\mathfrak{U}}

\def\dW{d\mathcal{W}}
\def\WW{\mathcal{W}}
\def\XX{\mathbb{X}}
\def\WB{\mathbb{W}}
\def\mm{{m'}}

\def\locCon{\beta}

\newcommand{\indFn}[1]{1 \! \! \! 1_{#1}}
\newcommand{\E}{\mathbb{E}}
\newcommand{\Prb}{\mathbb{P}}

\def\curl{\mathop{\rm curl} \nolimits}

\newcommand{\norm}[2]{ \Vert #1 \Vert_{#2}} 
\begin{document}


\begin{abstract}
We establish the local existence of pathwise solutions for the stochastic Euler equations in a three-dimensional bounded domain with slip boundary conditions and a suitable nonlinear multiplicative noise. In the two-dimensional case we obtain the global existence of these solutions with additive or linear-multiplicative noise. Lastly, we show that, in the three dimensional case, the addition of  linear multiplicative noise provides a regularizing effect; the global existence of solutions occurs with high probability if the initial data is sufficiently small, or if the noise coefficient is sufficiently large.
\end{abstract}



\maketitle

\setcounter{tocdepth}{1}
\tableofcontents

\section{Introduction}
\label{sec:intro} \setcounter{equation}{0}

In this paper we address the well-posedness of the stochastic incompressible Euler equations with multiplicative noise, in a smooth bounded simply-connected domain $\DD \subset \RR^d$
\begin{align}
  &du + \left( u \cdot \nabla u + \nabla \pi \right) dt = \sigma(u) \dW,\label{eq:E:1}\\
  &\nabla \cdot u = 0,\label{eq:E:2}
\end{align}
where $d=2$ or $3$, $u$ denotes the velocity vector field, and $\pi$ the pressure scalar field.
Here $\WW$ is a cylindrical Brownian motion
and $\sigma(u) \dW$ can be written formally in the expansion
$\sum_{k \geq 1} \sigma_k(u) dW_k$ where $W_k$ are a collection of $1D$ independent
Brownian motions.
The system \eqref{eq:E:1}--\eqref{eq:E:2}
is supplemented with the classical slip
boundary condition
\begin{align}
  u|_{\partial \DD} \cdot n = 0,
  \label{eq:E:3}
\end{align}
where $n$ denotes the outward unit normal to the boundary $\DD$. Here $\partial \DD$ is taken to be sufficiently smooth. 
In order to emphasize the stochastic effects and for the simplicity of exposition
we do not include a deterministic forcing $f$ in \eqref{eq:E:1}, 
but note that all the results of this paper may be easily 
modified to include this more general case.

The Euler equations are the classical model
for the motion of an inviscid, incompressible, homogenous fluid.
The addition of stochastic terms to the governing equations
is commonly used to account for numerical, empirical, and
physical uncertainties in applications ranging from climatology to turbulence theory.  
In view of the wide usage of stochastics in fluid dynamics, 
there is an essential need to improve the mathematical foundations of the
stochastic partial differential equations of fluid flow, and in particular
to study inviscid models such as the stochastic Euler equations.

Even in the deterministic case, when $d=3$ the global existence and uniqueness of smooth solutions remains
a famously open problem  for the Euler equations, 
and also for their dissipative counterpart, the Navier-Stokes equations. 
There is a vast literature on the mathematical theory for the 
deterministic Euler equations; see for instance 
the books \cite{Chemin1998,MajdaBertozzi2002}, the recent surveys
\cite{BardosTiti2007,Constantin2007}, and references therein.
While the stochastic Navier-Stokes equation has been extensively studied
dating back to the seminal works \cite{BensoussanTemam1972, BensoussanTemam} and
subsequently in e.g. \cite{Viot1,Cruzeiro1,
CapinskiGatarek, Flandoli1, MikuleviciusRozovskii4,
ZabczykDaPrato2, Breckner,
BensoussanFrehse, BrzezniakPeszat,
MikuleviciusRozovskii2,
GlattHoltzZiane2, ConstantinIyer11, DebusscheGlattHoltzTemam1}, rather less has been
written concerning the stochastic Euler equations.  Most of the existing
literature on this subject treats only the two dimensional case, see e.g.
\cite{BessaihFlandoli, Bessaih1999, CapinskiCutland1, BrzezniakPeszat1,
Kim2, CruzeiroFlandoliMalliavin}.
To the best of our knowledge, there are only two works, \cite{MikuleviciusValiukevivius2, Kim1},
which consider the local existence of solutions in {\em dimension three}.
Both of these works consider only an {\em additive noise}, and treat \eqref{eq:E:1}--\eqref{eq:E:2}
on the {\em full space}, avoiding difficulties which naturally arise in the presence
of boundaries, due to the nonlocal nature of the pressure.

In this paper we establish three main results for the system \eqref{eq:E:1}--\eqref{eq:E:3}.
The first result addresses the local existence and uniqueness of
solutions in both two and three dimensions. 
From the probabilistic point of view we study {\em pathwise} solutions, that is probabilistically {\em strong} solutions
where the driving noise and associated filtration is given in advance, as part of the data.
From the PDE standpoint, we consider solutions which evolve continuously in the Sobolev space
$W^{m,p}(\DD)$, for any integer $m > d/p + 1$ and any $p\geq 2$, where $d = 2,3$. 

This local existence result covers a large class of nonlinear multiplicative noise structures in $\sigma( \cdot )$.  
In particular we can handle Nemytskii operators corresponding
to \emph{any} smooth function $g: \RR^{d} \rightarrow \RR^{d}$.  Here, heuristically speaking,
$$
\sigma(u) \dW(t,x) = g(u) \dot{\eta}(t,x),
$$
where $\dot{\eta}(t,x)$ is formally a Gaussian process
with the spatial-temporal correlation structure described by
$
 \E (\dot{\eta}(t,x) \dot{\eta}(s,y)) =\delta_{t-s} K(x,y)
$
for any sufficiently smooth correlation kernel $K$ on $\DD$.  We can also handle
functionals of the solution forced by white noise, and of course the classical cases
of additive and linear multiplicative noise.  See Section~\ref{sec:NoiseExamples} 
below for further details on these examples.

As noted above such results appears to be new in dimension three;  
this seems to be the first work to address (nonlinear) multiplicative noise, 
or to consider the evolution on a bounded domain.
Moreover, our method of proof is quite different from those employed in previous works for a two-dimensional bounded domain. More precisely,
we do not approximate solutions of the Euler system by those to the Navier-Stokes equations subject to Navier boundary conditions, 
and instead construct solutions to the Euler system directly.

In the second part of the paper we address some situations where the global existence of
spatially smooth solutions evolving in $W^{m,p}(\DD)$, with $m > p/d + 1$ can be established.  In the case of an additive noise ($\sigma(u) = \sigma$), when 
$d=2$ we show that the solutions obtained
in the first part of the paper are in fact global in time. 
To the best of our knowledge such results for smooth solutions was only known in the Hilbert space setting,
i.e. where $p=2$;  see \cite{BessaihFlandoli} for a bounded domain and \cite{Kim2,MikuleviciusValiukevivius2} 
where the evolution is considered over the whole space.  

Lastly, we turn to the issue of global existence of smooth pathwise solutions with multiplicative noise, in both $d=2,3$.
Obtaining the global existence of solutions for generic multiplicative noise $\sigma(u) \dW$ seems out of reach in view of some open problems
that already arise in the deterministic setting for $d=2$ (cf.~Remark~\ref{rem:2D:noMultiplicative} below). 
However, in the particular case of a {\em linear multiplicative} stochastic forcing, 
that is when $\sigma(u) \dW = \alpha udW$, where $W$ is a one-dimensional standard Brownian motion, we show
that the noise provides a damping effect on the pathwise behavior of solutions.
In the {\em three-dimensional} case we prove that for any $R \geq 1$:
$$
\Prb( u \textrm{ is global}) \geq 1 - R^{-1/4}, \quad \mbox{ whenever }
\|u_{0}\|_{W^{m,p}(\DD)} \leq \kappa(\alpha^{2}, R), 
$$  
where $\kappa$ is strictly positive and satisfies
$$
	\lim_{\alpha^{2} \to \infty} \kappa(\alpha^{2}, R) = \infty,
$$
 for every fixed $R \geq 1$. This may be viewed as a kind of global existence result in the {\em large noise asymptotic}.
Furthermore, in the {\em two-dimensional} case, we show that solutions are global in time with probability one, 
for {\em any} $\alpha \in \RR$, and independently of the size of the data.
Note that in both cases the linear multiplicative noise allows us to transform \eqref{eq:E:1}--\eqref{eq:E:3}
into an equivalent system for which the presence of an additional damping term becomes  evident. We can exploit
this random damping by using certain estimates for the exit times of geometric Brownian motion, and hence may 
establish the improved pathwise behavior of solutions.  We note that in the deterministic setting
the presence of sufficiently large damping is known to enhance the time of existence of solutions 
(see e.g.~\cite{PaicuVicol2010}), but in order to carry over these ideas to the stochastic setting we need to overcome a series of technical difficulties.

The starting point of our analysis of \eqref{eq:E:1}--\eqref{eq:E:3} is to establish some suitable
a priori estimates in the space $L^{2}(\Omega; L^{\infty}(0,T;W^{m,p}(\DD)))$.  Here obstacles arise both due to the presence of boundaries
and because we have to estimate stochastic integrals taking values in Banach spaces, i.e. $L^p(\DD)$ for $p > 2$. While we handle the convective terms using direct commutator estimates, in order to bound the pressure terms
we need to consider the regularity of solutions to an elliptic Neumann problem.
At first glance this seems to require bounding expressions involving first 
order derivatives of the solution on the boundary, i.e. $((u \cdot \nabla) u) \cdot n$, which would prevent the estimates from closing.
However, by exploiting a geometric insight from \cite{Temam1975}, one may obtain
suitable estimates for the pressure terms in $W^{m,p}(\DD)$. 
In order to treat the stochastic elements of the problem we follow the construction
of stochastic integrals given in e.g. \cite{Krylov1999, MikuleviciusRozovskii2001}.  
Estimates for the resulting stochastic terms 
are more technically demanding than in the Hibert space setting,
and are dealt with by a careful application of the Burkholder-Davis-Gundy inequality.
Note also that we obtain bounds on $u$ in 
$W^{m,p}(\DD)$ only up to a strictly positive stopping time $\tau$.  In contrast to the
deterministic setting, quantitative lower bounds on this $\tau$ are unavailable.  This
leads to further difficulties later in establishing the compactness necessary to pass to the limit 
within a class of approximating solutions of \eqref{eq:E:1}--\eqref{eq:E:3}.

With these a priori estimates in hand, we proceed to the first steps of the rigorous analysis. 
For this purpose, we introduce a Galerkin approximation scheme directly for \eqref{eq:E:1}--\eqref{eq:E:3},
which we use to construct solutions for the Hilbert space setting $p=2$.   We later employ a density and stability 
argument to obtain $W^{m,p}(\DD)$ solutions from the solutions constructed via the Galerkin scheme.
We believe that this Galerkin construction is more natural than in the previous works on the stochastic Euler equations on bounded domains
\cite{BessaihFlandoli,Bessaih1999,CapinskiCutland1,BrzezniakPeszat1},
which use approximations via the Navier-Stokes equations with Navier boundary conditions,
and exploits the vorticity formulation of the equations, a method which is mostly suitable for the two-dimensional case.

As with other nonlinear  SPDEs, we face the essential challenge of establishing sufficient
compactness in order to be able to pass to the limit in the class of Galerkin approximations;
even if a space $\mathcal{X}$ is compactly embedded in another space  $\mathcal{Y}$ it is not usually the case that
$L^{2}(\Omega; \mathcal{X})$ is compactly embedded in $L^{2}(\Omega; \mathcal{Y})$.   As such, the
standard Aubin or Arzel\`a-Ascoli type compactness results, which classically make possible
the passage to the limit in the nonlinear terms, can not be directly applied in this stochastic setting.
With this in mind, we first establish the existence of martingale solutions following the approach 
in e.g. \cite{ZabczykDaPrato1} and see also \cite{FlandoliGatarek1, DebusscheGlattHoltzTemam1}.
Here the main mathematical tools are the Prokhorov theorem, which is used to obtain compactness in the collection of probability measures
associated to the approximate solutions, and the Skorohod embedding theorem, which provides almost sure convergences,
but relative to a new underlying stochastic basis.

At this stage there is another difficulty in comparison to previous works, e.g. \cite{FlandoliGatarek1}, 
which requires us to consider martingale solutions which are {\em very smooth} in $x \in \DD$, i.e. 
which evolve starting from data in 
$H^\mm(\DD)$, with $\mm$ sufficiently large (in particular we may take $m'=m+5$). The reason for this initially non-sharp range for $\mm$ stems 
from the following complication already alluded to above: the a priori estimates hold only up to a stopping time, so that when
we attempt to find uniform estimates the bounds hold only up to a sequence of times $\tau_{n}$,
which may depend on the order $n$ of  the approximation.  In contrast to the deterministic case, it is
not clear how to bound $\tau_{n}$ from below, uniformly in $n$.   
To compensate for this difficulty, we add a smooth cut-off function depending on the
size of $\|u\|_{W^{1,\infty}}$  in front of the nonlinear and noise terms in the Galerkin scheme.
This cut-off function however introduces additional obstacles for inferring uniqueness, 
which in view of the Yamada-Watanabe theorem is crucial for later arguments
that allow us to pass to the case of pathwise solutions. 
For uniqueness, estimates in the $L^2(\DD)$ norm give rise to terms involving the $W^{1,\infty}(\DD)$ norm, 
which prevents one from closing the estimates in the energy space.
On the other hand, if we attempt to prove uniqueness by estimating the difference of solutions
in the $H^\mm$ norm for arbitrary $\mm > d/2 +1$, we encounter problems due to terms which involve an
excessive number of derivatives.  By momentarily restricting ourselves to sufficiently large values of $\mm$, we 
manage to overcome both difficulties.

Having passed to the limit in the Galerkin scheme, we obtain the existence of very smooth solutions to a modified Euler equation with
a cut-off in front of the nonlinearity.   We can therefore \emph{a posteriori} introduce a stopping time
and infer the existence of a martingale solution of \eqref{eq:E:1}--\eqref{eq:E:3}. 
It still remains to deduce the existence of pathwise solutions, that
is solutions of \eqref{eq:E:1}--\eqref{eq:E:3} defined
relative to the initially given stochastic basis $\mathcal{S}$.
For this we are guided by the classical Yamada-Watanabe theorem from finite
dimensional stochastic analysis.  This result tells us that, for 
finite dimensional systems at least, pathwise solutions exist 
whenever martingale solutions may be found and pathwise uniqueness
holds (cf.~\cite{WatanabeYamada1971a,WatanabeYamada1971b}).
More recently a different proof of such results was developed in \cite{GyongyKrylov1}
which leans on an elementary characterization of convergence
in probability (cf.  Lemma~\ref{thm:GyongyKry} below).
Such an approach can sometimes be used for stochastic partial differential
equations, see e.g.~\cite{DebusscheGlattHoltzTemam1} in the context of viscous fluids equations.
Notwithstanding previous applications of Lemma~\ref{thm:GyongyKry}
for the stochastic Navier-Stokes and related systems, 
the inviscid case studied here presents some new challenges, most important of which is the difficulty 
in establishing the {uniqueness} of pathwise solutions.

With a class of pathwise solutions in very smooth spaces in hand, we next apply
a density-stability argument to obtain the existence of solutions evolving in $W^{m,p}(\DD)$
where the ranges for $m$, $p$ are now sharp, i.e. $m > d/p +1$ for any $p \geq 2$.  
Since, for all $m'$ sufficiently large,
$H^{m'}(\DD)$ is densely embedded in $W^{m,p}(\DD)$, 
we may smoothen (mollify) the initial data to obtain 
a sequence of very smooth pathwise approximating solutions $u^n$
which evolve in $H^{m'}(\DD)$.  By estimating these solutions pairwise we are able to show that
they form a Cauchy sequence in $W^{m,p}(\DD)$, up to a strictly positive stopping time.  Since almost sure
control is needed for the individual solutions which each have their own maximal
time of existence, we may use of an abstract lemma from \cite{MikuleviciusRozovskii2, GlattHoltzZiane2}.  
See also \cite{GlattHoltzTemam2} for an application to other SPDEs, 
and \cite{GlattHoltzTemam1} for related results in the deterministic setting.

As above for the uniqueness of solutions, when estimating $u^n - u^m$ we encounter terms
involving $\nabla u^n$ in the $W^{m,p}$ norm (which is finite since $u^{n} \in H^{m'}(\DD)$ and $m'$ is large).  
These terms are dealt with using some properties of 
the mollifier $F_\epsilon$  used to smoothen the initial data (here $\epsilon = 1/n$).
More precisely, the term  $\| \nabla u^n\|_{W^{m,p}}$ is of size $1/\epsilon$, but it is multiplied by
$\| u^{n} - u^{m}\|_{W^{m-1,p}}$, which converges to $0$ when $m\geq n$ and $n\rightarrow \infty$,
{\em even when multiplied by $1/\epsilon = n$}. 
See \cite{KatoLai1984,Masmoudi2007} for related estimates for the deterministic Euler equation.

In the second part of the manuscript we turn to establish some global existence results
for \eqref{eq:E:1}--\eqref{eq:E:2}.  We first study the case of additive noise 
in two spatial dimensions.  
To address the additive case we apply a classical Beale-Kato-Majda type inequality for 
$\|u\|_{W^{1,\infty}}$ (see e.g.~\cite{MajdaBertozzi2002}).  This shows that if we can control the vorticity of the solution in $L^\infty$ uniformly in time, then
the nonlinear terms may be bounded like $\log( \|u\|_{W^{m,p}}) \|u\|_{W^{m,p}}^p$.  As such
our proof relies on suitable estimates for the vorticity $\curl u$ in $L^\infty$, which  in this additive case can be achieved 
via a classical change of variables, and 
by establishing a suitable stochastic analogue of a logarithmic Gr\"onwall lemma.

The case of linear multiplicative noise is more interesting.  As noted above, such noise structures
evidence a pathwise damping of the solutions of \eqref{eq:E:1}--\eqref{eq:E:2}, which may be seen by analyzing the transformed
system \eqref{eq:Notes:21}--\eqref{eq:Notes:21b} for a new variable $v(t) = u(t) \exp(- \alpha W_{t})$.    In order to take
advantage of this damping in the three dimensional case, we need to carefully show
that the vortex stretching term is suitably controlled by the damping terms coming from the noise.    For 
a sufficiently large noise coefficient $\alpha$ (or equivalently, for a sufficiently small initial condition) we see that the vorticity
must be decaying, at least for some initial period during which $\|u\|_{W^{1,\infty}}$ remains below
a certain threshold value.  Via the usage of the Beale-Kato-Majda inequality we see in turn 
that the growth on $\|u\|_{W^{m,p}}$ is limited by the possible growth of a certain geometric Brownian motion during this initial period.   
We are therefore able to show that if $\|u_0\|_{W^{m,p}}$ is sufficiently small with respect to a function of $\alpha$
and a given $R> 0$ then, on the event that the geometric Brownian motion never grows to be larger than $R$, the quantity
 $\|u\|_{W^{m,p}}$ will remain below a certain bound. In turn, this guarantees that the quantity $\|u\|_{W^{1,\infty}}$
 will in fact never reach the 
critical value that would prevent the decay in vorticity, and we conclude that the solution is in fact global 
in time on this event that the geometric Brownian motion always stays below the value $R$.
Since we are able to derive probabilistic bounds on this event, which crucially are independent of $\alpha$, we obtain
the desired results.

The manuscript is organized as follows.  In Section~\ref{sec:prelim} we review some mathematical
background, deterministic and stochastic, needed throughout the rest of the work. 
We then make precise the conditions that we need to impose on the noise through $\sigma$
in Section~\ref{sec:NoiseConds}.  We conclude this section with a detailed discussion of some
examples of nonlinear noise structures covered under the given abstract conditions on $\sigma$.
Section~\ref{sec:results} contains the precise definitions of solutions to \eqref{eq:E:1}--\eqref{eq:E:3}, along with statements of our
main results.  We next carry out some a priori estimates in Section~\ref{sec:apriori}. 
In Section~\ref{sec:compact} we introduce the Galerkin scheme and establish the existence of very smooth solutions.  In
Section~\ref{sec:WmpSolutionsCons} we establish the existence of solutions in the optimal
spaces $W^{m,p}$ for any $m > d/p + 1$.  The final two Sections~\ref{sec:2dglobal} and~\ref{sec:GlobalExistenceLinMultNoise} are devoted to proofs of the global existence results for the cases of additive and linear multiplicative
noises respectively.  Appendices gather various additional technical tools used throughout the body of the paper.

\section{Preliminaries}
\label{sec:prelim} \setcounter{equation}{0}

Here we recall some deterministic and stochastic ingredients which will be used throughout this paper.

\subsection{Deterministic Background}
\label{sec:sec:deterministic}
We begin by defining the main function spaces used throughout
the work.  For each integer $m \geq 0$ and $p \geq 2$ we let  
    \begin{align}
    X_{m,p} &= \left\{ v \in (W^{m,p}(\DD))^d \colon \nabla \cdot v = 0, \ v|_{\partial \DD} \cdot n = 0 \right\}
    \label{eq:Xmp}
  \end{align}
  and for simplicity write $X_{m} = X_{m,2}$ (see also \cite{Temam1975}).
These spaces are endowed with the usual Sobolev norm of order $m$
  \begin{align*}
    \| v\|_{W^{m,p}(\DD)}^p :=  \sum_{ | \alpha| \leq m} \| \partial^\alpha v \|_{L^p(\DD)}^p.
  \end{align*}
  As usual, the norm on $X_m$ is denoted by $\Vert \cdot \Vert_{H^m}$. We make the convention 
  to  write $\| \cdot \|_{W^{m,p}}$ and $\|\cdot\|_{H^m}$ instead of $\| \cdot \|_{W^{m,p}(\DD)}$ and $\|\cdot\|_{H^m(\DD)}$, 
{\em unless}  Sobolev spaces on $\partial \DD$ are considered.
  We let $(\cdot,\cdot)$ denote the usual $L^{2}(\DD)$ inner product, which
  makes $X_{0} \subset L^{2}(\DD)$ a Hilbert space. The inner product on
  $X_{m}$ shall be denoted by
  $(\cdot,\cdot)_{H^m} = \sum_{|\alpha|\leq m} (\partial^{\alpha} \cdot, \partial^{\alpha} \cdot)$.

Throughout the analysis we shall make frequent use of certain classical ``calculus inequalities'' which can be
established directly from the Leibniz rule and the Gagliardo-Nirenberg inequalities.
Whenever $m > d/p$ we have the Moser estimate
\begin{align}
	\|uv\|_{W^{m,p}} \leq C ( \|u\|_{L^{\infty}} \|v\|_{W^{m,p}} +
	\|v\|_{L^{\infty}} \|u\|_{W^{m,p}}), 
	\label{eq:CalcInequality}
\end{align}
for all $u,v \in W^{m,p}(\DD)$ and some universal constant $C = C(m, p, \DD) > 0$.
Note that in particular this shows that $W^{m,p}$ is an algebra whenever $m > d/p$.
The following commutator estimate 
will also be used frequently
\begin{align}
\label{eq:commutator}
    \sum_{0\leq  |\alpha | \leq m}
       \| \partial^{\alpha} (u \cdot \nabla v) - u \cdot \nabla \partial^{\alpha} v \|_{L^{p}}
     \leq C \left( \| u \|_{W^{m,p}} \| \nabla v \|_{L^{\infty}} +
                 \| \nabla u \|_{L^{\infty}} \| v \|_{W^{m,p}} \right)
\end{align}
for some constant $C = C(m,p,\DD)>0$, where $ m > 1 + d/p,$ $u \in W^{m, p}$, and $v
\in W^{m+1,p}$.
Note that for what follows we shall assume that $m> 1 + d/p$ and $p \geq 2$, where $d = 2, 3$ is the dimension of $\DD$, 
allowing us to apply \eqref{eq:CalcInequality} and \eqref{eq:commutator}.

In order to treat the pressure term appearing in the Euler equations, 
we will need to bound the solutions of an elliptic Neumann problem taking the form:
\begin{align}
	- \Delta \pi &= f,  \textrm{ in } \DD,
	\label{eq:EllipticNeumann1}\\
	\frac{ \partial \pi} {\partial n}  &= g,  \textrm{ on } \partial \DD,
	\label{eq:EllipticNeumann2}
\end{align}
for given $f$ and $g$, sufficiently smooth.  For this purpose we recall
the result in \cite{AgmonDouglisNirenberg1} which gives the bound:
\begin{align} \label{eq:Agmon}
	\| \nabla \pi \|_{W^{m,p}(\DD)}  \leq
		C ( \| f \|_{W^{m-1,p}(\DD)}
		       + \| g \|_{W^{m-1/p,p}(\partial \DD)} )
\end{align}
where $C= C(m, p, \DD) > 0$ is a universal constant. In fact, \eqref{eq:Agmon} is usually combined with 
the bound given by the trace theorem: $\|  h_{| \partial \DD}\|_{W^{m-1/p,p}(\partial\DD)} \leq C \|h \|_{W^{m,p}(\DD)}$,
which holds for sufficiently smooth $h$, integers $m\geq 1$, and $p\geq 2$ (cf.~\cite{AdamsFournier}).

Also in relation to the pressure we consider $P$, the so-called Leray projector, to be the orthogonal projection in $L^{2}(\DD)$ onto the
closed subspace $X_{0}$.  Equivalently, for any $v \in L^2(\DD)$ we have $Pv = (1 - Q) v$ where 
$$
Qv = - \nabla \pi
$$ 
for any $\pi \in H^{1}(\DD)$ which solves the elliptic Neumann problem
\begin{align}
- \Delta \pi  &= \nabla \cdot v, \qquad \mbox{ in }\ \DD, \label{eq:Leray:1}\\
 \frac{\partial \pi}{\partial n} &= v \cdot n, \qquad \mbox{ on }\ \partial \DD. \label{eq:Leray:2}
\end{align}
Moreover, for $v \in W^{m,p}$, observe that $\nabla \cdot v \in W^{m-1,p}(\DD)$ and $v|_{\partial \DD}\cdot n \in W^{m-1/p,p}(\partial \DD)$.
Hence, by applying \eqref{eq:Agmon} and the trace theorem to \eqref{eq:Leray:1}--\eqref{eq:Leray:2}, we infer that
\begin{align}
\| P v \|_{W^{m,p}(\DD)} \leq C \| v \|_{W^{m,p}(\DD)}
\label{eq:P:bounded}
\end{align}
for any $v \in W^{m,p}(\DD)$. Thus $P$ is also a bounded linear operator from $W^{m,p}(\DD)$ into $X_{m,p}$.

We conclude this section with some bounds on the nonlinear terms which involve the Leray projector.  These bounds will be used throughout the 
rest of the work.
\begin{lemma}[\bf Bounds on the nonlinear term]
\label{lemma:Leray}
Let $m> d/p +1$, and $p\geq 2$. The following hold:
\begin{itemize}
\item[(a)] If $u \in W^{m,p}$ and $v \in W^{m+1,p}$ then $P(u\cdot \nabla v) \in X_{m,p}$, and 
\begin{align}
\| P (u \cdot \nabla v) \|_{W^{m,p}} & \leq C \left( \| u\|_{L^\infty} \| v\|_{W^{m+1,p}} + \| u \|_{W^{m,p}} \|v\|_{W^{1,\infty}} \right).	
\label{eq:P:product}
\end{align}
\item[(b)] If $u,v \in X_{m,p}$, then $Q (u \cdot \nabla v) \in W^{m,p}(\DD)$ and 
\begin{align}
\| Q( u \cdot \nabla v) \|_{W^{m,p}} & \leq  C \left( \| u\|_{W^{1,\infty}} \| v\|_{W^{m,p}} + \| u \|_{W^{m,p}} \|v\|_{W^{1,\infty}} \right).
\label{eq:Q:product}
\end{align}
\item[(c)] If $u \in X_{m,p}$ and $v \in X_{m+1,p}$ then
\begin{align}
& \left| \sum_{|\alpha| \leq m} (\partial^\alpha P( u \cdot \nabla v), \partial^\alpha v |\partial^\alpha v|^{p-2}) \right|
 \leq C \left( \| u\|_{W^{1,\infty}} \|v\|_{W^{m,p}} + 	\| u \|_{W^{m,p}} \| v\|_{W^{1,\infty}} \right) \|v \|_{W^{m,p}}^{p-1}.
\label{eq:P:commutator}
\end{align}
\end{itemize}
In \eqref{eq:P:product}--\eqref{eq:P:commutator}, $C = C(m,p,\DD)$ is positive universal constant.
\end{lemma}
\begin{proof}[Proof of Lemma~\ref{lemma:Leray}]
Firstly we observe that if $u \in W^{m,p}$ and $v \in W^{m+1,p}$ then by \eqref{eq:CalcInequality} we have $ u \cdot \nabla v \in W^{m,p}$ and $\|u \cdot \nabla v\|_{W^{m,p}}$ 
is bounded by the right side of \eqref{eq:P:product}. Thus (a) follows from \eqref{eq:P:bounded}.

The proof of item (b) is due to~\cite{Temam1975}. If $u$ and $v$ are divergence free, and satisfy the 
non-penetrating boundary condition (which occurs when $u, v \in X_{m,p}$) then boundary term $( u\cdot \nabla v) \cdot n$ may be re-written as $u_i v_j \phi_{ij}$, 
for some smooth functions $\phi_{ij}$, independent of $u,v$ which parametrize $\partial \DD$ in a suitable way. 
Also, again due to the divergence free condition, $\nabla \cdot (u \cdot \nabla v)$ may 
be re-written as $\partial_i u_j \partial_j v_i$.  Hence, neither the boundary condition nor the force have too many derivatives and 
the elliptic Neumann problem one has to solve for the function $\pi$ such that $Q (u \cdot \nabla v) = - \nabla \pi$ becomes
\begin{align*}
 -\Delta \pi &= \partial_i u_j \partial_j v_i \\
 \frac{\partial \pi}{\partial n} & = u_i v_j \phi_{ij}.
\end{align*}
The proof of (b) now follows by applying estimate \eqref{eq:Agmon} to the above system, using the trace theorem and finally 
\eqref{eq:CalcInequality}.

Lastly, in order to prove (c) one uses the cancellation property 
$
( u \cdot \nabla v , v |v|^{p-2}) = 0,
$
the definition of $P$, the bound \eqref{eq:commutator}, the H\"older inequality, and item (b) to obtain
\begin{align*}
 \left| \sum_{|\alpha| \leq m} (\partial^\alpha P( u \cdot \nabla v), \partial^\alpha v |\partial^\alpha v|^{p-2}) \right|
&\leq  \sum_{|\alpha| \leq m} \left|  (\partial^\alpha ( u \cdot \nabla v), \partial^\alpha v |\partial^\alpha v|^{p-2}) \right| +   \sum_{|\alpha| \leq m} \left|(\partial^\alpha Q( u \cdot \nabla v), \partial^\alpha v |\partial^\alpha v|^{p-2}) \right| \\
&\leq C \left( \sum_{|\alpha| \leq m} \| \partial^{\alpha}(u\cdot \nabla v) - u \cdot \nabla \partial^\alpha v\|_{L^p}  + \| Q (u\cdot \nabla v) \|_{W^{m,p}} \right) \| v\|_{W^{m,p}}^{p-1}\\
&\leq C \left( \| u\|_{W^{1,\infty}} \|v\|_{W^{m,p}} + 	\| u \|_{W^{m,p}}  \| v\|_{W^{1,\infty}} \right)\| v\|_{W^{m,p}}^{p-1},
\end{align*}
concluding the proof of item (c).
\end{proof}

\subsection{Background on Stochastic Analysis}
\label{sec:StochAnal}
We next briefly recall some aspects of the theory of
the infinite dimensional stochastic analysis which we use below.   We 
refer the reader to  \cite{ZabczykDaPrato1}
for an extended treatment of this subject.
For this purpose we start by fixing a stochastic
basis $\mathcal{S} := (\Omega, \mathcal{F}, \Prb,$ $
\{\mathcal{F}_t\}_{t \geq 0}, \WW)$.  Here  $(\Omega, \mathcal{F}, \Prb)$
is a complete probability space, and $\WW$ is a cylindrical
Brownian motion defined on an auxiliary Hilbert space
$\UU$ which is adapted to a complete, right continuous
filtration $\{\mathcal{F}_t\}_{t \geq 0}$.
By picking a complete orthonormal basis $\{e_k\}_{k \geq 1}$
for $\UU$, $\WW$ may be written as the formal sum
$\WW(t,\omega) = \sum_{k \geq 1} e_k W_k(t, \omega)$ where
the elements $W_k$ are a sequence of independent $1D$ standard
Brownian motions.
Note that $\WW(t,\omega) = \sum_{k \geq 1} e_k W_k(t, \omega)$
does not actually converge on $\UU$ and so we will sometimes
consider a larger space $\mathfrak{U}_0 \supset \mathfrak{U}$
we define according to
\begin{displaymath}
  \mathfrak{U}_0
  := \left\{ v = \sum_{k \geq 0} \alpha_{k} e_{k} :
  \sum_k \frac{\alpha^{2}_{k}}{k^2} < \infty \right\},
\end{displaymath}
and endow this family with the norm
  $\| v \|_{\mathfrak{U}_{0}}^{2} := \sum_{k} \alpha^{2}_{k} k^{-2}$, for any $v = \sum_k \alpha_{k} e_{k}$.
Observe that the embedding of $\mathfrak{U} \subset \mathfrak{U}_0$
is Hilbert-Schmidt. Moreover, using standard martingale arguments
with the fact that each $W_{k}$ is almost surely continuous
 we have that,
  $\WW \in C([0,\infty), \mathfrak{U}_0)$, almost surely.
See \cite{ZabczykDaPrato1}.

Consider now another separable Hilbert space $X$.  We denote
the collection of Hilbert-Schmidt operators, 
the set of all bounded operators $G$ from $\mathfrak{U}$ to $X$ such that
$\|G\|_{L_{2}(\mathfrak{U}, X)}^{2} := \sum_{k} |G e_{k}|^{2}_{X} < \infty$,
by $L_{2}(\mathfrak{U},X)$.
Whenever $X = \RR$, i.e. in the case where $G$ is a linear functional,
we will denote $L_{2}(\mathfrak{U},\RR)$ by simply $L_{2}$.
Given an $X$ valued predictable\footnote{Let
$\Phi = \Omega \times [0,\infty)$ and take
$\mathcal{G}$ to be the $\sigma$-algebra generated by sets of the
form
\begin{displaymath}
    (s,t] \times F, \quad 0 \leq s< t< \infty, F \in \mathcal{F}_s;
    \quad \quad
    \{0\} \times F, \quad F \in \mathcal{F}_0.
\end{displaymath}
Recall that a $X$ valued process $U$ is called predictable (with
respect to the stochastic basis $\mathcal{S}$) if it is measurable
from $(\Phi,\mathcal{G})$ into $(X, \mathcal{B}(X))$,
$\mathcal{B}(X)$ being the family of Borel sets of $X$.}
process
$G \in L^{2}(\Omega; L^{2}_{loc}
([0, \infty),L_{2}(\UU, X)))$ and taking
$G_k = G e_k$  one may define the
(It\={o}) stochastic integral
\begin{equation}\label{eq:StocIntDef}
   M_{t} := \int_{0}^{t} G \dW = \sum_k \int_0^t G_k dW_k,
\end{equation}
as an element in $\mathcal{M}^2_X$, that is the space of all
$X$ valued square integrable martingales.   If 
we merely assume that the predictable process
$G \in L^{2}_{loc} ([0, \infty),L_{2}(\UU, X))$ almost surely,
i.e. without any moment condition,
then $M_{t}$ can still be defined as in \eqref{eq:StocIntDef} by
a suitable localization procedure.  Detailed constructions 
in both cases may be found in e.g.
\cite{ZabczykDaPrato1} or \cite{PrevotRockner}.

The process $\{M_t \}_{t \geq 0}$ has
many desirable properties.  Most notably for the analysis here,
the Burkholder-Davis-Gundy inequality holds which in the present
context takes the form,
\begin{equation}\label{eq:BDG}
\begin{split}
 \E \left(\sup_{t \in [0,T]} \left| \int_0^t G \dW  \right|_X^r \right)
 \leq C \E \left(
   \int_0^T |G|_{L_2(\UU, X)}^2 dt \right)^{r/2},
      \end{split}
\end{equation}
valid for any $r \geq  1$, and where $C$ is an absolute constant depending only on $r$.
In the coordinate basis $\{e_{k}\}$, \eqref{eq:BDG} takes the form
\begin{align*}
     \E \left(\sup_{t \in [0,T]} \left|  \sum_k \int_0^t G_k dW_k  \right|_X^r \right)
 \leq C \E \left(
   \int_0^T \sum_k |G_k|_{X}^2 dt \right)^{r/2}.
\end{align*}

Since we consider solutions of \eqref{eq:E:1}-- \eqref{eq:E:3} evolving in $X_{m,p}$ for
any $p \geq 2$ and $m > d/p +1$, we will recall some details of the construction
of stochastic integrals evolving on $W^{m,p}(\DD)$.  Here we use the approach of~\cite{Krylov1999, MikuleviciusRozovskii2001}, to which
we refer the reader for further details.  See
also \cite{Neihardt1978, Brzezniak1995} and containing references
for a different, more abstract approach to stochastic integration in the Banach space setting. 
Suppose that $p \geq 2$, $m \geq 0$, define
\begin{align*}
  \mathbb{W}^{m,p} =
  \left\{ \sigma: \DD \rightarrow L_{2}:
  	\sigma_{k}(\cdot) = \sigma(\cdot) e_{k} \in W^{m,p}
	\textrm{ and }
	\sum_{|\alpha| \leq m} \int_{\DD} | \partial^{\alpha} \sigma|_{L_{2}}^{p} dx < \infty
	\right\},
\end{align*}
which is a Banach space according to the norm
\begin{align}
   \|\sigma\|_{\mathbb{W}^{m,p}}^{p}
:= \sum_{|\alpha| \leq m} \int_{\DD} | \partial^{\alpha} \sigma|_{L_{2}}^{p} dx
= \sum_{|\alpha| \leq m} \int_{\DD} \left( \sum_{k \geq 1} |\partial^{\alpha} \sigma_{k}|^{2} \right)^{p/2} dx.
\label{eq:LpHSnormAna}
\end{align}
Let $P$ be the Leray projection operator defined in Section~\ref{sec:sec:deterministic}.  
For $\sigma \in \mathbb{W}^{m,p}$ we define $P\sigma$ as an element
in $\mathbb{W}^{m,p}$ by taking $(P\sigma) e_k = P (\sigma e_k)$ so that 
$P$ is a linear continuous operator on $\mathbb{W}^{m,p}$. We take
$$\XX_{m,p} = P \mathbb{W}^{m,p} = \{ P\sigma: \sigma \in \mathbb{W}^{m,p}\}.$$
Note that $\XX_{m,2} = L_2(\mathfrak{U}, X_m)$ and in accordance with \eqref{eq:Xmp},
we will denote $\XX_{m,2}$ by simply $\XX_m$.

Consider any predictable process $G \in
L^p(\Omega; L^p_{loc}([0,\infty), \mathbb{X}_{m,p})$.   For such a $G$
we have, for any $T > 0$ and almost every $x \in \DD$, that
$
  \E \int_0^T \sum_{|\alpha| \leq m} | \partial^{\alpha} G(x) |_{L_{2}}^2 dt < \infty.
$
We thus obtain from the Hilbert space theory introduced above that
$M_{t}$ as in \eqref{eq:StocIntDef} is well defined for almost every $x \in \DD$ 
as a real valued martingale 
and that for each $|\alpha| \leq m$,
$
 \partial^{\alpha} M_t(x) = \int_0^t \partial^{\alpha}G(x) \dW
$.
  By applying the Burkholder-Davis-Gundy
inequality, \eqref{eq:BDG} we have that
\begin{displaymath}
\begin{split}
  \E \sup_{t \in [0,T]} \| M_{t}\|^{p}_{W^{m,p}} 
 \leq C\sum_{|\alpha| \leq m}  \int_{\DD} \E \left( \int_0^T   | \partial^{\alpha}G(x)|_{L_{2}}^2 dt \right)^{p/2}  dx
    \leq C \E \int_0^T |G|_{\mathbb{X}_{m,p}}^p dt.
\end{split}    
\end{displaymath}
Lastly, cf.~\cite{Krylov1999,MikuleviciusRozovskii2001} one may show that $M_t \in L^p(\Omega;C([0,\infty); X_{m,p}))$
and is an $X_{m,p}$ valued martingale.

\section{Nonlinear multiplicative noise structures and examples}
\label{sec:NoiseConds}
\setcounter{equation}{0}
 
In this section we make precise the conditions that we impose on the noise.  While, in abstract form, these conditions 
appear to be rather involved, they fact cover a very wide class of physically
realistic nonlinear stochastic regimes.  We conclude this section by detailing some of these 
examples.

\subsection{Abstract conditions}
We next describe, in abstract terms, the conditions imposed for $\sigma$. 
Consider any pair of Banach spaces $\mathcal{X}$, $\mathcal{Y}$
with $\mathcal{X} \subset L^\infty(\DD)$. We denote the space of locally bounded maps
\begin{align*}
  {\rm Bnd_{u,loc}}(\mathcal{X}, \mathcal{Y}) := 
  \Big\{ 
   \sigma \in C(\mathcal{X} \times [0, \infty); \mathcal{Y}) \colon
   \| \sigma(x,t) \|_{\mathcal{Y}} \leq \locCon(\|x\|_{L^\infty}) (1 + \|x\|_{\mathcal{X}}), \forall x     \in \mathcal{X}, t \geq 0 
   \Big\}
\end{align*}
where $\locCon(\cdot) \geq 1$ is an increasing function which is locally bounded and is independent of $t$.  In  addition we define the space of locally Lipschitz functions,
\begin{align*}
{\rm Lip_{u,loc}}(\mathcal{X}, \mathcal{Y}) =   \Big\{
   \sigma \in {\rm Bnd_{u,loc}}(\mathcal{X},\mathcal{Y}) \colon
   \| \sigma(x,t) - \sigma(y, t)\|_{\mathcal{Y}} \leq 
   \locCon(\|x\|_{L^\infty} + \|y\|_{L^\infty}) \|x - y \|_{\mathcal{X}},
   \forall x, y \in \mathcal{X}, t \geq 0
   \Big\}.
\end{align*}
Note that in in both cases the subscript {\rm u} is intended to emphasize
the that increasing function $\locCon$ appearing in the above inequalities
may be taken to be independent of $t \in [0,\infty)$.  Note furthermore 
that, by considering such locally Lipschitz spaces of functions, 
we are able to cover stochastic forcing involving Nemytskii operators,
i.e. smooth functions of the solutions multiplied by spatially correlated white in time Gaussian noise (see Section~\ref{sec:NoiseExamples} below).

For the main local existence results in the work, Theorem~\ref{thm:LocalExistence} below, we fix $p\geq 2$ and an integer $m > d/p + 1$, and suppose that
\begin{align}
  \sigma \in  {\rm Lip_{u,loc}}(L^{p}, \mathbb{W}^{0,p}) \cap {\rm Lip_{u,loc}}(W^{m-1,p}, \mathbb{W}^{m-1,p}) \cap {\rm Lip_{u,loc}}(W^{m,p}, \mathbb{W}^{m,p})
  \label{eq:sigMappings}.
\end{align}
Since $P$ is a continuous linear operator on $\mathbb{W}^{k,p}$, for $k \geq 0$ it follows 
that
$P\sigma \in {\rm Lip_{u,loc}}(W^{k,p}, \mathbb{X}_{k,p})$, for $k=  m-1,m$.
Observe that by \eqref{eq:sigMappings} we have that $\int_{0}^{t} P \sigma(u)\dW \in C([0,\infty);X_{m,p})$
for each predictable process $u \in C([0,\infty);X_{m,p})$.

We will also impose some additional technical conditions on $\sigma$ which are required for the proof of local existence of solutions (cf. Theorem~\ref{thm:LocalExistence}
below).  These conditions do no preclude any of the examples we give below.     
Firstly we suppose that
\begin{align}
\sigma \in {\rm Bnd_{u,loc}}(W^{m+1,p}, \WB^{m+1,p}).
\label{eq:ExtraDerivativeAssumptforDensityStability}
\end{align}
Fix some $\mm$ sufficiently large, such that $H^{\mm-2} \subset W^{m+1,p}$, e.g. $\mm> m + 3 + d (p-2)/(2p)$ by the Sobolev embedding. For simplicity we take an $\mm$ which works for all $p\geq 2$, and for the rest of this paper fix 
$$
\mm = m+5.
$$
We assume that
\begin{align}
 \sigma \in {\rm Bnd_{u,loc}}(H^{\mm}, \WB^{\mm,2}).
 \label{eq:BadNoiseAssumption} 
\end{align}
Condition \eqref{eq:ExtraDerivativeAssumptforDensityStability} is used for the 
density and stability arguments in Section~\ref{sec:WmpSolutionsCons}, while
condition \eqref{eq:BadNoiseAssumption} seems necessary in order to justify the construction of solutions to the Galerkin system (cf.~Section~\ref{sec:UniformGalerkinEstimates} below).

In the case of an additive noise when we assume that 
$\sigma$ is independent of $u$ (cf. Theorem~\ref{thm:GlobalExistence2DAdditive}), we may alternatively assume that:
\begin{align}
\sigma \in L^{p}(\Omega, L^{p}_{loc}([0,\infty);\WB^{m+1,p}))
\label{eq:fuckedAdditiveNoise}
\end{align}
and that $\sigma$ is predictable. Note that while \eqref{eq:sigMappings}--\eqref{eq:BadNoiseAssumption} covers many additive noise structures,
\eqref{eq:fuckedAdditiveNoise} is less restrictive and allows for $\omega \in \Omega$ dependence in $\sigma$.

\subsection{Examples} 
\label{sec:NoiseExamples}

We now describe some examples of stochastic forcing structures
for $\sigma(u) \dW$ covered under the conditions \eqref{eq:sigMappings}
--\eqref{eq:BadNoiseAssumption} imposed above, or alternatively \eqref{eq:fuckedAdditiveNoise} for additive noise.

\subsection*{Nemytskii operators}

One important example is stochastic forcing of a smooth function of the 
solution.  Suppose that $g: \RR^{d} \rightarrow \RR^{d}$ is $C^{\infty}$ smooth and consider
$\alpha \in \WB^{\mm,2}$, where as above $\mm = m + 5$.  We then take
\begin{align}
  \sigma_{k}(u) = \alpha_{k}(x) g(u), \quad k \geq 1.
  \label{eq:NemytskiiSetup}
\end{align}
In this case we have that:
$$
\sigma(u)\dW = \sum_{k \geq 1} \alpha_{k}(x) g(u)dW_{k} = g(u)  \sum_{k\geq 1} \alpha_{k}(x) dW_{k}
	               = g(u) \alpha \dW.
$$
Note that $\alpha\dW$ is formally a Gaussian process with the spatial-temporal correlation structure
\begin{align*}
	\E (\alpha \dW(x, t) \alpha \dW(y,s)) = K(x,y) \delta_{t -s}
	\quad \textrm{ for all } x, y \in \RR^d, t, s \geq 0,
\end{align*}
with $K(x,y) = \sum_{k\geq 1} \alpha_k(x) \alpha_k(y)$.
Observe that if $g(u) \in W^{n,q}$ for $q \geq 2$ and $n \geq d/q$ then
$$
	\|\sigma(u) - \sigma(v)\|_{\WB^{n,q}}^{q} := \sum_{|\alpha|\leq n} \int_{\DD} \left( \sum_{k \geq 1} | \partial^{\alpha} (\alpha_{k}  g(u)- \alpha_k g(v))|^{2}\right)^{q/2}dx	
	       \leq C \| \alpha \|_{\WB^{n,q}}^q \|g(u) - g(v)\|_{W^{n,q}}^{q}.
$$
We may therefore show that \eqref{eq:NemytskiiSetup} satisfies 
\eqref{eq:sigMappings}--\eqref{eq:BadNoiseAssumption} by making use of
the following general fact about the composition of functions.  
\begin{lemma}[\bf Locally Lipschitz and bounded] \label{lemma:compose}
Fix any $n > d/p$ with $p \geq 2$. 
Suppose that $g: \RR^{d} \rightarrow \RR^{d}$ and that $g \in W^{n+1,\infty}(\RR^{d})$.
Then 
\begin{align}
	\|g(u)- g(v)\|_{W^{n,p}(\DD)} \leq \locCon(\|u\|_{L^\infty} + \|u\|_{L^\infty}) \|u-v\|_{W^{n,p}(\DD)}
	\quad \textrm{ for every } u,v \in W^{n,p}(\DD).
	\label{eq:Lip}
\end{align}
holds for some positive, increasing function  $\locCon(\cdot) \geq 1$.
\end{lemma}
Note that \eqref{eq:sigMappings} follows from \eqref{eq:Lip}. Moreover setting $v=0$ in \eqref{eq:Lip} also proves \eqref{eq:ExtraDerivativeAssumptforDensityStability} and \eqref{eq:BadNoiseAssumption}.
The proof of Lemma~\ref{lemma:compose} is based on Moser-type estimates (similar to \eqref{eq:CalcInequality}), Gagliardo-Nirenberg interpolation inequalities, and the chain rule. See e.g.
\cite[Chapter 13, Section 3]{Taylor2011b} for further details.

\subsection*{Linear multiplicative noise}

One important example covered under this general class of Namytskii operators 
is a linear multiplicative noise.  Here we consider
$$
\sigma(u) \dW = \alpha u dW
$$
where now $\alpha \in \RR$ and $W$ is a $1D$ standard Brownian motion. 
We obtain this special case from the above framework by taking $g = Id$ and $\alpha_1 \equiv 1$, $\alpha_k = 0$ for $k \geq 2$.   
We shall treat such noise structures in detail in Section~\ref{sec:GlobalExistenceLinMultNoise} 
(cf. Theorem~\ref{thm:GlobalExistence2D3DLinearMultiplicative}).

\subsection*{Stochastic forcing of functionals of the solution}

We may also consider functionals (linear and nonlinear) of the solution, forced by 
independent white noise processes. Suppose that, for $k \geq 1$ we are given $f_k: L^p(\DD) \rightarrow \RR$ such that
\begin{align}
	|f_k(u) - f_k(v)| \leq C \|u -v \|_{L^p} \quad \textrm{ for } u, v \in L^p
	\label{eq:FunctionalsCond}
\end{align}
where the constant $C$ is independent of $k$.  We take
\begin{align}
  \sigma_k(u) = f_k(u) \alpha_{k}(x,t) 
  \label{eq:functionalsNoise}
\end{align}
then, for any  $n \geq d/q$ 
$$
\|\sigma(u) - \sigma(v)\|_{\WB^{n,q}}^{q} := \sum_{|\alpha|\leq n} \int_{\DD} \left( \sum_{k \geq 1}  |f_k(u)- f_k(v)|^2| \partial^{\alpha} \alpha_{k}|^{2}\right)^{p/2}dx	
		\leq \| \alpha \|_{\WB^{n,q}}^p \|u - v\|_{L^p}^p.
$$
Thus, under the assumption
\eqref{eq:FunctionalsCond} if we furthermore 
assume that $\sup_{t \geq 0} \| \alpha(t) \|_{\WB^{\mm,2}} < \infty$, 
then $\sigma$ given by \eqref{eq:functionalsNoise} satisfies conditions
\eqref{eq:sigMappings}--\eqref{eq:BadNoiseAssumption}.

\subsection*{Additive Noise}

For $\sigma: [0,\infty) \rightarrow H^\mm$, with $\sup_{t \geq 0} \|\sigma(t)\|_{H^\mm}< \infty$,
we may easily observe that $\sigma$ satisfies \eqref{eq:sigMappings}--\eqref{eq:BadNoiseAssumption}.
For such noise $\sigma \dW$ may be understood in the formal expansion
$$
	\sigma \dW(t,x,\omega) = \sum_k \sigma_k(t,x) dW_k(t,\omega).
$$
Note that our results for additive noise in Theorem~\ref{thm:GlobalExistence2DAdditive}
are established under a more general $\omega$-dependent $\sigma$, which satisfies~\eqref{eq:fuckedAdditiveNoise}.

\section{Main results}
\label{sec:results}
\setcounter{equation}{0}

With the mathematical preliminaries in hand
and having established the noise structures we shall consider, we now make precise the notions of
{\em local}, {\em maximal} and {\em global} solutions of  the stochastic Euler equation \eqref{eq:E:1}--\eqref{eq:E:3}.  
\begin{definition}[{\bf Local Pathwise Solutions}]
\label{def:PathwiseSol}
Suppose that $m> d/p + 1$ with $p \geq 2$ and $d = 2,3$.
Fix a stochastic basis $\mathcal{S} := (\Omega, \mathcal{F}, \Prb,$ $
\{\mathcal{F}_t\}_{t \geq 0}, \WW)$ and $u_{0}$ an $X_{m,p}$ valued
$\mathcal{F}_{0}$ measurable random variable.  Suppose that $\sigma$
satisfies the conditions \eqref{eq:sigMappings}--\eqref{eq:BadNoiseAssumption}
(or alternatively \eqref{eq:fuckedAdditiveNoise}).
\begin{itemize}
\item[(i)] A \emph{local pathwise} $X_{m,p}$ solution of the stochastic
Euler equation is a pair $(u,\tau)$, with $\tau$
a strictly positive stopping time, and $u: [0,\infty) \times
\Omega \rightarrow X_{m,p}$ is a predictable process satisfying
\begin{align*}
 u(\cdot \wedge \tau) \in  
 C([0,\infty), X_{m,p})
\end{align*}
and for every $t \geq 0$,
\begin{equation}\label{eq:EulerLocTmInt}
   u(t \wedge \tau) + \int_{0}^{t \wedge \tau} P (u \cdot \nabla u) dt
   = u(0) + \int_{0}^{t \wedge \tau} P \sigma(u) \dW.
\end{equation}
\item[(ii)]  We say that local pathwise solutions are {\em unique} (or {\em indistinguishable})
	if, given any pair $(u^{(1)}, \tau^{(1)})$, $(u^{(2)}, \tau^{(2)})$ of local pathwise
	solutions,
	\begin{equation}\label{eq:uniquenessCutoffSolns}
    \Prb \left( \indFn{u^{(1)}(0) = u^{(2)}(0) } ( u^{(1)}(t) - u^{(2)}(t))  = 0; \forall t \in [0, \tau^{(1)} \wedge \tau^{(2)}] \right) = 1.
  \end{equation}
\end{itemize}
\end{definition}
Given the existence and uniqueness of such local solutions
we can quantify the possibility of any finite time blow-up.
In some cases we are able to show that such pathwise solution
in fact are global in time.
\begin{definition}[{\bf Maximal and global solutions}]
\label{def:MaxandGlobalSol}
Fix a stochastic basis and assume the conditions $u_{0}$ and $\sigma$
are exactly as in Definition~\eqref{def:PathwiseSol} above.
A \emph{maximal} pathwise solution is a triple $(u,\{\tau_{n}\}_{n \geq 1}, \xi)$
such that each pair $(u, \tau_{n})$ is a local pathwise solution,
$\tau_{n}$ is increasing with $\lim_{n \rightarrow \infty} \tau_{n} = \xi$
and so that
\begin{equation} \label{eq:W1inftyBlowUp}
  \sup_{t \in [0, \tau_{n}]} \|u(t) \|_{W^{1,\infty}} \geq n
\textrm{ on the set } \{ \xi < \infty \}.
\end{equation}
A maximal pathwise solution  $(u,\{\tau_{n}\}_{n \geq 1}, \xi)$ is said to be \emph{global}
if $\xi = \infty$ almost surely.\footnote{Under this definition it is clear that, for every $T > 0$,
$\sup_{t \in [0,T]} \|u(t)\|_{W^{1,\infty}}$ is almost surely finite on the set $\{\xi = \infty\}$.}
\end{definition}

Our primary goal in this work is to study local and global pathwise solutions
of the stochastic Euler equation.  These type of solutions also fall under the designation
of ``strong solutions''; we prefer the term ``pathwise'' since it avoids possible confusion
with classical terminology used in deterministic PDEs. In any case
one can also establish the existence of ``martingale'' (or probabilistically ``weak'' solutions) of
\eqref{eq:E:1}--\eqref{eq:E:3} where the stochastic basis is an unknown
in the problem and the initial conditions are only specified in law.   Indeed
such type of solutions are essentially established as an intermediate step
in the analysis which is carried out in Section~\ref{sec:compact}; see Remark~\ref{rmk:MartingaleSoln} below. 

We now state the main results of this paper.  The first result concerns the 
local existence of solutions, the proof of which
is carried out in two steps, in Sections~\ref{sec:compact} and~\ref{sec:WmpSolutionsCons} below.
\begin{theorem}[\bf Local existence of pathwise solutions] \label{thm:LocalExistence}
Fix a stochastic basis $\mathcal{S} :=$ $(\Omega, \mathcal{F},$ $\Prb, \{\mathcal{F}_t\}_{t \geq 0},$ $\WW)$.
Suppose that $m> d/p + 1$ with $p \geq 2$ and $d = 2,3$.
Assume that $u_{0}$ is an $X_{m,p}$ valued, 
$\mathcal{F}_{0}$ measurable random variable, and that $\sigma$
satisfies the conditions \eqref{eq:sigMappings}--\eqref{eq:BadNoiseAssumption}.
Then there exists a unique maximal pathwise solution $(u,\{\tau_{n}\}_{n \geq 1}, \xi)$
of \eqref{eq:E:1}--\eqref{eq:E:3}, in the sense of Definitions~\ref{def:PathwiseSol} and 
\ref{def:MaxandGlobalSol}.
\end{theorem}

In Section~\ref{sec:2dglobal} we show that in
two space dimensions we have, in the case of an {\em additive noise}, the
global existence of solutions.   Note that in contrast to the 
situation for the $2D$ Navier-Stokes equations (cf. e.g. \cite{GlattHoltzZiane2}),
proving the global existence for a general Lipschitz nonlinear multiplicative noise
seems to be out of reach with current methods (see Remark~\ref{rem:2D:noMultiplicative} below for
further details).

\begin{theorem}[\bf Global existence for additive noise in 2D]\label{thm:GlobalExistence2DAdditive}
Fix $m> 2/p + 1$ with $p \geq 2$, a stochastic basis $\mathcal{S} := (\Omega, \mathcal{F}, \Prb,$ $\{\mathcal{F}_t\}_{t \geq 0}, \WW)$, and assume that $u_{0}$ is an $X_{m,p}$ valued,
$\mathcal{F}_{0}$ measurable random variable. Assume that $\sigma$ does not depend on $u$ and 
\eqref{eq:fuckedAdditiveNoise} (or \eqref{eq:sigMappings}--\eqref{eq:BadNoiseAssumption}) holds. Then, there exits a unique global pathwise solution 
of \eqref{eq:E:1}--\eqref{eq:E:3}, i.e. $\xi = \infty$ almost surely.
\end{theorem}
\begin{remark}\label{rmk:AddNoiseIsFucked}
The local existence of pathwise solutions with additive noise follows directly from Theorem~\ref{thm:LocalExistence} 
in the case of a (deterministic) continuous $\sigma: [0,\infty) \to \WB^{\mm,2}$,
with $ \sup_{t \geq 0}\|\sigma(t)\|_{\mathbb{X}_{\mm,2}} < \infty$, where $\mm$ is as in \eqref{eq:BadNoiseAssumption}.
On the other hand, the proof of local existence for additive noise does not require the involved machinery employed to 
deal with a general nonlinear multiplicative noise;
in this case one can transform \eqref{eq:E:1} into a random partial differential equation, which can be treated pathwise, 
using the classical (deterministic) methods for the Euler equations (cf.~\cite{MajdaBertozzi2002}).   Of course, one has to show
that this random transformed system is measurable with respect to the stochastic elements in the problem but this
may be achieve with continuity and stability arguments.  These technicalities are essentially contained in 
\cite{Kim1}, to which we refer for further details.
\end{remark}

Finally we address the case of a {\em linear multiplicative noise}. In $2D$ we show that the pathwise  solutions are global in time.
In $3D$ we go further and prove that the noise is
regularizing at the pathwise level.  Here we are essentially 
able to establish that the time of existence converges to $+\infty$ a.s. in the {\em large noise limit}. 
More precisely, we have:
\begin{theorem}[\bf Global existence for linear multiplicative noise]\label{thm:GlobalExistence2D3DLinearMultiplicative}
Fix $\mathcal{S} := (\Omega, \mathcal{F}, \Prb, \{\mathcal{F}_t\}_{t \geq 0}, \WW)$\footnote{For the noise
structure considered here, we need only to have defined a single $1D$ standard Brownian motion.}, a stochastic basis.
Suppose that $m> d/p + 1$ with $p \geq 2$ and $d = 2,3$, and
assume that $u_{0}$ is an $X_{m,p}$ valued,
$\mathcal{F}_{0}$ measurable random variable.  For $\alpha \in \RR$ we consider
\eqref{eq:E:1}--\eqref{eq:E:3} with a linear multiplicative noise
\begin{align*}
	\sigma_{k}(u) = \sigma(u) e_{k} =
	\begin{cases}
	  \alpha u  & \textrm{ if } k =1,\\
	  0 & \textrm{ otherwise.}\\
	\end{cases}
\end{align*}
\begin{itemize}
\item[(i)] Suppose  $d = 2$.  Then for any $\alpha \in \RR$ the maximal pathwise solution of guaranteed by
	Theorem~\ref{thm:LocalExistence} is in fact global, i.e. $\xi = \infty$ almost surely.
\item[(ii)] Suppose $d =3$.  Let $R \geq 1$ and $\alpha \not = 0$ be arbitrary parameters. 
Then there exists a positive deterministic function $\kappa(R,\alpha)$
which satisfies 
\begin{align*} 
\lim_{\alpha^2 \to \infty} \kappa(R,\alpha)  = \infty\\
\end{align*}
for every fixed $R\geq 1$, 
such that whenever
\begin{align}
\Vert u_{0} \Vert_{W^{m,p}} \leq \kappa(R,\alpha), \quad a.s.
\label{eq:LinNonBlowupRelation}
\end{align}
then
\begin{align*}
  \Prb \left(  \xi = \infty \right) \geq 1 - \frac{1}{R^{1/4}}.
\end{align*}
In particular, for every $\epsilon > 0$ and any given deterministic initial condition, the probability that solutions corresponding to sufficiently large $|\alpha|$ {\em never blow up}, is greater than $1-\epsilon$.
\end{itemize}	
\end{theorem}

\begin{remark}[\bf Lack of global well-posedness in two dimensions with generic multiplicative noise] \label{rem:2D:noMultiplicative}
We emphasize that even in the two-dimensional setting, and even for $\DD = \RR^{2}$, the global existence of smooth solutions to \eqref{eq:E:1}--\eqref{eq:E:3} 
for a {\em general} Lipschitz multiplicative noise appears to be out of reach. 
In fact, the analogous result remains open even in the deterministic setting
unless the forcing is linear. Indeed, let us consider the Euler equations with a {\em solution-dependent forcing}
\begin{align}
&\partial_t u + u \cdot \nabla u + \nabla \pi = f(u),\  \nabla \cdot u = 0 \label{eq:forcedEuler}
\end{align}
where $f$ is a smooth function mapping $\RR^{2} \to \RR^{2}$, which decays sufficiently fast at infinity.  In order to obtain the global in time regularity 
of \eqref{eq:forcedEuler} one must have an a priori global in time bound for the supremum of the vorticity $w = \nabla^{\perp} \cdot u$ 
(or at least a bound in a Besov space ``sufficiently close'' to $L^{\infty}$). However, using the Biot-Savart law, the evolution of $w$ is governed by
\begin{align}
\partial_{t} w + u \cdot \nabla w = - \partial_{1} f_{1}(u) w - (\partial_{1} f_{2}(u) + \partial_{2} f_{1}(u)) \RSZ_{12} w+ (\partial_{2} f_{2}(u) - \partial_{1} f_{1}(u)) \RSZ_{11} w \label{eq:forcedEulerVorticity}
\end{align}
where $\RSZ_{ij}$ are the Riesz transforms $\partial_{i} \partial_{j} (-\Delta)^{-1}$, and $f(u) = (f_1(u), f_2(u))$. 
While the first term on the  right side of of \eqref{eq:forcedEulerVorticity} 
is harmless for $L^{\infty}$ estimates on $w$, {\em unless} $f$ is such that $\partial_{1}  f_{2} + \partial_{2} f_{1}  = \partial_{2} f_{2}  - \partial_{1} f_{1} = 0$ 
identically (which is true for $f (u) = u$, that is $f_1(x,y) = x$ and $f_2(x,y) = y$), the remaining two terms prevent one from obtaining a bound on $\Vert w \Vert_{L^{\infty}}$ using classical methods, since Calder\'on-Zygmund 
operators are not bounded on $L^{\infty}$. Recently it was proven in~\cite{ConstantinVicol2011} that if one adds an {\em arbitrary} amount of dissipation, in the form of a positive power of $-\Delta$, or even dissipation as mild as $\log(1 -\Delta)$, 
to the left side of \eqref{eq:forcedEuler}, then the equations have global in time smooth solutions. The global well-posedness of \eqref{eq:forcedEuler} with no dissipation remains open for generic 
smooth forcing $f$.
\end{remark}

\section{A priori estimates}
\label{sec:apriori}\setcounter{equation}{0}

In this section we carry out a priori estimates for
solutions evolving in $X_{m,p}$ of \eqref{eq:E:1}--\eqref{eq:E:3} with $m > d/p + 1$, $p \geq 2$.  
The bounds established
in this section will be used extensively throughout the rest of the work.
We begin with the bounds in the Hilbert space case, namely
for solutions in $X_{m}$. These estimates will be used in Section~\ref{sec:compact} in the context of a Galerkin scheme.

\subsection{\texorpdfstring{$L^{2}$}{L 2}-based estimates}
\label{sec:aprioriL2}
We start with estimates in $H^{m}(\DD)$, where
$m > d/2 + 1$.   Let  $u$ be a solution of \eqref{eq:E:1}--\eqref{eq:E:2}, which lies in $H^{m+1}(\DD)$ and is defined up
to a (possibly infinite) maximal stopping time of existence $\xi >0$. Note however, that the a priori estimates \eqref{eq:J1}--\eqref{eq:p=2:apriori:final} involve only the $H^{m}$ norm of the solution $u$. 

Let $\alpha\in{\mathbb N}^d$ be a multi-index with $|\alpha| \leq m$.
Applying the Leray projector $P$ and then $\partial^\alpha$ to \eqref{eq:E:1} we obtain
\begin{align}
  d(\partial^\alpha u)
  + \partial^\alpha P (u \cdot \nabla u) dt
   = \partial^\alpha P \sigma(u) \dW. \label{eq:DalphaDerU}
\end{align}
By the It\={o} lemma we find
\begin{align}
  d \norm{\partial^\alpha u}{L^2}^2
       =& -2 \left(
      \partial^\alpha u ,
      \partial^\alpha P (u \cdot\nabla u)
       \right) dt
      + \|\partial^\alpha
      P \sigma (u)\|^2_{\XX_0}
       dt
     + 2 \left(
       \partial^\alpha u  ,
                       \partial^\alpha P \sigma(u)
       \right) \dW \notag\\
     =& (J_1^\alpha  + J_3^\alpha) dt + J_3^\alpha \dW . \label{eq:Ito:ODE}
\end{align}
Fix  $T > 0$ and any stopping time $\tau
\leq \xi \wedge T$. We find that for every $s \in [0,\tau]$,
\begin{align*}
 \norm{ \partial^{\alpha} u(s) }{L^2}^2 \leq
 \norm{ \partial^{\alpha} u_0 }{L^2}^2 +
 \int_{0}^{s}
    (|J_1^\alpha| + |J_2^\alpha|)ds' +
  \left|   \int_{0}^{s} J_3^\alpha \dW \right|.
\end{align*}
Hence, summing over all $|\alpha| \leq m$, taking
a supremum over $s \in [0, \tau]$ and then
taking the expected value we get
\begin{align}
    \E \sup_{s \in [0, \tau]} \norm{  u(s) }{H^{m}}^2
    \leq&
    \E \norm{ u_0 }{H^{m}}^2 +
    \E\sum_{|\alpha| \leq m} \int_{0}^{\tau}
       (|J_1^\alpha| + |J_2^\alpha|)dt'
     + \sum_{|\alpha| \leq m}  \E \left(
        \sup_{s \in [0, \tau]}
       \left|   \int_{0}^{s} J_3^\alpha \dW \right| \label{eq:IntIneq1}
       \right).
\end{align}

We first treat the drift terms $J_1^\alpha$ and $J_2^\alpha$
which may be estimated pointwise in time.  We bound the nonlinear term
$J_1^\alpha$ by setting $p=2$ and $v=u$ in \eqref{eq:P:commutator} 
to obtain
\begin{align}\label{eq:J1}
\sum_{ |\alpha|\leq m} |J_1^\alpha| \leq C \| u \|_{W^{1,\infty}} \| u \|_{H^{m}}^2
\end{align}
for some positive constant $C=C(m,\DD)$.
In view of the assumption \eqref{eq:sigMappings} the $J_2^\alpha$ term
is direct:
\begin{align}
  \sum_{|\alpha|\leq m} |J_2^\alpha| \leq \locCon(\|u\|_{L^\infty})^2(1 + \| u \|_{H^{m}}^2). \label{eq:J2}
\end{align}
where $\locCon$ is the increasing function given in \eqref{eq:sigMappings}.

We handle the stochastic term, involving $J_3^\alpha$, using the Burkholder-Davis-Gundy
inequality \eqref{eq:BDG} and assumption \eqref{eq:sigMappings}:
\begin{align*}
  \E \left(
  \sup_{s \in [0, \tau]}
  \left| \int_{0}^{s}  J_3^\alpha \dW \right|
\right)
       \leq C
       \E\left(
  \int_{0}^{\tau}
    \|\partial^\alpha u\|_{L^2}^2
    \|P\sigma(u) \|^2_{\WB^{m,2}}
  dt
    \right)^{1/2}
           \leq C
      \E\left(
\int_{0}^{\tau}
  \|\partial^\alpha u\|_{L^2}^2
  \locCon(\|u\|_{L^\infty})^2 (1 + \|  u \|_{H^{m}}^2) dt
  \right)^{1/2}.
\end{align*}
Now, summing
over $|\alpha| \leq m$, we infer
\begin{align}\label{eq:BDGest1}
  \sum_{|\alpha| \leq m}
  \E \left(
     \sup_{s \in [0, \tau]}
      \left| \int_{0}^{s}  J_3^\alpha \dW \right|
     \right)
     \leq \frac{1}{2} \E \sup_{s \in [0, \tau]} \| u\|_{H^m}^2
 	 +      C\E\int_{0}^{\tau}
	  \locCon(\|u\|_{L^\infty})^2 (1 + \|  u \|_{H^{m}}^2) dt.
\end{align}
In view of the estimate \eqref{eq:J1} for the nonlinear term,
we now define the stopping time
\begin{align}
	\xi_{R} = \inf
	\left\{t \geq 0 \colon \|u (t)\|_{W^{1,\infty}} \geq R \right\}.
	\label{eq:apriori:stopping:time}
\end{align}
Combining the estimates
\eqref{eq:J1}--\eqref{eq:BDGest1},
we find that for any $t > 0$, by taking $\tau = t \wedge \xi_R$,
\begin{align*}
    \E \sup_{s \in [0, \xi_R \wedge t]}
    \|u\|^2_{H^m}
    \leq&
    2\E \|u_0\|^2_{H^m}
    + C \E \int_0^{\xi_R \wedge t}
    (\| u\|_{W^{1,\infty}} + \locCon(\| u\|_{L^\infty})^2)(1+ \|u\|^2_{H^m} ) ds\\
    \leq& 2\E \|u_0\|^2_{H^m}
    + C \int_0^t \biggl(1 +
     \E \sup_{r \in [0, \xi_R \wedge s]}
    \|u\|^2_{H^m} \biggr) ds,
\end{align*}
where the final constant $C$ depends on $R$ through
$R + \locCon(R)^2$. From the classical Gr\"onwall inequality we infer
\begin{align}
    \E \sup_{s \in [0, \xi_R \wedge T]} \|u\|^2_{H^m}
    \leq C( 1+ \E \|u_0\|_{H^m}^2)
    \label{eq:p=2:apriori:final}
\end{align}
where $C = C(m,d, \DD, T, R, \locCon)$.

Of course estimate \eqref{eq:p=2:apriori:final} does not prevent $\|u\|_{W^{1,\infty}}$ from blowing up before $T$;
the bound  \eqref{eq:p=2:apriori:final} grows exponentially in $R$ and hence we do not a priori know that $\xi_R \to \infty$ as $R \to \infty$. 
Note also that, in contrast to the case of the full space (or in the periodic setting), when $\DD$ is a smooth simply-connected bounded domain,
the non-blow-up of solutions is controlled by $\| u \|_{W^{1,\infty}}$, rather than the classical $\| \nabla u \|_{L^{\infty}}$.   
This is due to the nonlocal nature of the pressure. In the bound \eqref{eq:p=2:apriori:final} this is inherently expressed through the definition of the stopping time $\xi_{R}$. 
Of course, the $L^\infty$ bound on $u$ is also needed to control the terms involving $\sigma$.

\subsection{\texorpdfstring{$L^{p}$}{L p}-based estimates, \texorpdfstring{$p>2$}{p>2}.}
\label{sec:LpEst}

We now return to \eqref{eq:DalphaDerU} again for any $\alpha$, $|\alpha| \leq m$.
We apply the It\={o}
formula, pointwise in $x$, for the function $\phi(v) = |v|^p = (|v|^2)^{p/2}$.
After integrating in $x$ and using the stochastic Fubini theorem
(see \cite{ZabczykDaPrato1}) we obtain:
\begin{align}
  d \|\partial^{\alpha}u\|^p_{L^p} =& -
   p \int_{\DD}
    \partial^\alpha u  \cdot
      \partial^\alpha P(u \cdot\nabla u) |\partial^{\alpha} u|^{p -2}dx dt 
 	\notag \\
      &+ \sum_{k \geq 1}\int_{\DD}   \left( \frac{p}{2}
       |\partial^\alpha P\sigma_k (u)|^2 |\partial^{\alpha} u|^{p -2}
       + \frac{p(p-2)}{2} (\partial^\alpha u  \cdot  \partial^\alpha P\sigma_k(u))^{2} |\partial^{\alpha} u|^{p -4}
               \right)
       \, dx  dt
       \notag\\
     &+ p \sum_{k \geq 1} \left(
       \int_{\DD}   \partial^\alpha u  \cdot
                       \partial^\alpha P \sigma_k(u) |\partial^{\alpha} u|^{p -2}\,dx
       \right) dW_k 
       \notag \\
       :=& I_1^{\alpha} dt +  I_2^{\alpha} dt + I_3^{\alpha}\dW.
       \label{eq:LpItoFormalSoln}
\end{align}
By letting $v=u$ in \eqref{eq:P:commutator} we
bound
\begin{align}
  |I_1^{\alpha}| 
        \leq C \| u\|_{W^{1,\infty}} \|u\|^{p}_{W^{m,p}}. \label{eq:J1EstLp}
\end{align}
We turn now to estimate the terms specific to the stochastic case.  For
$I_{2}^{\alpha}$, using \eqref{eq:sigMappings} we have
\begin{align}
  |I_{2}^{\alpha}| \leq& C \int_{\DD} \sum_{k \geq 1} |\partial^\alpha P \sigma_k (u)|^2 |\partial^{\alpha} u|^{p -2} dx
             \leq 
                       C \|P\sigma(u)\|_{\WB^{m,p}}^2
                       \|u\|_{W^{m,p}}^{p-2}
             \leq C \locCon(  \|u\|_{L^\infty})^2 (1 +
                       \|u\|_{W^{m,p}}^{p}).
                       \label{eq:ItoCorTemEstsLp}
\end{align}
To estimate the stochastic integral terms involving $I_{3}$, we
apply the Burkholder-Davis-Gundy inequality, \eqref{eq:BDG}, the Minkowski
inequality for integrals, and use
\eqref{eq:sigMappings}.  We obtain, for any stopping time $\tau \leq T \wedge \xi$,
\begin{align}
  \E \left( \sup_{s \in [0,\tau]} \left|\int_{0}^{s}I_{3}^{\alpha}\dW \right| \right)
  	\leq&  C\E \left( \int_{0}^{\tau} \sum_{k \geq 1} \left(
       \int_{\DD}   \partial^\alpha u  \cdot
                       \partial^\alpha P \sigma_k(u) |\partial^{\alpha} u|^{p -2}\,dx
       \right)^{2}ds
	\right)^{1/2} \notag \\
	\leq&  C\E \left( \int_{0}^{\tau} \left(
       \int_{\DD}  \left( \sum_{k \geq 1}
                       |\partial^\alpha P \sigma_k(u)|^2
                       |\partial^{\alpha} u|^{2(p -1)}
                       \right)^{1/2} \,dx
       \right)^{2}ds
	\right)^{1/2} \notag \\
	\leq&C
   \E \left( \int_{0}^{\tau}  \|\partial^{\alpha} u\|^{2(p-1)}_{L^p}
       \left(\int_{\DD} \left( \sum_{k \geq 1}
                       |\partial^\alpha P \sigma_k(u)|^2
                       \right)^{p/2} \,dx\right)^{2/p}
       ds
	\right)^{1/2}\notag \\
	\leq& C\E \left( \sup_{s \in [0,\tau]}\|\partial^{\alpha} u\|^{p/2}_{L^p}
	\left( \int_{0}^{\tau}  \|u\|^{p-2}_{W^{m,p}}
	\locCon(  \|u\|_{L^{\infty}})^2 (1 + \|u\|_{W^{m,p}}^2)
       ds
	\right)^{1/2}
	\right) \notag \\
	\leq& \frac{1}{2} \E \sup_{s \in [0,\tau]}\|\partial^{\alpha} u\|^{p}_{L^p} + C
	\E  \int_{0}^{\tau}
	\locCon(  \|u\|_{L^{\infty}})^2(1 + \|u\|_{W^{m,p}}^p)
       ds.
       \label{eq:MFJonesConsultsForBDG}
\end{align}
Combining the $L^p$ It\=o formula \eqref{eq:LpItoFormalSoln} with the estimates \eqref{eq:J1EstLp}--\eqref{eq:MFJonesConsultsForBDG}, 
and making use of the stopping time $\xi_R$ defined in \eqref{eq:apriori:stopping:time}, we may obtain, as in the Hilbert case,
\begin{align}
    \E \sup_{s \in [0, \xi_R \wedge T]} \|u\|^p_{W^{m,p}}
    \leq C( 1+ \E \|u_0\|_{W^{m,p}}^p)
    \label{eq:pgeq2:apriori:final}
\end{align}
where $C = C(m,d, \DD, T, R, \locCon)$.

\begin{remark}[\bf From a priori estimates to the construction of solutions]
Having completed the a priori estimates in $W^{m,p}$, we observe that,
 even for the deterministic Euler equations on a bounded domain, the construction of solutions 
 is non-trivial and requires a delicate treatment of the coupled elliptic/degenerate-hyperbolic system (see e.g.~\cite{KatoLai1984,Temam1975}).
 In addition, the stochastic nature of the equations introduces a number of additional difficulties, such as the the lack of compactness in the $\omega$ variable. 
 We overcome these difficulties in Sections~\ref{sec:compact} and \ref{sec:WmpSolutionsCons} below, by first constructing a sequence of very smooth approximate 
 solutions evolving from mollified initial data, and then passing to a limit using a Cauchy-type argument.
\end{remark}

\section{Compactness methods and the existence of very smooth solutions}
\label{sec:compact} \setcounter{equation}{0}

Leet $p\geq 2$ and $m> d/p +1$ be as in the statement of Theorem~\ref{thm:LocalExistence}.
In this section we establish the existence of ``very smooth'' solutions of
\eqref{eq:E:1}--\eqref{eq:E:3}, that is solutions in $H^{\mm}$, where $\mm = m+5$ (so that $\mm > m+ 3 + d(p-2)/(2p)$ for any $d=2,3$ and $p \geq 2$). We fix this $\mm$ throughout the rest of the paper.
In particular we shall use that $H^{\mm-2} \subset W^{m+1,p}$ and $\mm>d/2+3$.
Note that the the initial data in the statement of our main theorem only lies in $W^{m,p}$, not necessarily in $H^{\mm}$,
but we will apply the results in this section to a sequence of mollified initial data (cf.~\eqref{eq:JBONA:DATA} below), 
and then use a limiting argument in order to obtain the local existence of pathwise solutions for all data in $W^{m,p}$ (see Section~\ref{sec:WmpSolutionsCons}).

We begin by introducing a Galerkin scheme with cut-offs in front 
of both the nonlinear drift and diffusion terms.  
Crucially, these cut-offs allow us to obtain
uniform estimates in the Galerkin approximations 
globally in time (see Remark~\ref{rmk:CutOffReasons} below).
We then exhibit the relevant uniform estimates for
these systems which partially follow from the a priori 
estimates in Section~\ref{sec:apriori}.
We next turn to establish compactness with a variation
on the Arzela-Ascoli theorem, tightness 
arguments, and the Skorohod embedding theorem.
In this manner we initial infer the existence of martingale solutions to
a cutoff stochastic Euler system (cf. \eqref{eq:LimSys:EulerCutOff} below)
in a very smooth spaces. We finally turn to prove the existence of pathwise 
solutions by establishing
the uniqueness for this cutoff system and applying the Gy\"ongy-Krylov convergence
criteria as recalled in Lemma~\ref{thm:GyongyKry} below.

\subsection{Finite Dimensional Spaces and The Galerkin scheme}
\label{sec:GalerkinScheme}

For each $u \in X_{0}$, by the
Lax-Milgram theorem, there exists a unique $\Phi(u) \in X_{\mm}$ solving the variational problem
\begin{align*}
(\Phi(u),v)_{H^{\mm}} = (u,v), \qquad \mbox{for all} \qquad v \in X_{\mm}.
\end{align*}
Actually, the regularity of $\Phi(u)$ is expected to be better. In \cite{Ghidaglia1984} it is shown
that in fact the maximal regularity
$
\Phi(u) \in X_{2 \mm}
$
holds.
We let $\{\phi_{k}\}_{k=1}^{\infty}$ be the complete orthonormal system (in $X_{0}$) of eigenfunctions
for the linear map $u \mapsto \Phi(u)$, which is compact, injective and self-adjoint on $X_{0}$. Therefore,
$(\phi_{k},v)_{H^\mm} = \lambda_{k} (\phi_{k},v)$ for all $v \in X_{\mm}$, where $\lambda_{k}^{-1} > 0$ is the
eigenvalue associated to $\phi_{k}$, and by \cite{Ghidaglia1984} we know  $\phi_{k}$ lies in $X_{2\mm}$ for all $k \geq 1$.

For all $n\geq 1$, we consider the orthogonal projection operator $P_{n}$, mapping $X_{0}$ onto ${\rm span}\{\phi_{1},\dots,\phi_{n}\}$,  given explicitly by
\begin{align*}
P_{n} v = \sum\limits_{j=1}^{n} (v,\phi_{j}) \phi_{j}, \quad \textrm{ for all } v \in X_{0}.
\end{align*}
Note that these $P_n$ are also uniformly bounded in $n$ on $X_\mm$, $X_{\mm-1}$, etc.  See e.g. \cite{LionsMagenes1} for further details.

Fix $R >0$ to be determined, choose a $C^{\infty}$-smooth non-increasing function $\theta_{R}: [0,\infty) \mapsto [0,1]$ such that
\begin{displaymath}
\theta_{R}(x) =
\begin{cases}
 1 \quad &\textrm{ for } |x| < R, \\
 0 \quad &\textrm{ for } |x|> 2R.
 \end{cases}
\end{displaymath}
We consider the following Galerkin approximation scheme for  \eqref{eq:E:1}
\begin{align}
	&d u^{n} + \theta_{R} (\Vert u^{n} \Vert_{W^{1, \infty}})  P_{n} P(u^{n} \cdot \nabla u^{n}) dt
	= \theta_{R} (\Vert u^{n} \Vert_{W^{1,\infty}}) P_{n}P \sigma(u^{n}) \dW, \label{eq:Galerkin}\\
	&u^{n}(0) = P_{n} u_{0}.\label{eq:Galerkin:IC}
\end{align}
The system \eqref{eq:Galerkin}--\eqref{eq:Galerkin:IC} may be
considered as an SDE in $n$ dimensions, with locally Lipschitz drift (cf. Proposition~\ref{thm:Uniqueness} below) and
globally Lipschitz diffusion (cf.~\eqref{eq:sigMappings}).  Since we also have the additional
cancelation property $(P_n P(u \cdot \nabla u), u)_{L^{2}}= 0$ for all $u \in P_nX_\mm$ we may infer
that there exists a unique global in time
solution $u^{n}$ to \eqref{eq:Galerkin}--\eqref{eq:Galerkin:IC},
evolving continuously on $P_{n}X_{\mm}$.  See e.g. \cite{Flandoli1} for further details.

\begin{remark}\label{rmk:CutOffReasons}
The cutoff functions in \eqref{eq:Galerkin} allow us to obtain uniform estimates for $u^{n}$ in $L^{\infty}([0,T],X_{\mm})$
for \emph{any fixed, deterministic} $T> 0$.  Without this cutoff function we are only able
to obtain uniform estimates up to a sequence of
stopping times $\tau^{n}$, depending on $n$.
In contrast to the deterministic case it is unclear if, for example, $\inf_{n \geq 1} \tau^{n} >0$ almost surely.
Note however that the presence of this cut-off causes additional difficulties in the passage to the 
limit of martingale solutions, see Remark~\ref{rmk:MartingaleSoln},
and in order to establish the uniqueness of solutions associated to
the related to the limit cut-off system, see \eqref{eq:LimSys:EulerCutOff}, Proposition~\ref{thm:Uniqueness} and 
Remark~\ref{Rmk:MiscOnUniqueness} below.  
\end{remark}

\subsection{Uniform Estimates}
\label{sec:UniformGalerkinEstimates}

Applying the It\=o formula to \eqref{eq:Galerkin}, and using that $P_{n}$ is self-adjoint on $X_{\mm}$, similarly to \eqref{eq:Ito:ODE} we obtain
\begin{align*}
  d \norm{u^{n}}{H^\mm}^2
       =& -2 \theta_{R} (\Vert u^{n} \Vert_{W^{1, \infty}}) \left( u^{n}, P(u^{n} \cdot\nabla u^{n}) \right)_{H^{\mm}} dt\\
      &+ \theta_{R} (\| u^{n} \|_{W^{1,\infty}})^2\Vert P_{n}P\sigma (u^{n})\Vert^2_{\XX_\mm} dt
     + 2 \theta_{R} (\| u^{n} \|_{W^{1,\infty}}) \left(  u^{n}, P\sigma(u^{n}) \right)_{H^{\mm}} \dW. 
\end{align*}
Further on, in order to establish the needed compactness in the probability 
distributions associated to $u^{n}$, we need uniform estimates on 
higher moments of $\|u^{n}\|_{H^{\mm}}^{2}$.
For this purpose we fix any $r \geq 2$ and compute $d (\norm{u^{n}}{H^\mm}^2)^{r/2}$
from the It\={o} formula and the evolution of $\|u^{n}\|_{H^{\mm}}^{2}$. We find
\begin{align}
   d \norm{u^{n}}{H^\mm}^r
       =& -r \theta_{R} (\Vert u^{n} \Vert_{W^{1, \infty}})\norm{u^{n}}{H^\mm}^{r-2}  \left( u^{n}, P(u^{n} \cdot\nabla u^{n}) \right)_{H^{\mm}} dt \notag\\
       &+ \theta_{R} (\| u^{n} \|_{W^{1, \infty}})^2 \left(
        \frac{r}{2}  \norm{u^{n}}{H^\mm}^{r-2}  \Vert P_{n}P\sigma (u^{n})\Vert^2_{ \XX_\mm} 
       + \frac{r(r -2)}{2} \|u^{n}\|_{H^{\mm}}^{r-4}\left (  u^{n}, P\sigma(u^{n}) \right)_{H^{\mm}}^2  \right)dt\notag \\
     &+ r \theta_{R} (\| u^{n} \|_{W^{1, \infty}}) \norm{u^{n}}{H^\mm}^{r-2} \left(  u^{n}, P\sigma(u^{n}) \right)_{H^{\mm}} \dW. \label{eq:Ito:cutoffLP}
\end{align}

Using bounds similar to the a priori estimates of Section~\ref{sec:apriori}, we obtain the estimate
\begin{align*}
   \E \left( \sup_{s \in [0 , t]}  \norm{u^{n}}{H^{\mm}}^r \right)
       \leq&    \E \norm{P_{n}u_{0}}{H^{\mm}}^r +C \E \int_{0}^{t}
        \theta_{R} (\| u^{n} \|_{W^{1, \infty}}) \left(\locCon(\| u^{n} \|_{L^{\infty}})^2  +  \norm{u^{n}}{W^{1,\infty}} \right)
       \left(1+\norm{u^{n}(s)}{H^{\mm}}^r \right) ds\\
       &\qquad \qquad \quad+C \E \left( \int_{0}^{t} \theta_{R} (\| u^{n} \|_{W^{1, \infty}})^2 \locCon(\| u^{n} \|_{L^{\infty}})^2 \|u^{n}\|_{H^{\mm}}^{r} (1+\|u^{n}\|_{H^\mm}^{r})
       ds \right)^{1/2} \\
       \leq&    \E \norm{u_{0}}{H^{\mm}}^r +C \int_{0}^{t}  1+\E \left( \sup_{s' \in[0,s]}\norm{u^{n}(r)}{H^{\mm}}^r\right) ds +
       \frac{1}{2}\E \left( \sup_{s\in[0,t]}\norm{u^{n}(r)}{H^{\mm}}^r\right),
 \end{align*}
 where $C$ is a constant independent of $n$, depending on $\mathcal{D}$, $\mm$, $r$, and $R$ (through
 $\theta_{R}$ and $\locCon$).

 Therefore, rearranging and applying the standard Gr\"onwall inequality, we obtain that, for any $T>0$
 \begin{align}
  \sup_{n \geq 1}\E  \sup_{s \in [0 , T]}  \norm{u^{n}}{H^{\mm}}^r \leq C, \label{eq:Galerkin:HmBound}
 \end{align}
 for some positive finite constant $C = C (T,R,r, \locCon, \E\norm{u_{0}}{H^{\mm}}^{r})$.

In order to obtain the compactness needed to pass to the limit in 
$u^{n}$ we also would like to have uniform estimates on the time
derivatives of $u^{n}$.  Since in the stochastic case we do not expect $u^{n}$ to be differentiable
in time, we have to content ourselves instead with estimates on fractional time derivatives of order strictly
less than $1/2$.   In order to carry out such estimates we
shall also make use of a variation on the Burkholder-Davis-Gundy inequality \eqref{eq:BDG},
as derived in \cite{FlandoliGatarek1}.

For this purpose, let us briefly recall a particular characterization
of the Sobolev spaces $W^{\alpha, q}([0,T], X)$ 
where $X$ may be any separable Hilbert space.  See, for example, \cite{ZabczykDaPrato1}
for further details.
For $q> 1$ and $\alpha \in (0,1)$ we define
\begin{align*}
 W^{\alpha, q} ([0,T]; X) :=
 \left\{ v \in L^q([0,T]; X);
    \int_0^T \int_0^T
    \frac{\| v(t') - v(t'') \|_{X}^q}{|t'-t''|^{1 + \alpha q}}dt' dt''
        < \infty   \right\},
\end{align*}
which is endowed with the norm
\begin{displaymath}
  \| v \|_{W^{\alpha,p}([0,T]; X)}^q :=
  \int_0 ^T \| v(t') \|^q_X dt' +
  \int_0^T \int_0^T
  \frac{\|v(t') - v(t'')\|_{X}^q}{|t'-t''|^{1 + \alpha q}}dt' dt''.
\end{displaymath}
Note that for $\alpha \in (0,1)$, $W^{1,q}([0,T]; X) \subset W^{\alpha, q} ([0,T]; X)$ with
$\| v \|_{W^{\alpha,q}([0,T]; X)} \leq C \| v \|_{W^{1,q}([0,T]; X)}$.
As in \cite{FlandoliGatarek1} one can show from \eqref{eq:BDG} that for any $q \geq 2$ and any
$\alpha \in [0,1/2)$
\begin{equation}\label{eq:BDGfrac}
  \E \left(
     \left\| \int_0^{t} G dW \right\|_{W^{\alpha, q}([0,T];X)}^q \right)
     \leq C \E \left(
    \int_0^T \|G\|_{L_2(\mathfrak{U}, X)}^q dt \right),
\end{equation}
over all $X$ valued predictable $G \in L^{q}(\Omega; L^{q}_{loc}([0,
\infty),L_{2}(\mathfrak{U}, X)))$ and where $C = C(\alpha, q, T)$.

With these definitions and \eqref{eq:BDGfrac} in hand we return
to \eqref{eq:Ito:cutoffLP}.  For any $0 < \alpha < 1/2$, we have
\begin{align}
  \E  \norm{u^{n}}{W^{\alpha,r}([0,T],H^{\mm-1})}^{r} 
  \leq& C\, \E
  	\left\| P_{n}u_{0} + \int_{0}^{t}\theta_{R} (\Vert u^{n} \Vert_{W^{1, \infty}})  P_{n} P(u^{n} \cdot \nabla u^{n}) ds \right\|_{W^{1,r}([0,T], H^{\mm-1})}^{r}\notag\\
	&+ C\, \E\left\| \int_{0}^{t} \theta_{R} (\| u^{n} \|_{W^{1, \infty}}) P_{n} P \sigma(u^{n}) \dW \right\|_{W^{\alpha,r}([0,T],  H^{\mm-1})}^{r} 
	\label{eq:timederivative:1}
\end{align} for some positive constant $C=C(T)$, independent of $n$.
Since $P_{n}P$ is uniformly bounded in $X_{\mm-1}$ independently of $n$, 
using \eqref{eq:CalcInequality} and \eqref{eq:Galerkin:HmBound}
we bound the first term on the right hand side of \eqref{eq:timederivative:1} as
\begin{align}
 \E&
  	\left\| P_{n}u_{0} + \int_{0}^{t}\theta_{R} (\Vert u^{n} \Vert_{W^{1, \infty}})  P_{n} P(u^{n} \cdot \nabla u^{n}) ds \right\|_{W^{1,r}([0,T], H^{\mm-1})}^{r}
	\notag \\
	&\leq
  C\, \E  \|u_{0}\|^r_{H^\mm} +   C\, \E \int_{0}^{T} \theta_{R} (\Vert u^{n} \Vert_{W^{1, \infty}}) \left\| u^{n} \cdot \nabla u^{n}  \right\|_{H^{\mm-1}}^{r}dt \notag\\
   &\leq C\, \E  \|u_{0}\|^r_{H^\mm}  +  C\, \E \int_{0}^{T}\theta_{R} (\Vert u^{n} \Vert_{W^{1, \infty}})  \|  u^{n}  \|_{W^{1, \infty}}^{r} \|  u^{n}  \|_{H^{\mm}}^{r}dt
   \leq C\, \E \left( \sup_{t\in[0,T]} \Vert u^{n}(t) \Vert_{H^{\mm}}^{r}\right)  \leq C
   \label{eq:DetPartBnd}
\end{align}
where the final constant $C =  C (T,R,r,\E\norm{u_{0}}{H^{\mm}}^{r})$ does not depend on $n$.
For the second term on the left hand side of
\eqref{eq:timederivative:1} we make use of \eqref{eq:BDGfrac} with $q = r$ and $\alpha \in (0,1/2)$, then
\eqref{eq:sigMappings} and \eqref{eq:Galerkin:HmBound} to estimate
\begin{align}
   \E\left\| \int_{0}^{t} \theta_{R} (\| u^{n} \|_{W^{1,\infty}}) P_{n} P\sigma(u^{n}) \dW \right\|_{W^{\alpha,r}([0,T],  H^{\mm-1} ) }^{r}
   &\leq C\, \E \left( \int_{0}^{T} \theta_{R} (\| u^{n} \|_{W^{1,\infty}})^r  \| P_{n} P \sigma(u^{n})\|^r_{   \XX_{\mm-1} }dt\right) \notag\\
   &\leq C\, \E \int_{0}^{T} \theta_{R} (\| u^{n} \|_{W^{1,\infty}})^r \locCon(\| u^{n} \|_{L^{\infty}})^r(1+ \| u^{n}\|^{r}_{H^{\mm}})dt \notag\\
   &\leq C \, \E \left( 1 + \sup_{t\in[0,T]} \Vert u^{n}(t) \Vert_{H^{\mm}}^{r}\right) \leq C,
   \label{eq:UnifromStochIntLPBnd}
\end{align}
where in the final constant $C = C (T,R,r, \locCon,\E\norm{u_{0}}{H^{\mm}}^{r})$ 
is a sufficiently large constant independent on $n$.  Combining \eqref{eq:timederivative:1}--\eqref{eq:UnifromStochIntLPBnd} we have now shown that
\begin{align}
  \sup_{n \geq 1}\E \norm{u^{n}}{W^{\alpha,r}([0,T], H^{\mm-1})}^{r} \leq C,
  \label{eq:Finalq=2FracDerivBnd}
\end{align}
for some positive finite constant $C = C (T,R,r,\E\norm{u_{0}}{H^{\mm}}^{r}, \alpha)$. 
In summary, we have proven:
\begin{proposition}
\label{thm:UniformBndsSum}
Fix $m >d/2+1$,  $\mm = m+ 5$, $\alpha \in (0,1/2)$, $r \geq 2$, and suppose that $\sigma$ satisfies conditions 
\eqref{eq:sigMappings}--\eqref{eq:BadNoiseAssumption}.  
Given $u_0 \in L^r(\Omega; X_\mm)$,
$\mathcal{F}_0$ measurable, consider the associated sequence of solutions
$\{u_n\}_{n \geq 1}$ of the Galerkin system \eqref{eq:Galerkin}--\eqref{eq:Galerkin:IC}. Then the sequence
$\{u^n\}_{n \geq 1}$ is uniformly bounded in 
$$
L^r(\Omega; L^\infty([0,T], X_\mm) \cap W^{\alpha, r}(0,T; X_{\mm-1}))
$$  
for any $T > 0$.
Moreover, under the given conditions, we have 
\begin{align}
&\sup_{n\geq 1} \E\left\| \int_{0}^{t} \theta_{R} (\| u^{n} \|_{W^{1,\infty}})P_{n} P\sigma(u^{n}) \dW \right\|_{W^{\alpha,r}([0,T], H^{\mm-1})}^{r}  < \infty
\label{eq:UnifromTdLPBndsPre}\\
 &\sup_{n\geq 1} \E\left\| u^{n}(t) - \int_{0}^{t} \theta_{R} (\| u^{n} \|_{W^{1,\infty}}) P_{n} P\sigma(u^{n})  \dW\right\|_{W^{1,r}([0,T], H^{\mm-1})}^{r} < \infty.
\label{eq:UnifromTdLPBnds}
\end{align}
\end{proposition}

\subsection{Tightness, Compactness and The Existence of Martingale Solutions}
\label{sec:TightnessMgSol}

For a given initial distribution $\mu_0$ on $X_{\mm}$ we fix a stochastic basis
$\mathcal{S} = (\Omega, \mathcal{F},$ $\{\mathcal{F}_t\}_{t \geq
  0}, \mathbb{P}, \WW)$ upon which is defined an $\mathcal{F}_0$ measurable
random element $u_0$ with distribution $\mu_0$.  As described above,
we define  the sequence
of Galerkin approximations $\{u^n\}_{n\geq 1}$ solving \eqref{eq:Galerkin}--\eqref{eq:Galerkin:IC}
relative to this basis and initial condition.

To define a sequence of measures associated with $\{(u^n, \WW)\}_{n\geq 1}$ we consider the phase
space:
\begin{align}
 \mathcal{X} = \mathcal{X}_S \times \mathcal{X}_W,  \quad \textrm{ where } \quad \mathcal{X}_S =  C([0,T], X_{\mm-2}), \quad
  \mathcal{X}_W = C([0,T],\mathfrak{U}_0).
  \label{eq:StrongConvergenceSpace}
\end{align}
We may think of the first component, $\mathcal{X}_S \supset C([0,T], X_{\mm})$,  as
the space where the $u^n$ lives, and the second component, $\mathcal{X}_W$,
as being the space on which the driving Brownian motions are defined. On $\mathcal{X}$ we
define the probability measures
\begin{align}
   \mu^{n} = \mu^{n}_{S} \times \mu_{W}, \quad \textrm{ where } \quad \mu^{n}_{S}(\cdot) = \Prb( u^{n} \in \cdot), \quad \mu_{W}(\cdot) = \Prb( \WW \in \cdot).
   \label{eq:solutionnoiseMeasureDef}
\end{align}

We next show that the  collection $\{\mu^{n}\}_{n\geq 1}$ is in fact \emph{weakly compact}.
Let $Pr(\mathcal{X})$ be the collection of Borel probability measures on $\mathcal{X}$.
Recall that a sequence $\{ \nu_{n} \}_{n \geq 0} \subset Pr(\mathcal{X})$ is said to
\emph{converge weakly} to an element $\nu \in Pr(\mathcal{X})$ if
 $\int f d \nu_n \rightarrow \int f d \nu$
for all continuous bounded $f$ on $\mathcal{X}$.  As such, we say that a set
$\Lambda \subset Pr(\mathcal{X})$ is weakly compact if every
sequence $\{ \nu_n \} \subset \Lambda$ possesses a
weakly convergent subsequence.    On the other hand we say
that a collection $\Lambda \subset Pr(\mathcal{X})$ is
\emph{tight} if, for every $\epsilon > 0$, there exists a compact
set $K_\epsilon \subset \mathcal{X}$ such that,
  $\mu ( K_\epsilon) \geq 1 - \epsilon  \textrm{ for all } \mu \in \Lambda$.
The classical result of Prohorov (see e.g. \cite{ZabczykDaPrato1})
asserts that weak compactness and tightness are in fact equivalent conditions for
collections $\Lambda \subset Pr(\mathcal{X})$.  We have:

\begin{lemma}[\bf Tightness of Measures for the Galerkin Scheme]\label{thm:Tightness}
  Let $m >d/2+1$,  $\mm = m+ 5$, $r > 2$, 
  assume that $\sigma$ satisfies conditions \eqref{eq:sigMappings}--\eqref{eq:BadNoiseAssumption}, and consider any $\mu_0 \in Pr(X_\mm)$
  with $\int_{X_\mm} |u|^r d\mu_0(u) < \infty$.  
 Fix any stochastic basis
$\mathcal{S} = (\Omega, \mathcal{F},$ $\{\mathcal{F}_t\}_{t \geq
  0}, \mathbb{P}, \WW)$ upon which is defined an $\mathcal{F}_0$ measurable
random element $u_0$ with this distribution $\mu_0$ and take
 $\{u^n\}_{n \geq 1}$ to be the sequence solving 
 \eqref{eq:Galerkin}, \eqref{eq:Galerkin:IC} relative to this basis and initial condition.
  Define the sequence $\{\mu^n\}_{n \geq 1}$ according to \eqref{eq:solutionnoiseMeasureDef}
  using the sequence $\{u^n\}_{n \geq 1}$.
  Then $\{ \mu^{n} \}_{n \geq 1} \subset Pr(\mathcal{X})$ is tight and hence weakly compact.
\end{lemma}
In order to obtain the compact sets used to show that the sequence $\{ \mu^{n}\}_{n \geq 1}$ is tight
we use the following variation on the classical Arzela-Ascoli compactness
theorem from \cite{FlandoliGatarek1}.
\begin{lemma}\label{thm:FlanGatComp}
  Suppose that $Y^{(0)} \supset Y$ are Banach spaces with $Y$
    compactly embedded in $Y^{(0)}$.  Let
    $\alpha \in (0,1]$ and $q \in (1,\infty)$ be such that
    $\alpha q > 1$ then
    \begin{equation}\label{eq:fracSobEmbeddingContFns}
      W^{\alpha, q}([0,T]; Y) \subset \subset C([0,T],Y^{(0)})
   \end{equation}
    and the embedding is compact.
\end{lemma}
\noindent With this result in hand we proceed to the proof of Lemma~\ref{thm:Tightness}:
\begin{proof}[Proof of Lemma~\ref{thm:Tightness}]
Fix any $\alpha \in (0, 1/2)$ such that $\alpha r > 1$. 
According to Lemma~\ref{thm:FlanGatComp} we have
that both $W^{1,2}([0,T]; X_{\mm-1})$, $W^{r, \alpha}([0,T]; X_{\mm})$
are compactly embedded in
$\mathcal{X}_{S}$.  
Therefore, for $s > 0$, the sets 
\begin{align*}
  B^2_s :=
  \left\{
   u \in W^{1,2}([0,T]; X_{\mm-1}) : \|u\|_{W^{1,2}([0,T]; H^{\mm-1})} \leq s
  \right\}
  +
  \left\{
     u \in W^{\alpha,r}([0,T]; X_{\mm}) : \|u\|_{W^{\alpha,r}([0,T]; H^{\mm-1})} \leq s
  \right\}
\end{align*}
are pre-compact in $\mathcal{X}_{S}$.
Since  $\{ u^n \in B^2_s\}$ contains
\begin{align*}
  &\left\{ \left\| u^n(t) - \int_0^t  \locCon(\| u^{n} \|_{W^{m,p}})P_n P\sigma(u^n)  \dW \right\|_{W^{1,2}([0,T]; H^{\mm-1})} \leq s \right\}\\
	&\quad \quad \quad \cap\left\{ \left\| \int_0^t  \locCon(\| u^{n} \|_{W^{m,p}}) P_n P\sigma(u^n) \dW \right\|_{W^{\alpha,r}([0,T]; H^{\mm-1})} \leq s\right\},
\end{align*}
and using Proposition~\ref{thm:UniformBndsSum}, estimates \eqref{eq:UnifromTdLPBndsPre}--\eqref{eq:UnifromTdLPBnds},
and the Chebyshev inequality we bound
\begin{align*}
  \mu_S^n((B_s^2)^C) \leq&
  \Prb \left(
  	\left\| u^n(t) - \int_0^t \theta_R(\| u^{n} \|_{W^{1,\infty}}) P_n P\sigma(u^n)  \dW \right\|_{W^{1,2}([0,T]; H^{\mm-1})} \! \! > s
  \right) \\
  &+
  \Prb \left(
  	\left\| \int_0^t  \theta_R(\| u^{n} \|_{W^{1, \infty}}) P_n P\sigma(u^n) \dW \right\|_{W^{\alpha,r}([0,T]; H^{\mm-1})} \! \! > s
  \right) \leq \frac{C}{s},
\end{align*}
where $C$ is a universal constant independent of $s$ and $n$.
We infer that $\mu^n_S$ is a tight sequence on $\mathcal{X}$.
Now, since the sequence
$\{\mu_W\}$ is constant, it is trivially weakly compact and hence tight.  We
may thus finally infer that the $\{\mu^n\}$ is tight, completing the
proof.
\end{proof}

With this weak compactness in hand we next apply
the Skorokhod embedding theorem (cf.~\cite{ZabczykDaPrato1}) to a weakly
convergent subsequence of $\{\mu^n\}_{n \geq 1}$.  We obtain a new probability
space $(\tilde{\Omega}, \tilde{\mathcal{F}}, \tilde{\Prb})$ on which we have a sequence
of random elements $\{(\tilde{u}^n, \tilde{ \WW}^n)\}_{n \geq 1}$ converging almost surely in $\mathcal{X}$ to an
element $(\tilde{u}, \tilde{\WW})$, i.e.
\begin{align}
	\tilde{u}^n \rightarrow \tilde{u}, \quad  \textrm{ in } C([0,T], X_{\mm-2}) \textrm{ almost surely}
	\label{eq:AlmostSureConvSk1}
\end{align}
and
\begin{align}
	\tilde{\WW}^n \rightarrow \tilde{ \WW }, \quad  \textrm{ in } C([0,T], \mathfrak{U}_0) \textrm{ almost surely}.
	\label{eq:AlmostSureConvSk2}
\end{align}
One may verify as in \cite{Bensoussan1} that
$(\tilde{u}^n, \tilde{\WW}^n)$ satisfies the $n$th order
Galerkin approximation \eqref{eq:Galerkin}--\eqref{eq:Galerkin:IC}
relative to the stochastic basis
$\mathcal{S}^n := (\tilde{\Omega}, \tilde{\mathcal{F}}, \tilde{\Prb}, \{\tilde{\mathcal{F}}^n_t\}, \tilde{ \WW}^n)$
with $\tilde{\mathcal{F}}^n_t$ the completion of the $\sigma$-algebra generated by 
$\{ (u^n(s), \WW^n(s)) : s \leq t\}$.  Using the uniform
bound \eqref{eq:Galerkin:HmBound} and the almost sure convergences 
\eqref{eq:AlmostSureConvSk1}--\eqref{eq:AlmostSureConvSk2}
we may now show that
$(\tilde{u}, \tilde{W})$ solves the the cut-off system
\begin{equation}\label{eq:LimSys:EulerCutOff}
\begin{split}
  &d\tilde{u} + \theta_{R} (\Vert \tilde{u} \Vert_{W^{1, \infty}}) P \left( \tilde{u} \cdot \nabla \tilde{u}  \right) dt = 
   \theta_{R} (\| \tilde{u} \|_{W^{1,\infty}}) P\sigma(\tilde{u})  d\tilde{\WW}.
  \end{split}
\end{equation}
For the technical details of this passage to the limit we refer to e.g.~\cite{DebusscheGlattHoltzTemam1} where this analysis is carried out for the primitive equations. Applying these arguments to the Euler equations introduces no additional difficulties, so we omit further details.
More precisely we have established the following:
\begin{proposition}\label{thm:MGSolCutoff}
Fix any $\mm >d/2 +3 $, $r > 2$, and $R > 0$.  Suppose that $\mu_0 \in Pr(X_\mm)$ is given such that $\int_{X_\mm} \|u\|_{H^\mm}^r d\mu_0(u) < \infty$.
Then there exists a stochastic basis $\mathcal{S} := (\tilde{\Omega}, \tilde{\mathcal{F}}, \tilde{\Prb}, \{\tilde{\mathcal{F}}_t\}, \tilde{ \WW})$
and an $X_\mm$ valued, predictable process 
$$\tilde{u} \in L^2(\Omega; L^\infty_{loc}([0, \infty); X_\mm)) \cap L^2(\Omega; C([0,\infty), X_{\mm -2}))$$
with 
$\tilde{\Prb}(\tilde{u}(0) \in \cdot) = \Prb(u_0 \in \cdot)$ 
such that
$$
  \tilde{u}(t) +  \int_0^t\theta_{R} (\Vert \tilde{u} \Vert_{W^{1, \infty}}) P\left( \tilde{u} \cdot \nabla \tilde{u}\right) dt = \tilde{u}(0) + \int_0^t \theta_{R} (\| \tilde{u} \|_{W^{1,\infty}})  P\sigma(\tilde{u})  d\tilde{\WW}
$$
for every $t \geq 0$.
\end{proposition}

\begin{remark}\label{rmk:MartingaleSoln}
The assumption $\mm> d/2+3$ is needed facilitate the passage
from \eqref{eq:Galerkin} to \eqref{eq:LimSys:EulerCutOff}. Indeed, when passing to the limit we need to 
handle some stray terms arising due to the cut-off terms involving the $W^{1, \infty}$ norm of the
solution. These stray terms are of higher order than the other terms in the estimates, 
and in order to deal with them we need to have compactness in sufficiently regular spaces.  In the analysis above this
compactness is provided by the Arzela-Ascoli type result, Lemma~\ref{thm:FlanGatComp}.
In order to apply this lemma we need estimates on
(fractional) time derivatives of $u^{n}$, which in view of \eqref{eq:DetPartBnd}
must be made in $X_{\mm-1}$.  An additional degree of regularity is then lost
in order to obtain a compact embedding in $X_{\mm-1}$, as required by Lemma~\ref{thm:FlanGatComp},
and we therefore arrive at the condition $\mm > d/2 + 3$.
\end{remark}

We also observe that Proposition~\ref{thm:MGSolCutoff} immediately yield new results
on the existence of martingale solutions of the stochastic Euler equation.
\begin{remark}[\bf Existence of Martingale Solutions]
We may show that the pair $(\tilde{u}, \tilde{\mathcal{S}})$, obtained from Proposition~\ref{thm:MGSolCutoff} is a local martingale solution  of
\eqref{eq:E:1}--\eqref{eq:E:3} by introducing the stopping time
$$
\tau = \inf \{t \geq 0 \colon \|\tilde{u}\|_{W^{1,\infty}} \geq  R\}.
$$   
Of course, unless $\|\tilde{u}(0) \|_{W^{1,\infty}} < R$, i.e. unless
$\mu_0 ( \{u_0 \in X_{m'}: \|u_0\|_{W^{1,\infty}} < R \}) = 1$,
we have  $\tilde{P} (\tau = 0) > 0$.
Such stopping times $\tau$ will also be used further on to infer the existence 
of solutions in the pathwise case.  Note however that in this case the  
$L^\infty(\Omega)$ condition may be subsequently removed 
 with a cutting argument, cf. \eqref{eq:cutThatBitchUpSol}--\eqref{eq:cutThatBitchUpST} below.

\end{remark}

\subsection{Uniqueness, the Gy\"ongy- Krylov lemma, and the existence of strong solutions}
Having now established Proposition~\ref{thm:MGSolCutoff}, and 
guided by the classical Yamada-Wannabe  theorem (see \cite{WatanabeYamada1971a},
\cite{WatanabeYamada1971b}), we would now expect pathwise
solutions to exist once we establish that solutions are ``pathwise unique''.
\begin{proposition}[\bf Pathwise uniqueness] \label{thm:Uniqueness}
  Fix any $r > 2$, $R >0$, and $\mm =m+5$, where $p\geq 2$ and $m>d/p+1$.
  Assume that $\sigma$ satisfies \eqref{eq:sigMappings}--\eqref{eq:BadNoiseAssumption},
  and suppose $(\mathcal{S}, u^{(1)})$ and $(\mathcal{S}, u^{(2)})$ are two global
  solutions of \eqref{eq:LimSys:EulerCutOff} in the sense
  of Proposition~\ref{thm:MGSolCutoff}, relative to the same stochastic basis
  $\mathcal{S} := (\Omega, \mathcal{F}, \{\mathcal{F}_t\}_{t \geq
    0},$ $\mathbb{P}, \WW)$.
    If $u^{(1)}(0) = u^{(2)}(0) = u_{0}$, $a.s$, with $\E \|u_0\|_{H^\mm}^r < \infty$,
    then $u^{(1)}$ and $u^{(2)}$ are \emph{indistinguishable} i.e.
   \begin{equation}\label{eq:uniquenessCutoffSolns1}
    \Prb \left(  u^{(1)}(t)  = u^{(2)}(t);   \forall t
      \geq 0 \right) = 1.
  \end{equation}
\end{proposition}
\begin{proof}[Proof of Proposition~\ref{thm:Uniqueness}]
By the assumption on $u_0$ and Proposition~\ref{thm:MGSolCutoff},
we have for every $T > 0$
\begin{align}
 \E \left( \sup_{t \in [0,T]} (\|u^{(1)}\|^{2}_{H^{\mm}}  + \|u^{(2)}\|^{2}_{H^{\mm}})  \right) \leq C < \infty, \label{eq:TheDude}
\end{align}
where $C$ is a universal constant depending only on $\E \|u_{0}\|^{2}_{H^\mm}$, $R$, $\locCon$, and $T$. 
However, continuity in time is only guaranteed for the $H^{\mm-2}$ norms of $u^{(1)}$ and $u^{(2)}$, 
and so, in view of the choice of $\mm$, we may define the collection of stopping times
$$
\xi^{K} := \inf_{t \geq 0} \left\{ t \colon \|u^{(1)}\|_{W^{m+1,p}}^{2} + \|u^{(2)}\|_{W^{m+1,p}}^{2}> K \right\}, \quad K > 0.
$$
Observe that due to \eqref{eq:TheDude} we have $\xi^{K} \rightarrow \infty$ almost surely as $K \rightarrow \infty$.

Take $v = u^{(1)} - u^{(2)}$.  We have
\begin{align*}
  dv + \theta_{R} (\Vert u^{(1)} \Vert_{W^{1, \infty}}) P \left(u^{(1)} \cdot \nabla u^{(1)} \right) dt -
  &\theta_{R} (\Vert u^{(2)} \Vert_{W^{1, \infty}}) P \left(u^{(2)} \cdot \nabla u^{(2)} \right)dt \\
  &= \left( \theta_{R} (\| u^{(1)} \|_{W^{1,\infty}})P\sigma(u^{(1)})
     -  \theta_{R} (\| u^{(2)} \|_{W^{1,\infty}}) P\sigma(u^{(2)}) \right) \dW.
\end{align*}
We now estimate $v$ in $W^{m,p}$.  For any multi-index $|\alpha| \leq m$ we apply $\partial^{\alpha}$ to the equation
for $v$. With the It\={o} lemma in $L^{p}$ we find 
\begin{align*}
  d \|\partial^{\alpha}v\|^p_{L^p}
  =& -
   p \int_{\DD}
    \partial^\alpha v  \cdot
     \left( \theta_{R} (\Vert u^{(1)} \Vert_{W^{1, \infty}}) \partial^\alpha P(u^{(1)} \cdot \nabla u^{(1)}) -
  \theta_{R} (\Vert u^{(2)} \Vert_{W^{1, \infty}}) \partial^\alpha P(u^{(2)} \cdot \nabla u^{(2)}) \right) |\partial^{\alpha} v|^{p -2}dx dt
      \notag \\
      &+ \sum_{k \geq 1}\int_{\DD}   \biggl( \frac{p}{2}
       |\partial^\alpha P(\theta_{R} (\| u^{(1)} \|_{W^{m,p}})\sigma_{k}(u^{(1)}) -  \theta_{R} (\| u^{(2)} \|_{W^{m,p}})\sigma_{k}(u^{(2)}))|^2 |\partial^{\alpha} v|^{p -2} \notag\\
       &\qquad \quad+ \frac{p(p-2)}{2} (\partial^\alpha v  \cdot P (\theta_{R} (\| u^{(1)} \|_{W^{m,p}})\sigma_{k}(u^{(1)}) -  \theta_{R} (\| u^{(2)} \|_{W^{m,p}})\sigma_{k}(u^{(2)})))^{2} |\partial^{\alpha} v|^{p -4}
               \biggr)
       \, dx  dt \notag \\
     &+ p \sum_{k \geq 1} \left(
       \int_{\DD}   \partial^\alpha v  \cdot
                       \partial^\alpha P(\theta_{R} (\| u^{(1)} \|_{W^{1,\infty}})\sigma_{k}(u^{(1)}) -  \theta_{R} (\| u^{(2)} \|_{W^{1,\infty}})\sigma_{k}(u^{(2)})) |\partial^{\alpha} v|^{p -2}\,dx
       \right) dW_k \notag \\
       :=& (J_1^{\alpha} + J_2^{\alpha}) dt + J_3^{\alpha} {\dW}.
\end{align*}
Using the mean value theorem for $\theta_R$, the embedding $W^{1, \infty} \subset W^{m,p}$, and Lemma~\ref{lemma:Leray} we estimate $J_1^\alpha$ as
\begin{align}\label{eq:J1Est}
    |J_1^\alpha| 
    \leq&C  \left| \theta_{R} ( \Vert u^{(1)} \Vert_{W^{1, \infty}}) - \theta_{R} (\Vert u^{(2)} \Vert_{W^{1, \infty}}) \right|  \;
    	 \left| \left( \partial^\alpha P(u^{(1)} \cdot \nabla u^{(1)}), \partial^\alpha v  | \partial^\alpha v|^{p-2}\right) \right| \notag\\
	 & \qquad \qquad + C  \left| \left(\partial^\alpha P(u^{(1)} \cdot \nabla u^{(1)})
    - \partial^\alpha P (u^{(2)} \cdot \nabla u^{(2)}), \partial^\alpha v | \partial^\alpha v|^{p-2} \right) \right|  \notag \\
      \leq& C
      \left|
      \Vert u^{(1)} \Vert_{W^{1, \infty}} - \Vert u^{(2)} \Vert_{W^{1, \infty}}
      \right|
      \|P(u^{(1)} \cdot \nabla u^{(1)})\|_{W^{m,p}}
      \|v\|_{W^{m,p}}^{p-1} \notag \\
       & \qquad \qquad + C \left| \left(\partial^\alpha P (v \cdot \nabla u^{(1)}), \partial^\alpha v | \partial^\alpha v|^{p-2} \right)  \right| +C
      \left| \left( \partial^\alpha P(u^{(2)} \cdot \nabla v), \partial^\alpha v | \partial^\alpha v|^{p-2} \right) \right|  \notag \\
      \leq& C\|v\|_{W^{m,p}}^{p} \|u^{(1)}\|_{W^{m,p}}\|u^{(1)}\|_{W^{m+1,p}}+
      	        \|v\|_{W^{m,p}}^{p-1} \left(\|v\|_{L^{\infty}} \|u^{(1)}\|_{W^{m+1,p}} + \|u^{(1)}\|_{W^{1,\infty}} \|v\|_{W^{m,p}}\right) \notag \\
	        	&\qquad \qquad +C  \|v\|_{W^{m,p}}^{p-1}
		\left(  \| u^{(2)} \|_{W^{m,p}} \|v \|_{W^{1,\infty}} +
                 \| u^{(2)} \|_{W^{1,\infty}} \| v \|_{W^{m,p}} \right) \notag \\
       \leq& C \|v\|_{W^{m,p}}^{p} \left( (1 +  \|u^{(1)}\|_{W^{m,p}} )\|u^{(1)}\|_{W^{m+1,p}}  + \|u^{(2)}\|_{W^{m,p}}\right) .
\end{align}
Using the local Lipschitz condition on $\sigma$, i.e. \eqref{eq:sigMappings},
we have
\begin{align}
  |J_{2}^\alpha| 
  \leq&  C\|v\|_{W^{m,p}}^{p-2} \| \theta_{R} (\| u^{(1)} \|_{W^{1,\infty}})\sigma(u^{(1)}) -  \theta_{R} (\| u^{(2)} \|_{W^{1,\infty}})\sigma(u^{(2)})\|_{\WB^{m,p}}^2
  \notag \\
  \leq&  C\|v\|_{W^{m,p}}^{p-2} \left(  \theta_{R} (\| u^{(1)} \|_{W^{1,\infty}})^2  \|\sigma(u^{(1)}) -  \sigma(u^{(2)})\|_{\WB^{m,p}}^2 +
                                              \left| (\theta_{R} (\| u^{(1)} \|_{W^{1,\infty}}) -  \theta_{R} (\| u^{(2)} \|_{W^{1,\infty}}) \right|^2 \|\sigma(u^{(2)})\|_{\WB^{m,p}}^2
                                              \right)
  \notag \\
  \leq&  C \locCon(\| u^{(1)} \|_{L^{\infty}} + \| u^{(2)} \|_{L^{\infty}})^2 ( 1 + \|u^{(2)}\|_{W^{m,p}}^2) \|v\|_{W^{m,p}}^{p}                                           
  \label{eq:J3Est}.
\end{align}
For the terms involving $J_{3}^\alpha$ we make use of the Burkholder-Davis-Gundy inequality in a similar way to \eqref{eq:MFJonesConsultsForBDG} 
and then argue as in \eqref{eq:J3Est} in order
to finally estimate, that for every $t \geq 0$ 
\begin{align}
  \E &\sup_{s \in [0,t]} \left| \int_{0}^{s \wedge \xi^{K}} J_{3}^\alpha {\dW}\right| \notag\\
  &\leq C\E \left( \int_{0}^{t \wedge \xi^{K}} \sum_{k \geq 1} \left(
       \int_{\DD}   \partial^\alpha v  \cdot
                       \partial^\alpha P(( \theta_{R} (\| u^{(1)} \|_{W^{1,\infty}})\sigma_k(u^{(1)}) -  \theta_{R} (\| u^{(2)} \|_{W^{1,\infty}})\sigma_k(u^{(2)})) |\partial^{\alpha} v|^{p -2}\,dx
       \right)^{2}ds
	\right)^{1/2} \notag \\
  &\leq C\E \left( \int_{0}^{t \wedge \xi^{K}}  \|\partial^\alpha v \|_{L^p}^{2(p-1)} 
                 \| \theta_{R} (\| u^{(1)} \|_{W^{1,\infty}})\sigma(u^{(1)}) -  \theta_{R} (\| u^{(2)} \|_{W^{1,\infty}})\sigma(u^{(2)})\|_{\WB^{m,p}}^2 ds
	\right)^{1/2} \notag \\	
    & \leq \frac{1}{2} \E \sup_{s \in [0,t \wedge \xi^{K}]}\|\partial^{\alpha} v\|_{L^p}^{p} +   C \E  \int_{0}^{t \wedge \xi^{K}} \|v\|_{W^{m,p}}^{p} ds.\label{eq:J4Est}
\end{align}
We now combine the estimates obtained in \eqref{eq:J1Est}--\eqref{eq:J4Est} and sum over all $\alpha$ with $|\alpha| \leq m $. We find that for any fixed $K >0$
\begin{align*}
    \E \sup_{s \in [0,t \wedge \xi^{K}]}\|v\|_{W^{m,p}}^{p}
    \leq& C \E  \int_{0}^{t \wedge \xi^{K}} \|v\|_{W^{m,p}}^{p}
     (\locCon(\| u^{(1)} \|_{L^\infty} + \| u^{(2)} \|_{L^\infty})^2 + 1)\left(1 +  \|u^{(1)}\|_{W^{m+1,p}}^{2} + \|u^{(2)}\|_{W^{m,p}}^2\right) ds \notag\\
    \leq& C \int_{0}^{t} \E \sup_{r \in [0,s \wedge \xi^{K}]}  \|v\|_{W^{m,p}}^{p} ds
\end{align*}
where the constant $C$ may depend on $K$ via the definition of the stopping time $\xi_{K}$.
By a classical version of the Gr\"onwall lemma, the monotone convergence theorem and the fact
that $\xi^{K}  \rightarrow \infty$ as $K \to \infty$ we infer that, for every $T \geq 0$
$$
\E \sup_{t \in [0,T]}\|v\|_{W^{m,p}}^{p} =0.
$$
Since $T$ is arbitrary, \eqref{eq:uniquenessCutoffSolns1} follows, and the proof of uniqueness is therefore complete.
\end{proof}
\begin{remark}\label{Rmk:MiscOnUniqueness}
With obvious modifications the above proof can be used to show that if $(u^{(1)}, \tau^{(1)})$
and $(u^{(2)}, \tau^{(2)})$ are local pathwise solutions of \eqref{eq:E:1}--\eqref{eq:E:2}
then
   \begin{equation}\label{eq:uniquenessCutoffSolns2}
    \Prb \left( \indFn{u^{(1)}(0) = u^{(2)}(0) } ( u^{(1)}(t) - u^{(2)}(t))  = 0; \forall t \in [0, \tau^{(1)} \wedge \tau^{(2)}] \right) = 1.
  \end{equation}
\end{remark}

With uniqueness for \eqref{eq:LimSys:EulerCutOff} in hand, in order
to establish the existence of pathwise solution,
we shall use the following criteria from
\cite{GyongyKrylov1}.
\begin{lemma}\label{thm:GyongyKry}
  Let $X$ be a complete separable metric space and consider 
  a sequence of $X$ valued random variables $\{Y_j\}_{j \geq 0}$.
  We denote the collection of joint laws of
$\{Y_j\}_{j \geq 1}$ by
   $\{\nu_{j,l}\}_{j, l \geq 1}$, i.e. we take
\begin{displaymath}
    \nu_{j,l}(E) := \Prb( (Y_j, Y_l) \in E),
    \quad E \in \mathcal{B}( X \times X ).
\end{displaymath}
  Then $\{Y_{j}\}_{j \geq 1}$ converges
  in probability if and only if for
  every subsequence of joint probabilities laws, $\{\nu_{j_k, l_k} \}_{ k \geq 0}$,
  there exists a further subsequence which converges weakly to a probability
  measure $\nu$ such that
  \begin{equation}\label{eq:diagonalCondGK}
    \nu( \{(u,v) \in X \times X:  u= v  \}) = 1.
  \end{equation}
\end{lemma}
With this result in mind let us now return again to the sequence of
solutions $u^{j}$ to the system \eqref{eq:Galerkin}
relative to some stochastic basis
$\mathcal{S} = (\Omega, \mathcal{F},$ $\{\mathcal{F}_t\}_{t \geq 0}, \mathbb{P}, {\WW})$
which we fix in advance.
We define sequences of measures $\nu_{j,l}(\cdot) = \Prb((u^{j},u^{l}) \in \cdot)$
and $\mu_{j,l}(\cdot) = \Prb((u^{j},u^{l}, {\WW}) \in \cdot)$ on the phase spaces
$\mathcal{X}_J = \mathcal{X}_{S}\times \mathcal{X}_{S} =C([0,T], X_{\mm-2}) \times C([0,T], X_{\mm-2})$,
$\mathcal{X}_T = \mathcal{X}_{J} \times C([0,T],\mathfrak{U}_0)$
respectively.   
With only minor modifications to the
arguments in Lemma~\ref{thm:Tightness} we see that the
collection $\{\mu_{j,l}\}_{j,l \geq 1}$ is weakly compact.
Extracting a convergent subsequence $\mu_{j,l} \rightharpoonup \mu$ and invoking the
Skorokhod theorem we infer the existence of a probability
space $(\tilde{\Omega}, \tilde{\mathcal{F}}, \tilde{\Prb})$
on which there are defined random elements $(\tilde{u}^{l},\tilde{u}^{l},
\tilde{{\WW}}^{j,l})$
equal in law to  $\mu_{j,l}$ and so that
\begin{equation}\label{eq:jointConv}
  (\tilde{u}^{j},\tilde{u}^{l}, \tilde{\WW}^{j,l}) \rightarrow (\tilde{u}, \tilde{u}^{*}, \tilde{\WW}),
\end{equation}
where the convergence occurs $\tilde{\Omega}$ almost surely in $\mathcal{X}_{T}$.  As above we
infer that each of $(\tilde{u}, \tilde{\WW})$ and $(\tilde{u}^{*},\tilde{\WW})$
are solutions of \eqref{eq:LimSys:EulerCutOff} relative to the
\emph{same} stochastic basis $\mathcal{S} := (\tilde{\Omega}, \tilde{\mathcal{F}}, \tilde{\Prb}, \{\tilde{\mathcal{F}}_t\}, \tilde{\WW})$
with $\tilde{\mathcal{F}}_t$ the completion of  $\sigma$ algebra generated by $\{\tilde{u}(s), \tilde{u}^{*}(s), \tilde{\WW}(s)) : s \leq t\}$.
Define $\nu( \cdot ) = \tilde{\Prb}( (\tilde{u}, \tilde{u}^{*}) \in \cdot)$ and observe that, due to \eqref{eq:jointConv} $\nu^{j,l} \to \nu$, weakly.
Now Proposition~\ref{thm:Uniqueness} implies that $\nu(\{(u,u^{*}) \in \mathcal{X}_{J}:  u= u^{*}\}) = 1$.  Here we use that $H^{\mm-2} \subset W^{m,p}$, and so uniqueness in $W^{m,p}$ (which is proven in Proposition~\ref{thm:Uniqueness}) implies uniqueness everywhere, and hence in $H^{\mm-2}$. We may therefore infer (passing if needed to a subsequence) that $u^{j} \rightarrow u$ in  $\mathcal{X}_{S}$
almost surely, and \emph{on the original probability space}.  Having obtained this convergence and
referring again to \eqref{eq:Galerkin:HmBound} we may thus show that $u$ is a pathwise
solution of \eqref{eq:LimSys:EulerCutOff}.  We finally define the stopping time
$$
	\tau = \inf \left\{ t \geq 0 \colon \|u \|_{W^{m,p}} > R \right\}.
$$
 Note that this stopping time 
 is well defined since $u \in C([0,\infty),X_{\mm-2}) \subset C([0,\infty), X_{m,p})$
for $\mm = m+5$.  Hence,
relative to the initial fixed stochastic basis $\mathcal{S}$, $(u,\tau)$ is a local pathwise
solution of the stochastic Euler equation \eqref{eq:E:1}--\eqref{eq:E:2}, in the sense that
$u(\cdot \wedge \tau) \in L^{\infty}_{loc}([0,\infty); X_{\mm}) \cap C([0,\infty); X_{\mm-2})$
and \eqref{eq:EulerLocTmInt} holds for every $t \geq 0$.

In order to show that $\tau > 0$ we initially assume
$\|u_{0}\|_{{H^\mm}} \leq M$ for some {\em deterministic} $M >0$, and choose
$R > \bar{C}M$, where 
$\bar{C}\geq 1$  is the constant such that $\|u\|_{W^{1,\infty}}
\leq \bar{C} \|u\|_{H^{\mm}}$,  in the cut-off function in \eqref{eq:Galerkin}.  To pass to the
general case $\|u_{0}\|_{H^{\mm}} < \infty$ almost surely, 
we proceed as follows (see e.g. \cite[Section 4.2]{GlattHoltzZiane2}).  
For $k \geq 0$ we define
$u_{0}^{k} = u_{0} \indFn{ k \leq \|u_{0}\|_{H^{\mm}} < k+1}$ and obtain a corresponding
local pathwise solution $(u_{k}, \tau_{k})$ by applying the above construction
with any $R > \bar C (k+1)$ in the cut-off function $\theta_{R}$.  We then define
\begin{align}
u &= \sum_{k \geq 0} u_{k}\indFn{ k \leq \|u_{0}\|_{H^{\mm}} < k+1}
\label{eq:cutThatBitchUpSol}\\
\tau &= \sum_{k \geq 0} \tau_{k}\indFn{ k \leq \|u_{0}\|_{H^{\mm}} < k+1}
\label{eq:cutThatBitchUpST}
\end{align}
and find that
$(u,\tau)$ is in fact the local pathwise solution corresponding to the initial
condition $u_{0}$.

For any fixed $u_{0} \in X_{\mm}$ we
next extend the solution $(u, \tau)$ to a maximal time
of existence $\xi$ (cf. \cite{GlattHoltzZiane2,  MikuleviciusRozovskii2, Jacod1}).
Take $\mathcal{E}$ to be the set of all stopping times $\sigma$ corresponding
to a local pathwise solution of \eqref{eq:E:1}--\eqref{eq:E:2} with initial condition
$u_{0}$.   Let $\xi = \sup \mathcal{E}$ and consider a sequence
$\sigma_{k} \in \mathcal{E}$ increasing to $\xi$.  Due to the local uniqueness of
pathwise solutions we obtain a process $u$ defined on $[0, \xi)$ such that
$(u, \sigma_{k})$ are local pathwise solutions.  For each $r > 0$ we now take
$$
 \rho_{r} = \inf \left\{ t \geq 0 \colon  \| u(t) \|_{W^{1,\infty}} > r \right\} \wedge \xi.
$$
Note that $u$ is continuous on $W^{1,\infty}$ and so $\rho_{r}$ is a well-defined stopping time.
By continuity and uniqueness arguments we may infer that $(u,\rho_{r})$ is a local pathwise
solution for each $r > 0$.\footnote{Note that, for a given $r > 0$, we
may have $\Prb(\rho_{r} = 0) \neq 0$.  However, for almost every $\omega \in \Omega$, there
exists $r > 0$ such that, $\rho_{r}(\omega)> 0$.}  Suppose toward a contradiction
that, for some $T,r >0$ we have
$
 \Prb( \xi = \rho_{r}\wedge T) > 0
$.
Since $(u,\rho_{r}\wedge T)$ is a local
pathwise solution then there exists, another stopping time $\zeta > \rho_{r} \wedge T$
and a process $u^{*}$ such that $(u^{*}, \zeta)$ is a local pathwise
solution corresponding to $u_{0}$,
contradicting the maximality of $\xi$.
Hence we have proven that for every $T,r>0$ we have $\Prb(\xi = \rho_{r}\wedge T) = 0$.
Observe that on the set $\{\xi <\infty\}$, by suitably choosing $T$, we obtain that $\rho_{r}<\xi$
for every $r>0$. On this set we hence have
$
\sup_{t \in [0,\rho_{r}]} \| u(t) \|_{W^{1,\infty}} = r
$
for all $r>0$, which gives
\begin{align}
\sup_{t\in[0,\xi)} \| u(t) \|_{W^{1,\infty}} = \infty, \quad
\textrm{ on the set } \{\xi < \infty\}.\label{eq:BlowUpAtMaxTime}
\end{align}

In summary in this section we have
so far constructed maximal local pathwise $H^{\mm}$ solutions, but only for the non-sharp smoothness regime $\mm = m+5$, with the solution guaranteed to evolve continuously only in $X_{\mm-2}$, and which remains bounded in $X_\mm$. 
In the next section we shall use these very smooth solutions to construct local (maximal) pathwise
$W^{m,p}$ solutions for all $m > d/p + 1$, and for all $p\geq 2$, which will then prove Theorem~\ref{thm:LocalExistence}.

\section{Construction of \texorpdfstring{$W^{m,p}$}{W m p} solutions}
\setcounter{equation}{0}
\label{sec:WmpSolutionsCons}

For $m > d/p +1$ with $p \geq 2$,  we now establish
the local existence of solutions for any initial data
$u_{0} \in X_{m,p}$, which is $\mathcal{F}_{0}$ measurable, which concludes the proof of Theorem~\ref{thm:LocalExistence}. For this purpose we will adapt a density and stability argument from
\cite{KatoLai1984,Masmoudi2007}, which makes use of the very smooth solutions
constructed in Section~\ref{sec:compact}, as approximating solutions.
Indeed, when the initial data lies in $X_{m'}$, where $m'=m+5$, we obtained in Section~\ref{sec:compact} maximal pathwise solutions in the sense of Definition~\ref{def:PathwiseSol}. 
In order to make use of these smooth solutions
we define a sequence of regularized initial data 
\begin{align}
u^{j}_{0} &= F_{j^{-1} }\; u^{0} \label{eq:JBONA:DATA}
\end{align}
where the smoothing operators $F_{j^{-1}}$
and their properties are recalled in Appendix~\ref{sec:SmoothingOperator} below (see also \cite{KatoLai1984}).
For technical reasons 
we assume initially, that $\|u_{0}\|_{W^{m,p}} \leq M$
for some deterministic fixed constant $M$.   As in Section~\ref{sec:compact},
once we obtain the local existence of solutions for each fixed
${{M}}$, this assumption can be relaxed to the general case via
a cutting argument as given in \eqref{eq:cutThatBitchUpSol}--\eqref{eq:cutThatBitchUpST}.
Note that in view of Lemma~\ref{thm:MolProps}, estimate \eqref{eq:MollUniformBnd}
\begin{equation} \label{eq:UniformDataBnd}
	\sup_{j \geq 1} \|u_{0}^{j}\|_{W^{m,p}} \leq C \|u_{0}\|_{W^{m,p}} \leq C{{M}}
\end{equation}
where $C=C(m,p,\DD)$ is a universal constant.
The bound \eqref{eq:UniformDataBnd} will be used in a crucial way in the forthcoming estimates.  Since $F_{j^{-1}}$ is smoothing,
$\{u^{j}_{0}\}_{j \geq 1} \subset X_{m'}$, and we obtain from
the results in Section~\ref{sec:compact} a sequence $(u^{j}, \xi^{j})$
of maximal, pathwise solutions evolving continuously in $X_{m'-2}$ which are bounded in $X_{m'}$.  
In order to show that
this sequence converges to a local $X_{m,p}$ solution corresponding to the
initial condition $u_{0}$ we show that, up to some stopping time $\tau > 0$
the sequence $\{u^{j}\}_{j\geq 1}$ is Cauchy and hence convergent in $C([0,\tau); X_{m,p})$.

To obtain this convergence (along with an associated stopping time $\tau$) we apply
an abstract result from \cite{GlattHoltzZiane2} (and see also \cite{MikuleviciusRozovskii2}). 
For this purpose pick
fix any $T >0$ and define the sequence of stopping times
\begin{equation}\label{eq:stoppingTimes}
  \tau^{T}_j := \inf \left\{ t \geq 0\colon
     \|u^{j}(t)\|_{W^{m,p}} \geq 2 + \|u^{j}_{0}\|_{W^{m,p}} \right\} \wedge T,
\end{equation}
and let 
\begin{align}
\tau^{T}_{j,k} = \tau^{T}_j \wedge \tau^{T}_k \label{eq:stoppingTimes:Dual}
\end{align}
for $j,k\geq 1$.
Since $W^{m,p}$ is continuously embedded in $W^{1,\infty}$ it is clear that $\tau^{T}_{j} < \xi^{j}$, where  as usual $\xi^j$ is the maximal (stopping) time of existence of $u^j$, i.e.
\begin{equation}\label{eq:SeqReg1}
  \sup_{t \in [0, \xi^j]} \|u^{j}(t)\|_{W^{m,p}} = \infty, \quad \textrm{ on the set } \{\xi^j < \infty\}.
\end{equation}
From \cite[Lemma 5.1]{GlattHoltzZiane2} we recall:
\begin{lemma}[\bf Abstract Cauchy lemma]\label{thm:absConv}
  For $T>0$ and $\tau_{j,k}^{T}$ as defined in \eqref{eq:stoppingTimes:Dual}, suppose that we have
    \begin{equation}\label{eq:strngTypeCompCondXmp}
      \lim_{k \to \infty}
          \sup_{j \geq k}
          \mathbb{E} \sup_{t \in [0, \tau_{j,k}^{T}]}\| u^j(t) - u^k(t) \|_{W^{m,p}}= 0
    \end{equation}
    and
    \begin{equation}\label{eq:GrowthAsmpXmp}
      \lim_{S \rightarrow 0}
      \sup_{j\geq 1}
      \mathbb{P}\left[ \sup_{t \in [0,\tau_{j}^{T} \wedge S]}
        \|u^j(t)\|_{W^{m,p}}
         > \|u_{0}^{j}\|_{W^{m,p}} +  1
       \right] = 0.
    \end{equation}
    Then, there exists a stopping time $\tau$ with:
    \begin{equation}\label{eq:limStpTmConcls}
      \mathbb{P}( 0 < \tau \leq T) = 1,
    \end{equation}
    and a process predictable process $u(\cdot) = u(\cdot \wedge \tau) \in
   C([0,\tau], X_{m,p})$ such that
    \begin{equation}\label{eq:convConcl}
      \sup_{t \in [0, \tau]} \| u^{j_l} - u \|_{W^{m,p}} \rightarrow 0, \quad a.s.
    \end{equation}
    for some subsequence $j_l \rightarrow \infty$.  Moreover, the bound
    \begin{equation}\label{eq:pntWiseLimitBnd}
     \|u(t)\|_{W^{m,p}} \leq 2 + \sup_j \|u^{j}_0\|_{W^{m,p}},\quad a.s.
    \end{equation}
     holds uniformly for $t \in [0, \tau]$.
\end{lemma}

In view of Lemma~\ref{thm:absConv}, we may now establish the
essential convergence needed for Theorem~\ref{thm:LocalExistence} 
in the general case by verifying \eqref{eq:strngTypeCompCondXmp}
and \eqref{eq:GrowthAsmpXmp}.  
To prove \eqref{eq:strngTypeCompCondXmp} we fix arbitrary $j,k \geq 1$ and denote $v = u^{k} - u^{j}$ where $v_{0} = u^{k}_{0} - u^{j}_{0}$.
We have
\begin{align*}
  dv + P\left(v \cdot \nabla u^{k} + u^{j} \cdot \nabla v \right) dt = P(\sigma(u^{j}) -  \sigma(u^{k}))  {\dW}.
\end{align*}
Applying $\partial^{\alpha}$ to this system and then the It\={o}
lemma in $L^{p}$ we obtain
\begin{align}
  d \|\partial^{\alpha}v\|^p_{L^p} =& -
   p \int_{\DD}
    \partial^\alpha v  \cdot
      \partial^\alpha P(v \cdot \nabla u^{k} + u^{j} \cdot \nabla v  ) |\partial^{\alpha} v|^{p -2}dx dt
      \notag \\
      &+ \sum_{l \geq 1}\int_{\DD}   \biggl( \frac{p}{2}
       |\partial^\alpha P(\sigma_{l}(u^{j}) -  \sigma_{l}(u^{k}))|^2 |\partial^{\alpha} v|^{p -2} \notag\\
       &\qquad \qquad 
       + \frac{p(p-2)}{2} (\partial^\alpha v  \cdot P (\sigma_{l}(u^{j}) -  \sigma_{l}(u^{k})))^{2} |\partial^{\alpha} v|^{p -4}
               \biggr)
       \, dx  dt \notag \\
     &+ p \sum_{l \geq 1} \left(
       \int_{\DD}   \partial^\alpha v  \cdot
                       \partial^\alpha P(\sigma_{l}(u^{j}) -  \sigma_{l}(u^{k})) |\partial^{\alpha} v|^{p -2}\,dx
       \right) dW_l \notag \\
       :=& (J_1^{\alpha} + J_2^{\alpha}) dt + J_3^{\alpha} {\dW}.
       \label{eq:DiffOfPairLpEvolution}
\end{align}
Using \eqref{eq:DiffOfPairLpEvolution}, we now estimate $v$ in $W^{m,p}$.  For the nonlinear terms we use Lemma~\ref{lemma:Leray}
and infer
\begin{align}
 \sum_{\alpha \leq m}|J^{\alpha}_{1}|
 &\leq C\|P (v \cdot \nabla u^{k})\|_{W^{m,p}}\|v\|_{W^{m,p}}^{p-1}
 +   
     \sum_{\alpha \leq m}  |(\partial^\alpha P( u^{j} \cdot \nabla v ),  \partial^\alpha v  |\partial^{\alpha} v|^{p -2})| \notag\\
 &\leq C \|v\|_{W^{m,p}}^{p-1} ( \|v\|_{L^{\infty}} \| u^{k} \|_{W^{m+1,p}} + \|v\|_{W^{m,p}} \| \nabla u^{k} \|_{L^{\infty}}) \notag\\
 & \qquad \qquad   + C \|v\|_{W^{m,p}}^{p-1}(\|u^{j}\|_{W^{1,\infty}} \|v\|_{W^{m,p}}
+  \|v\|_{W^{1,\infty}} \|u^{j}\|_{W^{m,p}})\notag\\
 &\leq C  \| u^{k} \|_{W^{m+1,p}}  \|v\|_{W^{m-1,p}} \|v\|_{W^{m,p}}^{p-1} + C  (\|u^{k}\|_{W^{m,p}} + \|u^{j}\|_{W^{m,p}}) \|v\|_{W^{m,p}}^{p}\notag\\
  &\leq C  \| u^{k} \|_{W^{m+1,p}}^{p}  \|v\|_{W^{m-1,p}}^{p} + C  (\|u^{k}\|_{W^{m,p}} + \|u^{j}\|_{W^{m,p}} + 1) \|v\|_{W^{m,p}}^{p}.
 \label{eq:XmContDependenceProblem}
\end{align}
Note that the first term in the final inequality prevents one from directly closing
the estimates for $v$ in $W^{m,p}$.  We will therefore need to make further estimates
for $u^{k}$ in $W^{m+1,p}$ and $v$ in $W^{m-1,p}$, cf.~\eqref{eq:EvolutionIndividualTermsProdu}--\eqref{eq:EvolutionIndividualTermsProdv} below.
For the terms involving $J_{2}^{\alpha}$ we use the local Lipschitz condition \eqref{eq:sigMappings}
and obtain
\begin{align}
	\sum_{\alpha \leq m } |J_{2}^{\alpha}|
	&\leq C\|v\|^{p-2}_{W^{m,p}}
        \left( \sum_{\alpha \leq m } \int_{\DD}
        \left(\sum_{l \geq 1} |\partial^\alpha P(\sigma_{l}(u^{j}) -  \sigma_{l}(u^{k}))|^2 \right)^{p/2} \right)^{2/p} \notag \\
        &\leq C  \|v\|^{p-2}_{W^{m,p}} \|P(\sigma(u^{j}) -  \sigma(u^{k}))\|^2_{\WB^{m,p}}
        \leq  C \locCon(\|u^{k}\|_{L^\infty} + \| u^{j}\|_{L^\infty})^2 \|v\|^{p}_{W^{m,p}}. 	
        \label{eq:XmContDependenceItoCorrectionTerms}
\end{align}
Finally, estimating in a similar manner to \eqref{eq:MFJonesConsultsForBDG},
we find that for any stopping time $\tau$, 
\begin{align}
  \E \left( \sup_{s \in [0,\tau]} \left|\int_{0}^{s}J_{3}^{\alpha}dW \right| \right)
  	\leq&  C\E \left( \int_{0}^{\tau} \sum_{k \geq 1} \left(
       \int_{\DD}   \partial^\alpha v  \cdot
                       \partial^\alpha P(\sigma_{l}(u^{j}) -  \sigma_{l}(u^{k})) |\partial^{\alpha} v|^{p -2}\,dx
       \right)^{2}ds
	\right)^{1/2} \notag \\
	\leq& \frac{1}{2} \E \sup_{s \in [0,\tau]}\|\partial^{\alpha} v\|^{p}_{L^p} + C
	\E  \int_{0}^{\tau}\locCon(\|u^{k}\|_{L^\infty} + \| u^{j}\|_{L^\infty})^2 
	\|v\|_{W^{m,p}}^p
       ds.
       \label{eq:XmContDependenceStochasticIntegralTerms}
\end{align}
Combining the estimates obtained in \eqref{eq:XmContDependenceProblem}--\eqref{eq:XmContDependenceStochasticIntegralTerms} 
and recalling the definition of $\tau^T_{j,k}$ in \eqref{eq:stoppingTimes:Dual} we find that
\begin{align*}
   \E \left( \sup_{ [0, \tau^{T}_{j,k} \wedge t]} \|v\|_{W^{m,p}}^{p} \right)
   \leq& 2\E  \|v_{0}\|_{W^{m,p}}^{p}
   	  + C \E  \int_{0}^{\tau^{T}_{j,k} \wedge  t}  
	  \left(\| u^{j}\|_{W^{m,p}} + \|u^{k}\|_{W^{m,p}} + \locCon(\|u^{k}\|_{L^\infty} + \| u^{j}\|_{L^\infty})^2\right)
	  \|v\|^{p}_{W^{m,p}}  ds\\
	  &+ C \E  \int_{0}^{\tau^{T}_{j,k} \wedge  t}  \left(  \| u^{k} \|_{W^{m+1,p}}^{p}  \|v\|_{W^{m-1,p}}^{p} \right)ds\\
   \leq& 2\E  \|v_{0}\|_{W^{m,p}}^{p}
   	  + C   \int_{0}^{t}
	  \left( \E  \sup_{ [0, \tau^{T}_{j,k} \wedge s]} \|v\|_{W^{m,p}}^{p}
	  +
	  \E  \sup_{ [0, \tau^{T}_{j,k} \wedge s]}( \|v\|_{W^{m-1,p}}^{p} \|u^{k}\|_{W^{m+1,p}}^{p}) \right) ds
\end{align*}
where $C$ is a positive constant that depends on ${M}$ and $\locCon$ but
is independent of $j,k$. 
By an application of the classical Gr\"onwall lemma we obtain that
\begin{align*}
  \E \left( \sup_{ [0, \tau^{T}_{j,k}]} \| u^{k}- u^{j} \|_{W^{m,p}}^{p} \right)=  \E \left( \sup_{ [0, \tau^{T}_{j,k}]} \|v\|_{W^{m,p}}^{p} \right)
    \leq C\left( \E \|u_{0}^{k} - u_{0}^{j}\|_{W^{m,p}}^{p} +
        \E  \sup_{[0, \tau^{T}_{j,k}]  }( \|v\|_{W^{m-1,p}}^{p} \|u^{k}\|_{W^{m+1,p}}^{p})
    \right)
\end{align*}
where $C = C(m, p, \DD,{M}, T)$ is  {\em independent} of both $j, k$.  Observe that, in view
of Lemma~\ref{thm:MolProps}, estimate \eqref{eq:MollConvEspWmp}, and applying
the dominated convergence theorem we conclude that $\sup_{j \geq k} \E \|u_{0}^{k} - u_{0}^{j}\|_{W^{m,p}}^p$
goes to zero as $k \rightarrow \infty$.  As such, \eqref{eq:strngTypeCompCondXmp} will follow once we
show that
\begin{align}
 \lim_{k \rightarrow \infty}
 \sup_{j \geq k}\E  \sup_{  [0, \tau^{T}_{j,k} ] }( \|v\|_{W^{m-1,p}}^{p} \|u^{k}\|_{W^{m+1,p}}^{p})
 =0.
 \label{eq:AnnoyingProdEstGoal}
\end{align}

With this goal of establishing \eqref{eq:AnnoyingProdEstGoal} in mind, 
let us determine $d (\|v\|_{W^{m-1,p}}^{p} \|u^{k}\|_{W^{m+1,p}}^{p})$.
We have cf. \eqref{eq:LpItoFormalSoln} and \eqref{eq:DiffOfPairLpEvolution} that
\begin{align}
  d\|u^{k}\|_{W^{m+1,p}}^{p}
  &=  ({I}_{1} +{I}_{2}) dt +  {I}_{3} {\dW},
    \label{eq:EvolutionIndividualTermsProdu}\\
  d \|v\|_{W^{m-1,p}}^{p}
  &= ({J}_{1} +  {J}_{2} )dt +  {J}_{3} {\dW},
  \label{eq:EvolutionIndividualTermsProdv}
\end{align}
where, to make the notation less cumbersome, we take
\begin{align*}
  I_{l} = \sum_{|\alpha| \leq m +1 } I_{l}^{\alpha}, \qquad \textrm{ and }
 \quad J_{l} = \sum_{|\alpha| \leq m - 1 } J_{l}^{\alpha}  \quad \textrm{ for } l = 1, 2,3.
\end{align*}
The elements $I_{l}^{\alpha}$ are defined as in \eqref{eq:LpItoFormalSoln}
(with $u$ replaced with $u^{k}$ throughout) and $J_{l}^{\alpha}$ are as in 
\eqref{eq:DiffOfPairLpEvolution}. 
By an application of the It\={o} product rule we find that
\begin{align*}
    d (\|v\|_{W^{m-1,p}}^{p}& \|u^{k}\|_{W^{m+1,p}}^{p}) \\
    =&
     \|v\|_{W^{m-1,p}}^{p}  d \|u^{k}\|_{W^{m+1,p}}^{p}
    +   \|u^{k}\|_{W^{m+1,p}}^{p} d \|v\|_{W^{m-1,p}}^{p}
    +d \|v\|_{W^{m-1,p}}^{p} d \|u^{k}\|_{W^{m+1,p}}^{p}\\
     =& \left( \|v\|_{W^{m-1,p}}^{p} ({I}_{1} + {I}_{2})
    +  \|u^{k}\|_{W^{m+1,p}}^{p} ({J}_{1} +  {J}_{2})
    + {K} \right) dt
    + \left(  \|v\|_{W^{m-1,p}}^{p}{I}_{3}  +\|u^{k}\|_{W^{m+1,p}}^{p}{J}_{3} 
    \right) \dW,
\end{align*}
where $K$ is the term arising from $I_{3} \dW J_{3} \dW$ and is given by
\begin{align*}
{K} = p^{2} \sum_{l \geq 1}
  \left( \sum_{|\alpha| \leq m +1}        \int_{\DD}   \partial^\alpha u^{k}  \cdot
                       \partial^\alpha P\sigma_{l}(u^{k}) |\partial^{\alpha} u^{k}|^{p -2}\,dx \right)
 \left( \sum_{|\alpha| \leq m -1}         \int_{\DD}   \partial^\alpha v  \cdot
                       \partial^\alpha P(\sigma_{l}(u^{j}) -  \sigma_{l}(u^{k})) |\partial^{\alpha} v|^{p -2}\,dx \right).
\end{align*}
In view of the estimates carried out in Section~\ref{sec:apriori} (cf. \eqref{eq:J1EstLp}--\eqref{eq:ItoCorTemEstsLp})
and making use of the assumption \eqref{eq:ExtraDerivativeAssumptforDensityStability} we immediately infer that
\begin{align}
   \left| \|v\|_{W^{m-1,p}}^{p} ({I}_{1} + {I}_{2}) \right|
   \leq& C (\locCon(\|u^{k}\|_{L^\infty} )^2 + \|u^{k}\|_{W^{1,\infty}})\|u^{k}\|_{W^{m+1,p}}^{p} \|v\|_{W^{m-1,p}}^{p} 
   + C \locCon(\|u^{k}\|_{L^\infty} )^2 \|v\|_{W^{m-1,p}}^{p}.
   \label{eq:vujProdEstuTerms}
\end{align}
We next treat the drift terms in \eqref{eq:EvolutionIndividualTermsProdv}.  For ${J}_{1}$, recalling that $P = I -Q$ we write
\begin{align}
  | {J}_{1} | 
  &\leq p \sum_{|\alpha|\leq m-1}   \left| \int_{\DD}
    \partial^\alpha v  \cdot 
      \partial^\alpha P(v \cdot \nabla u^{k} + u^{j} \cdot \nabla v  ) |\partial^{\alpha} v|^{p -2}dx \right| \notag\\
 &\leq C \| v\|_{W^{m-1,p}}^{p-1} \| P(v \cdot \nabla u^k)\|_{W^{m-1,p}} + C\sum_{|\alpha|\leq m-1} \left| \int_{\DD}
    \partial^\alpha v  \cdot \partial^\alpha (u^{j} \cdot \nabla v) |\partial^{\alpha} v|^{p -2}dx \right| \notag\\
   &\qquad \qquad \qquad
    + C \sum_{|\alpha|\leq m-1} \left| \int_{\DD} \partial^\alpha v  \cdot \partial^\alpha Q(u^{j} \cdot \nabla v) |\partial^{\alpha} v|^{p -2}dx \right|\notag\\
   &= J_{1,1} + J_{1,2} + J_{1,3} \label{eq:NastySobolevEstimate}
 \end{align}
The right side of the above estimate may be bounded as follows. 
To bound $J_{1,1}$ we use Lemma~\ref{lemma:Leray} and obtain
\begin{align}
J_{1,1} \leq C  \| v\|_{W^{m-1,p}}^{p-1} \left( \|v\|_{L^{\infty}} \| u^{k}\|_{W^{m,p}} + \|v\|_{W^{m-1,p}} \| u^{k}\|_{W^{1,\infty}} \right)
\leq C \|v\|_{W^{m-1,p}}^{p} \| u^{k}\|_{W^{m,p}}. \label{eq:J11}
\end{align}
For the other two terms on the right side of \eqref{eq:NastySobolevEstimate} we cannot estimate as in Lemma~\ref{lemma:Leray} directly;
we would obtain bound of the type $\|u^{j}\|_{W^{m-1,p}} \|v\|_{W^{m,p}} \|v\|_{W^{m-1,p}}^{p-1}$, which would prevents us from closing
the estimates involving $\|v\|_{W^{m-1,p}}^{p}$. 
To bound $J_{1,2}$ we we use the Leibniz rule, the H\"older and Gagliardo-Nirenberg inequalities. 
There is only one non-standard term $\| \partial^\alpha u^j \cdot \nabla v \|_{L^p}$, which is bounded as 
$$
\sum_{|\alpha|\leq m-1}  \| \partial^\alpha u^j \cdot \nabla v \|_{L^p} \leq C \| u^j \|_{W^{m-1,q}} \|\nabla v\|_{L^{r}} \leq \| u^j \|_{W^{m,p}} \|v\|_{W^{m-1,p}},
$$ 
where $q =p d /(d-p)$, $r =pq/(q-p) = d$ if $p<d$, and $q = \infty$, $r = p$ if $p \geq d$. The other terms are bounded as in Lemma~\ref{lemma:Leray}, and we obtain
\begin{align}
   {J}_{1,2} &\leq C \|v\|_{W^{m-1,p}}^{p} \|u^{j}\|_{W^{m,p}}.
 \label{eq:J12}
\end{align}
Lastly, the ``pressure term'' $J_{1,3}$ is estimated using the the H\"older inequality, the Agmon-Douglis-Nirenberg estimate 
\eqref{eq:Agmon}, and the Gagliardo-Nirenberg inequality as 
\begin{align}
J_{1,3} \leq& \|Q(u^{j} \cdot \nabla v)\|_{W^{m-1,p}} \|v \|_{W^{m-1,p}}^{p-1} \notag\\
\leq&  C ( \|  \partial_i u_l^{j} \partial_l v_i \|_{W^{m-2,p}} + \| u^{j} v \|_{W^{m-1,p}}) \|v \|_{W^{m-1,p}}^{p-1} 
\leq C \|v\|_{W^{m-1,p}}^{p} \| u^{j}\|_{W^{m,p}} . \label{eq:J13}
\end{align}
Combining  \eqref{eq:J11} --\eqref{eq:J13} we conclude
\begin{align}
    |{J}_{1}|
 	\leq C \|v\|_{W^{m-1,p}}^{p} (\|u^{j}\|_{W^{m,p}} + \| u^{k}\|_{W^{m,p}} ).
	 \label{eq:estvJ2Wm1}
\end{align}
For  ${J}_{2}$ we find, as above in \eqref{eq:XmContDependenceItoCorrectionTerms}
that
\begin{align}
	|{J}_{2}|
	\leq C\|v\|^{p-2}_{W^{m-1,p}} \|P(\sigma(u^{j}) -  \sigma(u^{k}))\|^2_{\WB^{m-1,p}}
        \leq C \locCon(\|u^{k}\|_{L^\infty} + \| u^{j}\|_{L^\infty})^2\|v\|^{p}_{W^{m-1,p}}.
         \label{eq:estvJ3Wm1}
\end{align}
Combining \eqref{eq:estvJ2Wm1}--\eqref{eq:estvJ3Wm1}  we find
\begin{align}
  \left| \|u^{k}\|_{W^{m+1,p}}^{p} ({J}_{1} +  {J}_{2} ) \right|
  \leq C  (\locCon(\|u^{k}\|_{L^\infty} + \| u^{j}\|_{L^\infty})^2  + \|u^{j}\|_{W^{m,p}} + \| u^{k}\|_{W^{m,p}} )  \|u^{k}\|_{W^{m+1,p}}^{p} \|v\|_{W^{m-1,p}}^{p}.
   \label{eq:vujProdEstvTerms}
\end{align}
The term $K$ is estimated using the H\"older and Minkowski inequalities followed by the standing
assumption on $\sigma$, \eqref{eq:sigMappings},
\begin{align}
|K| 
\leq& C \left( \sum_{l \geq 1} \left( \sum_{|\alpha| \leq m +1} \int_{\DD} |\partial^\alpha P \sigma_{l}(u^{k})| |\partial^{\alpha} u^{k}|^{p -1}\,dx \right)^{2}  \right)^{1/2} \notag\\
& \qquad \qquad 
 	  \left( \sum_{l \geq 1} \left( \sum_{|\alpha| \leq m -1}  \int_{\DD} |\partial^\alpha P(\sigma_{l}(u^{j}) -  \sigma_{l}(u^{k}))| |\partial^{\alpha} v|^{p -1}\,dx \right)^{2}  \right)^{1/2} 		\notag \\
\leq& C \sum_{|\alpha| \leq m+1} \int_{\DD} \left( \sum_{l \geq 1}  |\partial^\alpha P\sigma_{l}(u^{k})|^{2} \right)^{1/2} |\partial^{\alpha} u^{k}|^{p -1}\,dx \notag\\
& \qquad \qquad 
 	  \sum_{|\alpha| \leq m-1} \int_{\DD} \left( \sum_{l \geq 1} |\partial^\alpha P(\sigma_{l}(u^{j}) -  \sigma_{l}(u^{k}))|^{2}  \right)^{1/2} |\partial^{\alpha} v|^{p -1}\,dx
                     \notag\\
\leq&C \|u^{k}\|_{W^{m+1,p}}^{p-1} \left(  \sum_{|\alpha| \leq m +1}   \int_{\DD} \left( \sum_{l \geq 1}  |\partial^\alpha P\sigma_{l}(u^{k})|^{2} \right)^{p/2} dx \right)^{1/p} 	\notag\\
& \qquad \qquad 
					 \|v\|_{W^{m-1,p}}^{p-1} \left( \sum_{|\alpha| \leq m -1}   \int_{\DD} \left( \sum_{l \geq 1} |\partial^\alpha P(\sigma_{l}(u^{j}) -  \sigma_{l}(u^{k}))|^{2}  \right)^{p/2} dx\right)^{1/p} 
		\notag\\
\leq& C  \|u^{k}\|_{W^{m+1,p}}^{p-1} \|P \sigma(u^{k})\|_{\WB^{m+1, p}} \|v\|_{W^{m-1,p}}^{p-1}  \|P  (\sigma(u^{j}) -  \sigma(u^{k})  ) \|_{\WB^{m-1,p}} 
		\notag\\
\leq& C \locCon(\|u^{k}\|_{L^\infty} + \| u^{j}\|_{L^\infty})^2
 \left( \|u^{k}\|_{W^{m+1,p}}^{p} \|v\|_{W^{m-1,p}}^{p} + \|v\|_{W^{m-1,p}}^{p} \right).
\label{eq:vujProdEstWeirdItoProdTerm}
\end{align}
To treat the stochastic terms we proceed similarly to \eqref{eq:MFJonesConsultsForBDG}
and find that for any stopping time $\tau$
\begin{align}
  &\E \left( \sup_{s \in [0,\tau]} \left|\int_{0}^{s} \|u^{k}\|_{W^{m+1,p}}^{p}{J}_{3} {\dW} \right| \right) \notag\\
  &\qquad \leq  C\E \left( \int_{0}^{\tau}   \|u^{k}\|_{W^{m+1,p}}^{2p} \sum_{l \geq 1} \left(
       \sum_{|\alpha| \leq m-1} \int_{\DD}
                      |\partial^\alpha P(\sigma_{l}(u^{j}) -  \sigma_{l}(u^{k}))|
                       |\partial^{\alpha} v|^{p -1}\,dx
       \right)^{2}ds
	\right)^{1/2} \notag \\
  &\qquad \leq  C\E \left( \int_{0}^{\tau}   \|u^{k}\|_{W^{m+1,p}}^{2p} \left(
       \sum_{|\alpha| \leq m-1} \int_{\DD}  \left( \sum_{l \geq 1}
                      |\partial^\alpha P(\sigma_{l}(u^{j}) -  \sigma_{l}(u^{k}))|^{2} \right)^{1/2}
                       |\partial^{\alpha} v|^{p -1}\,dx
       \right)^{2}ds
	\right)^{1/2} \notag \\
  &\qquad \leq  C\E \left( \int_{0}^{\tau}   \|u^{k}\|_{W^{m+1,p}}^{2p} 
	\|P(\sigma(u^{j}) -  \sigma(u^{k})) \|^{2}_{\WB^{m-1,p}}
	\|v\|_{W^{m-1,p}}^{2(p-1)}
	ds
	\right)^{1/2} \notag \\
  &\qquad \leq \frac{1}{4} \E \sup_{s \in [0,\tau]}\left(\| v\|^{p}_{W^{m-1,p}} \|u^{k}\|_{W^{m+1,p}}^{p}\right) + C
	\E  \int_{0}^{\tau} \locCon(\|u^{k}\|_{L^\infty} + \| u^{j}\|_{L^\infty})^2
	\|v\|_{W^{m-1,p}}^p\|u^{k}\|_{W^{m+1,p}}^{p}
       ds.
        \label{eq:vujProdEstvStocTermsvHard}
\end{align}
Similarly to \eqref{eq:vujProdEstvStocTermsvHard} above we also obtain
\begin{align}
  \E& \left( \sup_{s \in [0,\tau]} \left|\int_{0}^{s} \|v\|_{W^{m-1,p}}^{p}{I}_{3} {\dW} \right| \right) \notag\\
  	&\qquad \quad \leq  C\E \left( \int_{0}^{\tau}   \|v\|_{W^{m-1,p}}^{2p} 
	\|P( \sigma(u^{k})) \|^{2}_{\WB^{m+1,p}}
	\|u^{k}\|_{W^{m+1,p}}^{2(p-1)}
	ds
	\right)^{1/2} \notag \\
	&\qquad \quad \leq \frac{1}{4} \E \sup_{s \in [0,\tau]}\left(\| v\|^{p}_{W^{m-1,p}} \|u^{k}\|_{W^{m+1,p}}^{p}\right) + C
	\E  \int_{0}^{\tau} \locCon(\|u^{k}\|_{L^\infty})^2
	\|v\|_{W^{m-1,p}}^p(\|u^{k}\|_{W^{m+1,p}}^{p} + 1)
       ds.
        \label{eq:vujProdEstvStocTermsuHard}
\end{align}

Summarizing, from the estimates \eqref{eq:vujProdEstuTerms}, \eqref{eq:vujProdEstvTerms}--\eqref{eq:vujProdEstvStocTermsuHard},
and the definition of $\tau^{T}_{n,m}$ in \eqref{eq:stoppingTimes} we find that
\begin{align*}
   \E &\left( \sup_{t \in [0, \tau^{T}_{j,k} \wedge t]} \|v\|_{W^{m-1,p}}^{p} \|u^{k}\|_{W^{m+1,p}}^{p}  \right)\\
   &\quad \leq 2\E  \left( \|v_{0}\|_{W^{m-1,p}}^{p} \|u^{k}_{0}\|_{W^{m+1,p}}^{p} \right)
   	  + C \E  \int_{0}^{t} \left(
	  \sup_{t \in [0, \tau^{T}_{j,k} \wedge s]}
	  \left( \|v\|_{W^{m-1,p}}^{p} \|u^{k}\|_{W^{m+1,p}}^{p} \right)
	  +	
	  \sup_{t \in [0, \tau^{T}_{j,k} \wedge s]} \|v\|_{W^{m-1,p}}^{p}\right)
	  ds
\end{align*}
for any  $t > 0$ where the constant $C$ depends on ${M}$, $\locCon$ and the data
but not on $j$, $k$.
Thus, again invoking the Gr\"onwall lemma finally conclude that
\begin{align}
 \E& \left( \sup_{t \in [0, \tau^{T}_{j,k}]}  \|u^{k}\|_{W^{m+1,p}}^{p} \|v\|_{W^{m-1,p}}^{p}  \right)\notag\\
 &\qquad \qquad \qquad \leq  C \E \left( \|u^{k}_{0}\|_{W^{m+1,p}}^{p} \|u^{k}_{0} - u^{j}_{0}\|_{W^{m-1,p}}^{p}  \right) +C
  \E \left( \sup_{t \in [0, \tau^{T}_{j,k}]}  \|u^{k}- u^{j}\|_{W^{m-1,p}}^{p}  \right)
   \label{eq:FracEspConvWeakerNormProd}
\end{align}
where the constant $C$ is independent of $j,k$.  By the dominated convergence theorem
(for $(\Omega, \mathcal{F}, \Prb)$)
and making use of the properties of the smoothing operator $F_{\epsilon}$, cf.~\eqref{eq:MollDivEpsScaling} and
\eqref{eq:MollConvEspWmmin1pGain}, we find
\begin{align*}
   \lim_{k \rightarrow \infty}
    \sup_{j \geq k}\E \left( \|u^{k}_{0}\|_{W^{m+1,p}}^{p} \|u^{k}_{0} - u^{j}_{0}\|_{W^{m-1,p}}^{p}  \right)
    \leq C     \lim_{k \rightarrow \infty}
            \sup_{j \geq k}\E \left( \|u_{0}\|_{W^{m,p}}^{p} k^{p}\|u^{k}_{0} - u^{j}_{0}\|_{W^{m-1,p}}^{p}  \right)
            = 0.
\end{align*}
To handle the second term in \eqref{eq:FracEspConvWeakerNormProd}
we refer back to \eqref{eq:EvolutionIndividualTermsProdv} and the
estimates in \eqref{eq:estvJ2Wm1}--\eqref{eq:estvJ3Wm1}.  The stochastic
terms involving ${J}_{3}$ are handled in a similarly to \eqref{eq:vujProdEstvStocTermsvHard}
(and also cf.~\eqref{eq:MFJonesConsultsForBDG} above).  Combining these observation,
using the Gr\"onwall inequality and the properties of $F_{\epsilon}$ we finally infer:
\begin{align}
   \lim_{k \rightarrow \infty} \sup_{j \geq k} \E \left( \sup_{t \in [0, \tau^{T}_{j,k}]} \|u^{j} - u^{k}\|_{W^{m-1,p}}^{p}  \right) = 0.
   \label{eq:FracEspConvWeakerNorm}
\end{align}
With this final observation in place we have now established \eqref{eq:AnnoyingProdEstGoal} and hence
the first requirement \eqref{eq:strngTypeCompCondXmp} of Lemma~\ref{thm:absConv}.

To establish the second condition  \eqref{eq:GrowthAsmpXmp} required by Lemma~\ref{thm:absConv}, we return to
\eqref{eq:LpItoFormalSoln}.  We find that, for any $k \geq 1$ and $S > 0$
\begin{align*}
	\sup_{t \leq [0, \tau^{T}_{k} \wedge S]} \|u^{k}(t)\|_{W^{m,p}}^{p}
	 \leq  \|u^{k}_{0}\|_{W^{m,p}}^{p}
	 + \sum_{| \alpha| \leq m} \int_{0}^{\tau^{T}_{k} \wedge S}
	 	|I^{\alpha}_{1} + I^{\alpha}_{2}| dt
	+  \sup_{t \leq [0, \tau^{T}_{k} \wedge S]}
	\left| \int_{0}^{t}   \sum_{| \alpha| \leq m} I^{\alpha}_{3} \dW \right|,
\end{align*}
and hence
\begin{align}
  \Prb &\left(
  \sup_{t \leq [0, \tau^{T}_{k} \wedge S]} \|u^{k}(t)\|_{W^{m,p}}^{p}
  > \|u^{k}_{0}\|_{W^{m,p}}^{p} + 1 \right) 
  \notag \\
  &\qquad \qquad \qquad
  \leq
    \Prb \left(
    \sum_{| \alpha| \leq m} \int_{0}^{\tau^{T}_{k} \wedge S}
	 	|I^{\alpha}_{1} + I^{\alpha}_{2} | dt
  > \frac{1}{2} \right)
  +
      \Prb \left(
      \sup_{t \leq [0, \tau^{T}_{k} \wedge S]}
	\left| \int_{0}^{t}   \sum_{| \alpha| \leq m} I^{\alpha}_{3} dW \right|
  > \frac{1}{2} \right).
  \label{eq:SmallGrowthAtZeroSplit}
\end{align}
For the first term on the right of \eqref{eq:SmallGrowthAtZeroSplit}
we apply the estimates in \eqref{eq:J1EstLp}--\eqref{eq:ItoCorTemEstsLp} and then the Chebyshev Inequality and find
\begin{align}
    \Prb  \left(
    \sum_{| \alpha| \leq m} \int_{0}^{\tau^{T}_{k} \wedge S}
	 	|I^{\alpha}_{1} + I^{\alpha}_{2} | dt
  > \frac{1}{2} \right)
  \leq&
      \Prb \left(
     \int_{0}^{\tau^{T}_{k} \wedge S}
	C (\locCon(\|u^{k}\|_{L^\infty})^2 + \|u^{k}\|_{W^{m,p}})\|u^{k}\|_{W^{m,p}}^{p} dt
  > \frac{1}{2} \right) 
  \notag\\
  \leq& 
   C \E      \int_{0}^{\tau^{T}_{k} \wedge S}
	 (\locCon(\|u^{k}\|_{L^\infty})^2+ \|u^{k}\|_{W^{m,p}})\|u^{k}\|_{W^{m,p}}^{p}  dt
     \leq C S \label{eq:SmallGrowthAtZeroEst1Cheb},
\end{align}
where the constant $C = C(m,p, M, \locCon,\DD)$ is independent of
$k$ and $S$. With Doob's inequality, the It\={o} Isometry
and the integral Minkowski inequality we estimate the second term
\begin{align}
      \Prb \left(
      \sup_{t \leq [0, \tau^{T}_{k} \wedge S]}
	\left| \int_{0}^{t}   \sum_{| \alpha| \leq m} I^{\alpha}_{3} {\dW} \right|
  > \frac{1}{2} \right) 
  \leq& 4
	\E \left( \int_{0}^{\tau^{T}_{k} \wedge S}   \sum_{| \alpha| \leq m} I^{\alpha}_{3} {\dW} \right)^{2} \notag\\
  \leq& C
	\E  \int_{0}^{\tau^{T}_{k} \wedge S}
	\sum_{| \alpha| \leq m}
	\sum_{l \geq 1} \left(
       \int_{\DD}   \partial^\alpha u^{k}  \cdot
                       \partial^\alpha P\sigma_l(u^{k}) |\partial^{\alpha} u^{k}|^{p -2}\,dx
       \right)^{2}
	dt \notag\\
  \leq& C
	\E  \int_{0}^{\tau^{T}_{k} \wedge S} \locCon(\|u^{k}\|_{L^\infty})^2
	(1+ \|u^{k}\|_{W^{m,p}}^{2p})
	dt  \leq C S
\label{eq:SmallGrowthAtZeroEst2Doob},
\end{align}
where again the constant $C$ is independent of $S$ and $k$.  With
\eqref{eq:SmallGrowthAtZeroSplit}--\eqref{eq:SmallGrowthAtZeroEst2Doob} we now conclude the proof of the second item in Lemma~\ref{thm:absConv}, i.e.~\eqref{eq:GrowthAsmpXmp}. 

Having finally established both \eqref{eq:strngTypeCompCondXmp} and \eqref{eq:GrowthAsmpXmp},
we apply Lemma~\ref{thm:absConv} to infer the existence of a strictly positive stopping time $\tau$,
a subsequence $\{u^{j_l}\}_{l\geq 1}$ of $\{u^j\}_{j\geq 1}$, and a predictable process $u$ such that, up to 
a set of measure zero, $u^{j_k}$ converges
to $u$ in $C(0,\tau; X_{m,p})$ and $\sup_{t \in [0,\tau]} \| u\|_{W^{m,p}} \leq C
< \infty$.  We may infer that $(u, \tau)$ is a local pathwise solution of \eqref{eq:E:1}--\eqref{eq:E:3}
in the sense of Definition~\ref{def:PathwiseSol}.  Note that, in order to initially obtain this $u$ we have had
impose the almost sure bound on the initial data, $u_0$ in \eqref{eq:UniformDataBnd}.  This restriction
is easily removed with the cutting argument as employed in Section~\ref{sec:compact} (cf. 
\eqref{eq:cutThatBitchUpSol}--\eqref{eq:cutThatBitchUpST}).   We may pass from the case of 
local to maximal pathwise solutions as given in Definition~\ref{def:MaxandGlobalSol}
via maximality arguments similar to those at the end of Section~\ref{sec:compact},  in
\eqref{eq:BlowUpAtMaxTime}. Recall that this maximality argument involves considering the set of all stopping times up to which the solution exists. We then show by contradiction that the supremum of all these stopping times yields the maximal time of existence of the solution (see Section~\ref{sec:compact} for further details). The  proof of Theorem~\ref{thm:LocalExistence} is now complete.

\section{Global existence in the two-dimensional case for additive noise}
\label{sec:2dglobal} \setcounter{equation}{0}

In this section we establish the global existence of solutions to \eqref{eq:E:1}--\eqref{eq:E:3} in dimension two forced by an additive noise.
Note that, while the local existence of solutions for  \eqref{eq:E:1}--\eqref{eq:E:3} in the case of a general $\omega$ dependent additive noise 
(cf. \eqref{eq:fuckedAdditiveNoise} above), is not covered under the proof of local existence given here, equations with additive noise 
can be treated ``pathwise'' via a simple change of variables.  In this way the local existence follows from more classical arguments.
See Remark~\ref{rmk:AddNoiseIsFucked} above and the proof of Lemma~\ref{lemma:2D:vorticity} below.

Recalling the a priori estimates in Section~\ref{sec:apriori}, we have that,
for any $m > d/p + 1$,
\begin{align}
  d \|u\|_{W^{m,p}}^{p} = X dt + Z \dW,  \label{eq:BasicWmpChallenge}
\end{align}
where $X$ and $Z$ are defined according to \eqref{eq:LpItoFormalSoln}.  Making use of the
estimates in \eqref{eq:J1EstLp}--\eqref{eq:ItoCorTemEstsLp}, we have
\begin{align}
  |X| &\leq C(1+ \|u\|_{W^{1,\infty}})\|u\|_{W^{m,p}}^{p}  + C\|\sigma\|_{\WB^{m,p}}^p,  \label{eq:UniversalBndDriftTerms}
\end{align}
for some universal constant $C = C(m,d, \DD)$.   For $Z$ we observe with similar
estimate to \eqref{eq:MFJonesConsultsForBDG} that
\begin{align}  
  \| Z \|_{L_{2}} &\leq 
  \left(\sum_{k \geq 1} \left(
       \int_{\DD}   \partial^\alpha u  \cdot
                       \partial^\alpha P \sigma_k |\partial^{\alpha} u|^{p -2}\,dx\right)^{2} \right)^{1/2}
  \leq C\|\sigma\|_{\WB^{m,p}} \| u \|_{W^{m,p}}^{p-1}. \label{eq:UniversalBndDiffTerms}
\end{align}
Thus, in view of \eqref{eq:UniversalBndDriftTerms}--\eqref{eq:UniversalBndDiffTerms}, to close the estimates for \eqref{eq:BasicWmpChallenge} we make use
of the Beale-Kato-Majda type inequality
\begin{align}
	\Vert u\Vert_{W^{1,\infty}}
	\leq C_{2} \Vert u \Vert_{L^{2}}
		+ C_{2} \Vert \curl u \Vert_{L^{\infty}} 	
			\left( 1 + \log^{+} \left( \frac{\Vert u \Vert_{W^{m,p}}}{ \Vert \curl u \Vert_{L^{\infty}}}  \right) \right),
	\label{eq:Log-Sobolev:Lp}
\end{align}
where $C_{2}$ is a universal 
constant depending only on $\DD$, $m$, $p$.  See e.g. \cite{Ferrari1993}
for the simply-connected bounded domain case.
As such the proof of global existence requires us to obtain
uniform bound on the vorticity of the solution in $L^{\infty}$ and also for $\|u\|_{L^{2}}$
and to establish a stochastic analogue of the $\log$-Gr\"{o}nwall lemma.  The latter is developed in Appendix~\ref{sec:Gronwal:SODE} below (and see also related results in \cite{FangZhang2005}).

In order to carry out suitable estimates for $w =\curl u$ we apply $\nabla^{\perp}=(\partial_{2},-\partial_{1})$ to \eqref{eq:EulerLocTmInt} and obtain the evolution:
\begin{align}
&dw + u \cdot \nabla w dt = \rho {\dW},\label{eq:vorticity}\\
&w = \nabla^{\perp} \cdot u,\ \nabla \cdot u =0,\label{eq:vorticity:2}
\end{align}
where for ease of notation we denoted $\rho =
\nabla^{\perp}\cdot\sigma$.
Note that crucially, in contrast to the three
dimensional case, no vortex stretching term $w \cdot \nabla u$  appears
in \eqref{eq:vorticity}.    For $w$ we now establish the following result:

\begin{lemma}[\bf Non-blow-up of the energy and the supremum of vorticity]\label{lemma:2D:vorticity}
Fix $m > 2/p +1$, consider any $\sigma$ that satisfies \eqref{eq:fuckedAdditiveNoise}, 
and any $u_0 \in X_{m,p}$. 
Take $(u,\xi)$ be the maximal solution corresponding to this $\sigma$ and $u_0$.
Then we have
\begin{align}
	 \sup_{t \in [0, T \wedge \xi]} \|u\|^2_{L^2}
		+ \sup_{t\in [0,T \wedge \xi]} \Vert w \Vert_{L^{\infty}}  < \infty,
		\label{eq:UniformBndVortandEngery2D}
\end{align}
almost surely,  for each $T > 0$.
\end{lemma}
\begin{proof}[Proof of Lemma~\ref{lemma:2D:vorticity}]

The bound for $\| u \|_{L^{2}}$ required in \eqref{eq:UniformBndVortandEngery2D}
follows directly in view of the cancelation $(P (u\cdot \nabla u) , u)_{L^2} =0$.
cf. Section~\ref{sec:aprioriL2}.

We turn to estimate the vorticity term \eqref{eq:UniformBndVortandEngery2D}.
Since \eqref{eq:vorticity} is forced with an the additive noise we have the option to introduce the stochastic process
\begin{align}
dz = \rho\, {\dW}, \qquad z(0)=0\label{eq:z:def}
\end{align}
and then consider the evolution of $\tilde{w} := w-z$. The equation for $\tilde{w}$ is the random
partial differential equation 
\begin{align}
&\partial_{t} \tilde{w} + u \cdot \nabla \tilde{w} + u\cdot\nabla z = 0 \label{eq:tilde:w:1}\\
&\tilde{w} = \nabla^{\perp} \cdot u - z,\ \nabla \cdot u =0, \label{eq:tilde:w:2}\\
&\tilde{w}(0) = w_{0}. \label{eq:tilde:w:3}
\end{align}
This system can be treated pathwise with the methods of ordinary calculus.
Multiplying \eqref{eq:tilde:w:1} by $\tilde{w} |\tilde{w}|^{p-2}$ and integrating over $\DD$ we obtain
\begin{align*}
\frac{d}{dt} \Vert \tilde{w} \Vert_{L^{p}} \leq \Vert u \Vert_{L^{p}} \Vert \nabla z \Vert_{L^{\infty}}
\end{align*}
where we have used the divergence-free nature of $u$. Integrating in time and
sending $p$ to $\infty$, the above estimate gives
\begin{align}
\Vert \tilde{w}(t) \Vert_{L^{\infty}} &\leq \Vert w_{0} \Vert_{L^{\infty}} + \int_{0}^{t} \Vert u(s) \Vert_{L^{\infty}} \Vert \nabla z(s) \Vert_{L^{\infty}}\, ds.
	\label{eq:shiftwBndSetsAChallenge}
\end{align}
We can use the two-dimensional
Sobolev embedding and the Biot-Savart law to bound
\begin{align}
\Vert u \Vert_{L^{\infty}} \leq C \Vert \nabla u \Vert_{L^{4} } + C
\Vert u \Vert_{L^{2}} \leq C \Vert w \Vert_{L^{4}} + C \Vert u
\Vert_{L^{2}},\label{eq:2D:u:infty}
\end{align}
where $C = C (\DD)$. Thus, in view of \eqref{eq:shiftwBndSetsAChallenge}--\eqref{eq:2D:u:infty}
and the fact that $w = \tilde{w} + z$, the proof will be complete once we obtain suitable bounds for
the quantities $\|w\|_{L^4}$ and $\|\nabla z\|_{L^{\infty}}$.

In order to obtain bounds on $\| w \|_{L^{4}}$
we apply the It\={o} formula in $L^{4}$ to \eqref{eq:vorticity} an obtain
\begin{align}
	d \|w\|^{4}_{L^{4}} =&
		\int_{\DD} \left(  2  |w|^{2} \sum_{k \geq 1} |\rho_{k}|^{2}
			+ 4 \sum_{k \geq 1} (w  \rho_{k})^{2}
						\right) dx dt
		+ 4  \sum_{k \geq 1} \left( \int_{\DD} |w|^{2} w  \rho_{k} dx \right) dW_{k}, \label{eq:AdditiveFIXFIX}
\end{align}
where we have used the cancelation $(u\cdot \nabla w, w |w|^{2})_{L^2} = 0$.
Let 
\begin{align}
\sigma_R = \inf\left\{t \geq 0 \colon \|w(t)\|_{L^4} > R \right\} \wedge  \inf\left\{t \geq 0 \colon \int_0^t \|\rho\|_{\WB^{0,4}}^2 ds > R \right\} \wedge \xi. \label{eq:SIGMARFIXFIX}
\end{align}
From \eqref{eq:fuckedAdditiveNoise} and the definition of $\xi$ as the maximal time of existence, it follows that $\sigma_{R} \to \xi$ almost surely as $R \to \infty$. In addition, for every $T>0$ and a.s. $\omega$, if $R$ is sufficiently large we have that $\sigma_{R}\wedge T = \xi \wedge T$.

Upon taking a supremum in time in \eqref{eq:AdditiveFIXFIX}, and applying the H\"older
inequality in the last term, we obtain on the set $\{ \sigma_{R} >0 \}$
\begin{align*}
	&\sup_{t\in[0,\sigma_R \wedge T]} \| w(t) \|_{L^{4}}^{4}\\
	&\ \leq \| w_{0} \|_{L^{4}}^{4}
               + 4 \sup_{t\in[0,\sigma_R \wedge T]} \left| \sum_{k\geq 1}
               \int_{0}^{t} \int_{\DD} |w|^{2} w \cdot \rho_{k}\, dx dW_{k}\right|
	+ 4
	\int_{0}^{\sigma_R \wedge T}
	\| w(t) \|_{L^{4}}^{2} \|\rho \|_{\WB^{0,4}}^2 dt\\
	&\ \leq \| w_{0} \|_{L^{4}}^{4}
               + 4 \sup_{t\in[0,\sigma_R \wedge T]} \left| \sum_{k\geq 1}
               \int_{0}^{t} \int_{\DD} |w|^{2} w \cdot \rho_{k}\, dx dW_{k}\right|
	+ \frac{1}{4} \sup_{t \in [0, \sigma_R \wedge T]} \| w(t) \|_{L^{4}}^{4}
	+ C \left(\int_{0}^{\sigma_R \wedge T}
	 \|\rho \|_{\WB^{0,4}}^2 dt \right)^{2} .
\end{align*}
To estimate the stochastic integral terms we find with the Burkholder-Davis-Gundy
inequality, \eqref{eq:BDG} that
\begin{align*}
    &\E \sup_{t\in[0,\sigma_R \wedge T]} \left| \indFn{\sigma_{R}>0} \sum_{k\geq 1}
               \int_{0}^{t} \int_{\DD} |w|^{2} w \cdot \rho_{k}\, dx dW_{k}\right|\notag\\
   &\qquad \leq
   C \E \left(\indFn{\sigma_{R}>0}
               \int_{0}^{\sigma_R \wedge T} \sum_{k\geq 1}  \left( \int_{\DD} |w|^{3} |\rho_{k}| \, dx \right)^2 dt \right)^{1/2}\\
               & \qquad \leq
   C \E \left(\indFn{\sigma_{R}>0}
               \int_{0}^{\sigma_R \wedge T}  \left( \int_{\DD} |w|^{3}   \left( \sum_{k\geq 1} |\rho_{k}|^2 \right)^{1/2} \, dx \right)^2 dt \right)^{1/2}\\
               &\qquad \leq
   C \E \left(\indFn{\sigma_{R}>0}
               \int_{0}^{\sigma_R \wedge T} \|w\|_{L^4}^{6} \|\rho \|_{\WB^{0,4}}^2 dt \right)^{1/2}\\
                              & \qquad \leq
   \frac{1}{4}  \E \left( \indFn{\sigma_{R}>0}\sup_{t \in [0, \sigma_R \wedge T]} \|w\|_{L^4}^{4} \right)+
             C \E  \left( \indFn{\sigma_{R}>0} \int_{0}^{\sigma_R \wedge T}  \|\rho \|_{\WB^{0,4}}^2 dt \right)^{2}.
\end{align*}
Combining the above observations we find
$
	\E ( \indFn{\sigma_{R}>0}\sup_{t \in [0, \sigma_R \wedge T]} \|w\|_{L^4}^{4} ) \leq C,
$
by recalling the definition of $\sigma_{R}$ (cf.~\eqref{eq:SIGMARFIXFIX}),
for some $C>0$ which depends on $R$. Since $\| w_{0} \|_{L^{4}} < \infty$ almost surely we conclude that $\sup_{t \in [0, \sigma_R \wedge T]} \|w\|_{L^4}^{4} < \infty$ almost surely for all $R>0$. Thus we finally conclude that for almost every $\omega$ that 
\begin{align}
\sup_{t \in [0, \xi \wedge T]} \|w\|_{L^4}^{4} < \infty.\label{eq:L4:ClusterFuck}
\end{align}

We now turn to make estimates for $z$.
In view of the Sobolev embedding $W^{1,\infty} \subset W^{m,p}$
and the definition of $z$, given in \eqref{eq:z:def}, we estimate
using \eqref{eq:BDG}
\begin{align*}
\E \sup_{t\in[0,T]} \left\Vert \int_{0}^{t} \rho {\dW} \right\Vert_{W^{m,p}}^p
 \leq& \sum_{|\alpha| \leq m} \int_{\DD} \E \sup_{t \in [0,T]}  \left| \int_0^t \partial^{\alpha} \rho {\dW} \right|^p dx\\
\leq& C\sum_{|\alpha| \leq m} \int_{\DD} \E   \left( \int_0^T |\partial^{\alpha} \rho|_{L_2}^2 dt \right)^{p/2} dx
\leq C \E \int_0^T \|\rho\|_{\WB^{m,p}}^p dt.
\end{align*}
We therefore infer that
\begin{align}
	\E \sup_{t\in[0,T]} \Vert z(t) \Vert_{W^{1,\infty}}^2
	\leq C\left( \E \sup_{t\in[0,T]} \left\Vert \int_{0}^{t} \rho {\dW} \right\Vert_{W^{m,p}}^p \right)^{p/2}
	<\infty. \label{eq:bndForRandomShiftProcess}
\end{align}

Taking the supremum in time over $[0,T \wedge \xi]$ for \eqref{eq:shiftwBndSetsAChallenge}, and applying \eqref{eq:2D:u:infty}, we obtain for almost every $\omega$ that 
\begin{align}
&\sup_{t\in[0,T \wedge \xi]} \Vert \tilde{w}(t) \Vert_{L^{\infty}} \notag\\
& \quad \leq  \Vert w_{0} \Vert_{L^{\infty}} + C\, \left( \sup_{t\in[0,T \wedge \xi]}\Vert u(t) \Vert_{L^{2}} \int_{0}^{T} \Vert \nabla z(t) \Vert_{L^\infty} \, dt \right)
+ C\, \left( \sup_{t\in[0,T \wedge \xi]}\Vert w(t) \Vert_{L^{4}} \int_{0}^{T} \Vert \nabla z(t) \Vert_{L^\infty} \, dt \right)\notag\\
& \quad \leq \Vert w_{0} \Vert_{L^{\infty}}
	+ C\, \left(\sup_{t\in[0,T \wedge \xi]} \Vert u(t) \Vert_{L^{2}}^{2}
		+  \sup_{t\in[0,T \wedge \xi]}\Vert w(t) \Vert_{L^{4}}^{2}
		+ \sup_{t\in[0,T ]} \Vert z(t) \Vert_{W^{1,\infty}}^2 \right),
 \label{eq:before:z:bound}
 \end{align}
where $C$ may depend on $T$.
Given the bounds established in \eqref{eq:L4:ClusterFuck}--\eqref{eq:bndForRandomShiftProcess},
and since by construction $w = \tilde{w} + z$, referring once more to
\eqref{eq:bndForRandomShiftProcess}, the proof of the lemma is now complete.
\end{proof}

With the estimates in Lemma~\ref{lemma:2D:vorticity} in hand we apply the results established
in Appendix~\ref{sec:Gronwal:SODE} below, to show that $(u,\xi)$ is a \emph{global} pathwise solution.

\begin{proof}[Proof of Theorem~\ref{thm:GlobalExistence2DAdditive}]  
We need to verify that the conditions in Lemma~\ref{thm:NonBlowupLogLipp}
are satisfied.  In what follows we will assume, without loss of generality that $\|u_{0}\|_{W^{m,p}} \leq M$,
for some deterministic constant $M>0$.  Indeed, after we obtain global existence in this
special case, the general case, $u_{0}\in X_{m,p}$ a.s, follows from a cutting
argument as in Section~\ref{sec:compact}, see \eqref{eq:cutThatBitchUpSol}--\eqref{eq:cutThatBitchUpST}.

Define the collection of stopping times
\begin{align}
	\tau_R
		:= \inf \left\{ t \geq 0 \colon \|u(t)\|^2_{L^2} + \|w(t)\|_{L^\infty} > R \right\} \wedge \xi,
		\label{eq:ControlingTimesforVortEnergy}
\end{align}
where we recall that $w = \mbox{curl}u$.  Obviously, $\tau_{R}$ is increasing 
in $R$, almost surely.
We need to verify that \eqref{eq:WeHitXiWithFiniteR} is satisfied.  In other words,
we need to show
\begin{align}
\Prb \left( \bigcap_R \{  \tau_R  < T \wedge 
\xi \}  \right) =0,
\label{eq:Epsiode4}
\end{align}
for every $T > 0$.  For this purpose we make use of the 
conclusions of Lemma~\ref{lemma:2D:vorticity}. 
Owing to the fact that $\tau_R$ is increasing in $R$
and \eqref{eq:UniformBndVortandEngery2D} we infer
\begin{align*}
\Prb \left( \bigcap_{R>0} \{  \tau_{R}  < T \wedge \xi \} \right)
   &= \lim_{R^* \rightarrow \infty}   \Prb \left( \bigcap_{0 < R \leq R*} \{  \tau_R  < T \wedge \xi \} \right)
   =  \lim_{R^* \rightarrow \infty}   \Prb \left(  \tau_{R^*}  < T \wedge \xi  \right)\\
  &\leq  \lim_{R^* \rightarrow \infty}
   \Prb \left( \sup_{ t \in [0, T \wedge \xi]}  \left(\|u\|^2_{L^2} + \|w\|_{L^\infty}\right)  > R^*  \right)\\
  &\leq    \Prb \left( \bigcap_{R^{*}> 0} \left\{\sup_{ t \in [0, T \wedge \xi]}  \left(\|u\|^2_{L^2} + \|w\|_{L^\infty}\right)  > R^*  \right\}\right)= 0,
\end{align*}
for every $T > 0$.

Returning to the a priori estimates \eqref{eq:BasicWmpChallenge}--\eqref{eq:UniversalBndDiffTerms} we 
now define the quantities
\begin{align}
	Y = 1 + \|u\|^{p}_{W^{m,p}}, \quad
	\eta = (1+ \|\sigma\|_{\WB^{m,p}})^p.
	\label{eq:MainProcessforLogLemmaDef}
\end{align}
Of course, $Y$ satisfies $dY = X dt + Z \dW$. Combining \eqref{eq:UniversalBndDriftTerms},
\eqref{eq:Log-Sobolev:Lp},  and the definition of $\tau_{R}$,
we find that for each $R$ there exists a deterministic constant $K_R$ such that on $[0,\tau_{R}]$ we have
\begin{align}
  |X| 
  &\leq C \left(1+  \Vert u \Vert_{L^{2}}
		+ \Vert w \Vert_{L^{\infty}} 	
			\left( 1 + \log^{+} \left( \frac{\Vert u \Vert_{W^{m,p}}}{ \Vert w \Vert_{L^{\infty}}}  \right) \right) \right)
			\|u\|_{W^{m,p}}^{p}  +  C\|\sigma\|_{\WB^{m,p}}^p \notag \\
 &\leq C \left(2 +  R^{1/2} + R 
		+ \Vert w \Vert_{L^{\infty}} 	
			\log^{+} \Vert u \Vert_{W^{m,p}} \right)
			Y  + C\|\sigma\|_{\WB^{m,p}}^p \notag \\
  &\leq K_R (1+ \log Y ) Y + C (1+ \|\sigma\|_{\WB^{m,p}})^p , \label{eq:Epsiode5}
\end{align}
and from \eqref{eq:UniversalBndDiffTerms} we in addition obtain
\begin{align}
  \| Z \|_{L_{2}} &\leq C\|\sigma\|_{\WB^{m,p}} \| u \|_{W^{m,p}}^{p-1} \leq C (1+\|\sigma\|_{\WB^{m,p}}) Y^{(p-1)/p} .
  \label{eq:Epsiode6}
\end{align}

We now have all the ingredients need to apply Lemma~\ref{thm:NonBlowupLogLipp}.  
More precisely we take $Y$ and $\eta$ according to \eqref{eq:MainProcessforLogLemmaDef}, 
$r = 1/p$,  $\xi$ as the maximal time of existence of $u$
and  $\tau_R$ 
according to \eqref{eq:ControlingTimesforVortEnergy}.  Having established
\eqref{eq:Epsiode4}--\eqref{eq:Epsiode6} and recalling the 
standing assumption \eqref{eq:fuckedAdditiveNoise} we infer from Lemma~\ref{thm:NonBlowupLogLipp}
that indeed $\xi = \infty$.  The proof of Theorem~\ref{thm:GlobalExistence2DAdditive} is therefore
complete.
\end{proof}

\section{Global existence for linear multiplicative noise}
\label{sec:GlobalExistenceLinMultNoise}
\setcounter{equation}{0}

In this section we consider the stochastic Euler equations in two and three dimensions, with {\em linear multiplicative noise}
\begin{align}
  &du + P\left( u \cdot \nabla u \right) dt =  \alpha u dW,\label{eq:ELinM}
\end{align}
where in this case $\alpha \in \RR$ and $W$ is a {\em single $1D$ Brownian motion}.
This forcing regime is covered under the theory developed in the previous sections, so we are guaranteed the
existence of a local pathwise solution in the sense of Definition~\ref{def:PathwiseSol} (cf.~Theorem~\ref{thm:LocalExistence}).

As in the case of an additive noise above we may transform \eqref{eq:ELinM} to an random
PDE.  To this end consider the (real valued) stochastic process
\begin{align}
\gamma(t) = e^{-\alpha W_{t}}. \label{eq:gamma:def}
\end{align}
Due to the It\={o} formula we find the $\gamma$ satisfies
$$
 d \gamma = - \alpha \gamma dW + \frac{1}{2} \alpha^{2}\gamma dt, \quad \gamma(0)=1.
$$
By apply the It\={o} product rule we therefore find that
\begin{align}
    d(\gamma u) &=\gamma du +  u d\gamma  + d\gamma d u\notag\\
    &= -  \gamma P\left( u \cdot \nabla u \right) dt  +
        \alpha \gamma u dW   -  \alpha \gamma u dW
        +  \frac{1}{2} \alpha^{2}\gamma u dt
        - \alpha^{2} \gamma u dt \notag \\
     &= -  \gamma P\left( u \cdot \nabla u \right) dt  - \frac{1}{2} \alpha^{2} (\gamma u) dt. \label{eq:GammaComp}
\end{align}
By defining $v = \gamma u$ we therefore obtain the system
\begin{align}
  & \partial_{t} v +  \frac{\alpha^{2}}{2} v  + \gamma^{-1} P( v \cdot \nabla v) = 0, \label{eq:Notes:21}\\
  & v(0) = u_{0}. \label{eq:Notes:21b}
\end{align}
Fix $p \geq 2$, and  $m > d/p + 1$ throughout the rest of this section. First, using the standard
 estimates on the nonlinear term (cf.~\eqref{eq:J1} for
 $p=2$, or \eqref{eq:J1EstLp} for $p>2$), we may obtain
\begin{align}
	\frac{d}{dt} \Vert v \Vert_{W^{m,p}} + \frac{\alpha^{2}}{2} \Vert v \Vert_{W^{m,p}}
	\leq C_{1} \gamma^{-1} \Vert v \Vert_{W^{1,\infty}} \Vert v \Vert_{W^{m,p}}
	\label{eq:Notes:22}
\end{align}
for a positive constant $C_{1} = C_{1}(m,p,\DD)$.
In order to bound the right side of \eqref{eq:Notes:22} we recall
the Beale-Kato-Majda-type inequality (cf.~\eqref{eq:Log-Sobolev:Lp})
\begin{align}
	\Vert v \Vert_{W^{1,\infty}}
	\leq C_{2} \Vert v \Vert_{L^{2}}
		+ C_{2} \Vert w \Vert_{L^{\infty}} 	
			\left( 1 + \log^{+} \left( \frac{\Vert v \Vert_{W^{m,p}}}{ \Vert w \Vert_{L^{\infty}}}  \right) \right)
	\label{eq:Notes:25}
\end{align}
where the constant $C_{2} = C_{2}(m,p,\DD)$ is fixed, and as usual $w = \curl v$. Due to the cancellation property $( P (v \cdot \nabla v), v) =0$,  it follows directly from \eqref{eq:Notes:21} that
\begin{align}
\Vert v(t) \Vert_{L^{2}} \leq \Vert v_{0} \Vert_{L^{2}} e^{-\alpha^{2}t /2}
\label{eq:ExpoL2}
\end{align}
for all $t\geq 0$. On the other hand, obtaining an a priori estimate on $\Vert w (t) \Vert_{L^{\infty}}$ is more delicate.
For this purpose, we return to  \eqref{eq:Notes:21} and consider the equation satisfied by $w = \curl v$, i.e.
\begin{align}
	\partial_{t} w +\frac{\alpha^{2}}{2} w + \gamma^{-1} v \cdot \nabla w
	=
	\begin{cases}
	0, & \textrm{ for } d = 2,\\
	\gamma^{-1} w \cdot \nabla v, & \textrm{ for } d = 3.
	\end{cases}
	\label{eq:Notes:23}
\end{align}
Multiplying \eqref{eq:Notes:23} by $w |w|^{p-2}$, integrating in $x$, and making use of the divergence-free nature of $v$,  we obtain
\begin{align*}
	\frac{1}{p} \frac{d}{dt} \Vert w \Vert_{L^{p}}^{p} + \frac{\alpha^{2}}{2} \Vert w \Vert_{L^{p}}^{p}
	\leq
	\begin{cases}
	0, & \textrm{ for } d = 2,\\
	\gamma^{-1} \Vert v \Vert_{W^{1,\infty}} \Vert w \Vert_{L^{p}}^{p},  & \textrm{ for } d = 3.
	\end{cases}
\end{align*}
Upon canceling $\|w\|_{L^{p}}^{p-1}$, and sending $p\rightarrow \infty$ in the above estimate we have
\begin{align}
	\frac{d}{dt} \Vert w \Vert_{L^{\infty}} + \frac{\alpha^{2}}{2} \Vert w \Vert_{L^{\infty}}
	\leq
	\begin{cases}
	0, & \textrm{ for } d = 2,\\
	\gamma^{-1} \Vert v \Vert_{W^{1,\infty}} \Vert w \Vert_{L^{\infty}},  & \textrm{ for } d = 3.
	\end{cases}
	 \label{eq:Notes:24}
\end{align}
In view of the different bounds obtained in \eqref{eq:Notes:24} in $2D$ versus $3D$, we now treat the two cases separately.
For this purpose it is convenient to first fix the Sobolev embedding constant $C_{3} = C_{3}(m,p,\DD)$ such that
\begin{align}
\Vert v \Vert_{L^{2}} + \Vert v \Vert_{W^{1,\infty}} \leq C_{3} \Vert v \Vert_{W^{m,p}}
\label{eq:fix:Sobolev:constant}
\end{align}
and to let $\bar C = C_{1} C_{2} + C_{3} + 1$.

\subsection{The two-dimensional case}\label{sec:GlobalExistence:2D}
In two dimensions we prove the global in time existence of smooth pathwise solutions, as stated in Theorem~\ref{thm:GlobalExistence2D3DLinearMultiplicative}.
From \eqref{eq:Notes:24} we immediately obtain that the function
$$
z(t) = \Vert w(t) \Vert_{L^{\infty}} \exp\left( \frac{\alpha^{2} t}{2} \right)
$$
is such that
\begin{align}
z(t)\leq z(0) = \Vert w_{0} \Vert_{L^{\infty}} \label{eq:2D:Linfty:exp:decay}
\end{align}
for all $t \geq 0$. Therefore, letting
$$
y(t ) = \Vert v(t) \Vert_{W^{m,p}} \exp\left( \frac{\alpha^{2} t}{2} \right)
$$
we obtain from \eqref{eq:Notes:22}--\eqref{eq:ExpoL2}, and \eqref{eq:2D:Linfty:exp:decay} that
\begin{align}
\frac{d y}{dt}
&\leq \bar C \gamma^{-1} y \left( \| v(t)\|_{L^{2}} + \|w(t)\|_{L^{\infty}} 
	\left( 1 + \log^{+} \left( \frac{y(t)}{\|w(t)\|_{L^{\infty}} \exp(\alpha^{2}t/2)}\right) \right)\right) \notag\\
&\leq \bar C \gamma^{-1} \exp\left(- \frac{\alpha^{2} t}{2} \right) y
	 \left( \Vert v_{0} \Vert_{L^{2}} + \Vert w_{0} \Vert_{L^{\infty}} + z \log^{+} \left( \frac{y}{z}\right) \right). \label{eq:2d:global:ODE}
\end{align}
A short computation reveals that $z \log^+ (y/z) \leq 1/e + z \log^+(y)$. 
In view of \eqref{eq:2D:Linfty:exp:decay}, and defining 
$
\rho_\alpha(t) = \exp\left( \alpha W_{t} - \alpha^{2} t / 2\right)
$
estimate \eqref{eq:2d:global:ODE} gives
\begin{align}
\frac{d y}{dt}&\leq \bar C \rho_\alpha y
	 \left( \Vert v_{0} \Vert_{L^{2}} + \Vert w_{0} \Vert_{L^{\infty}} + 1 + \Vert w_{0} \Vert_{L^{\infty}} \log^{+} (y) \right).
	 \label{eq:2d:global:ODE:1}
\end{align}
By the law of iterated logarithms we have
$
	\sup_{t\geq 0} \rho_{\alpha} < \infty \mbox{ a.s.}
$
for every $\alpha>0$. Hence, \eqref{eq:2d:global:ODE:1} implies
\begin{align}
\frac{dy}{dt} \leq  A  y
	 \left(  1 +  \log^{+} (y) \right). \label{eq:2d:global:ODE:2}
\end{align}
where
\begin{align}
A  = \bar C \left( \sup_{t\geq 0} \rho_{\alpha} \right) ( \Vert v_{0} \Vert_{L^{2}} + \Vert w_{0} \Vert_{L^{\infty}} + 1 ). \label{eq:A:def:def}
\end{align}
Let $Y(t) = \log (1 +y(t))$. We obtain from \eqref{eq:2d:global:ODE:2} that 
\begin{align*}
\frac{dY}{dt}\leq A \left( 1 + Y(t) \right)
\end{align*}
for all $t\geq 0$. This gives $Y(t) \leq Y(0) \exp(tA) + t A \exp(t A)$, and hence
\begin{align}
y(t) \leq (1+y_{0})^{\exp\left(t A \right) }
\exp\left( tA  \exp\left(t A\right)  \right) \label{eq:app:CC:2}.
\end{align}
 
Recalling the definition of $y(t)$, we note that $\| u(t)\|_{W^{m,p}} = \gamma^{-1}(t) y(t) \exp(-\alpha^2 t/2) = \rho_\alpha(t) y(t)$. 
Thus, estimate \eqref{eq:app:CC:2} shows that
\begin{align*}
 \| u(t)\|_{W^{m,p}} \leq \rho_\alpha(t)  (1+\| u_{0} \|_{W^{m,p}})^{\exp\left(t A \right) }
\exp\left( tA  \exp\left(t A\right)  \right)
\end{align*}
with $A$ as defined in \eqref{eq:A:def:def}.
Therefore, for all $T>0$ we have proven 
\begin{align*}
 \sup_{t \in [0,T\wedge \xi]} \|u\|_{W^{1,\infty}} < \infty, \mbox{ a.s.}
\end{align*}
So that necessarily $(u, \xi)$ is a global pathwise solution, i.e. we have $\xi = \infty$ (cf.~Definition~\ref{def:MaxandGlobalSol}).  
We have thus now established part (i) of Theorem~\ref{thm:GlobalExistence2D3DLinearMultiplicative}.

\subsection{The three-dimensional case}\label{sec:GlobalExistence:3D}

Fix $\alpha>0$.   
Let $(u, \xi)$ be the maximal strong
solution of \eqref{eq:ELinM}.
As in the two-dimensional case, the key ingredient to global regularity is an a priori bound on $\| w\|_{L^\infty}$.
However, due to the presence of the vortex stretching term, in the three-dimensional case we have (cf.~\eqref{eq:Notes:24} above)
\begin{align}
 \frac{d}{dt} \Vert w \Vert_{L^{\infty}} + \frac{\alpha^{2}}{2} \Vert w \Vert_{L^{\infty}}
	\leq \gamma^{-1} \Vert v \Vert_{W^{1,\infty}} \Vert w \Vert_{L^{\infty}}.
	\label{eq:3D:vorticity:bound}
\end{align}
To exploit the damping in \eqref{eq:3D:vorticity:bound}, we now define the stopping time
\begin{align}
\sigma = \inf_{t \geq 0} \left\{ t \colon \gamma^{-1}(t) \Vert v(t) \Vert_{W^{m,p}} \geq \frac{\alpha^{2}}{4 \bar C } \right\}
       = \inf_{t \geq 0} \left\{ t \colon \Vert u(t) \Vert_{W^{m,p}} \geq \frac{\alpha^{2}}{4 \bar C } \right\}
\label{eq:Notes:28}
\end{align}
where $\bar C \geq 1$ is the constant defined above \eqref{eq:fix:Sobolev:constant}. Note that $\sigma < \xi$ on the set $\{ \xi <\infty\}$ (cf.~\eqref{eq:W1inftyBlowUp} and the Sobolev embedding). In order to ensure that $\sigma>0$ a.s. we will at least need to impose the condition 
\begin{align}
\| u_0\|_{W^{m,p}} < \frac{\alpha^2}{4 \bar C}. \label{eq:IC:COND:1}
\end{align}
In fact, in order to close the estimates we shall impose additional assumptions on $u_0$ (cf.~\eqref{eq:THE:IC:CONDITION} below).

Due to the Sobolev embedding, on $[0,\sigma]$ we have
\begin{align}
\gamma^{-1} \Vert w \Vert_{L^{\infty}}
\leq \gamma^{-1} \Vert v \Vert_{W^{1,\infty}}
\leq \frac{\alpha^{2}}{4}. \label{eq:Notes:29}
\end{align}
Hence, by \eqref{eq:3D:vorticity:bound} and \eqref{eq:Notes:29} we obtain
\begin{align}
\frac{d}{dt} \Vert w \Vert_{L^{\infty}} + \frac{\alpha^{2}}{4} \Vert w \Vert_{L^{\infty}} \leq 0
\label{eq:Notes:30}
\end{align}
on $[0,\sigma)$. Therefore, letting
$$
z(t) = \Vert w (t) \Vert_{L^{\infty}}\exp \left( \frac{\alpha^{2} t}{4} \right)
$$
we find from \eqref{eq:Notes:29} and \eqref{eq:Notes:30} that
\begin{align}
z(t) \leq z(0) = \Vert w_{0} \Vert_{L^{\infty}} \leq \frac{\alpha^{2}}{4} \label{eq:Notes:31}
\end{align}
where we also used that $\gamma(0)=1$. Similarly to above, we now let
\begin{align}
y(t) = \Vert v(t) \Vert_{W^{m,p}} \exp \left( \frac{\alpha^{2} t}{4} \right).
\label{eq:yDef3dCase}
\end{align}
By \eqref{eq:Notes:22} and \eqref{eq:Notes:25} we obtain
\begin{align*}
\frac{dy}{dt}
&\leq \bar C \gamma^{-1} y \left( \|v \|_{L^2} + \| w\|_{L^\infty} \left(1 + \log^+ \left( \frac{y}{z} \right) \right) \right).
\end{align*}
Using the decay of $\|v(t)\|_{L^2}$ obtained in \eqref{eq:ExpoL2}, and assumption \eqref{eq:IC:COND:1}, the above estimate implies
\begin{align}
\frac{dy}{dt} 
&\leq \bar C  \gamma^{-1} \exp\left( - \frac{\alpha^{2} t}{4} \right) y
	\left( \| u_0\|_{L^2} + z \left(1 + \log^+ \left( \frac{y}{z} \right) \right) \right) \notag\\
&\leq \bar C  \rho_\alpha \exp\left( - \frac{\alpha^{2} t}{8} \right) y
	\left( \frac{\alpha^{2}}{4} + z  + z \log^{+} \left( \frac{y}{z} \right) \right)
\label{eq:Notes:32}
\end{align}
where we now denote
\begin{align}
\rho_\alpha(t) =  \gamma^{-1}(t) \exp\left( - \frac{\alpha^{2} t}{8} \right) = \exp\left( \alpha W_t - \frac{\alpha^{2} t}{8} \right)
\label{eq:GBM3Dcase}.
\end{align}
To simplify the right side of \eqref{eq:Notes:32}, it is convenient to observe that
\begin{align}
\frac{\alpha^{2}}{4}  + z  + z \log^{+} \left( \frac{y}{z} \right) \leq \bar C  + \alpha^{2} + z \log y \label{eq:Notes:13}
\end{align}
holds whenever $0 < z \leq \alpha^{2}/4$, and $z \leq \bar C  y$ (note that we indeed have these a priori bounds on $z$, due to \eqref{eq:fix:Sobolev:constant} and \eqref{eq:Notes:31}). In order to prove \eqref{eq:Notes:13} we distinguish two cases: $z < y$, and $z/\bar C \leq y \leq z$.
If $z<y$, then $\log^{+}(y/z) = \log(y/z) = \log(y) - \log(z)$. Hence the left side of \eqref{eq:Notes:13} is bounded by
$$
 \frac{\alpha^{2}}{2} + z \log(y) - \indFn{z \in (0,1]}z \log(z) \leq \alpha^{2} + z \log y + \bar C
$$
where we have used the fact that  $0 \leq - z \log(z)  \leq  1/e \leq \bar C$ for all $z\in (0,1]$. This concludes the proof of \eqref{eq:Notes:13} for $y > z$. On the other hand, if $y \leq z$, then $\log^{+}(y/z) = 0$, and hence we need to prove that $\alpha^{2}/4 + z$ is less than the left side of \eqref{eq:Notes:13}.
For this purpose, it is sufficient to prove that
$$
 \bar C + z \log y \geq 0,
$$
for all $y \in [z/ \bar{C}, z]$ and all $z >0$. Indeed, the right side of the above inequality is monotone
increasing in $y$, so the minimum is attained at $y = z/\bar C$,
and it equals $ \bar C + z \log (z/\bar C)$. A simple calculation shows that
$\bar C + z \log(z/\bar C) \geq \bar C - \bar C/e > 0$, for all $z \geq 0$, concluding the proof of \eqref{eq:Notes:13}.

Therefore, by \eqref{eq:Notes:32} and \eqref{eq:Notes:13} we have
\begin{align}
\frac{dy}{dt}
\leq \bar C  \rho_\alpha \exp\left( - \frac{\alpha^{2} t}{8} \right) y
	\left(\bar C  + \alpha^{2} + z \log y \right). \label{eq:ODE:before:tau:R}
\end{align}	
Fix any $R\geq 1$ and define the stopping time
\begin{align}
 \tau_R = \inf \left\{t\geq 0 \colon \rho_\alpha(t) \geq R\right\}.
 \label{eq:KRAFTWERK}
\end{align}
From \eqref{eq:ODE:before:tau:R} we obtain the bound
\begin{align}
\frac{dy}{dt} \leq
\bar C  R \exp\left( - \frac{\alpha^{2} t}{8} \right) y
	\left(\bar C  + \alpha^{2} + z \log y \right)
\label{eq:THE:ODE}
\end{align}
for all $t \in [0,\tau_R \wedge \sigma]$. We now may apply Lemma~\ref{lemma:ODE:log}, which is a suitable version of the logarithmic Gr\"onwall inequality. Lemma~\ref{lemma:ODE:log} guarantees the existence of a positive deterministic function $\kappa(R,\alpha)$ with the properties
\begin{align*}
\kappa(R,\alpha) \leq \frac{\alpha^2}{8 \bar C}&, \quad \mbox{for every } R \geq 1\\
\lim_{R\to \infty} \kappa(R,\alpha) = 0&, \quad \mbox{for every fixed } \alpha \neq 0\\
\lim_{\alpha^2 \to \infty} \kappa(R,\alpha)  = \infty&, \quad \mbox{for every fixed } R \geq 1\\
\lim_{\alpha^2 \to 0} \kappa(R,\alpha)  = 0&, \quad \mbox{for every fixed } R \geq 1
\end{align*}
 such that if the initial data satifies
\begin{align} 
\| u_0\|_{W^{m,p}} = y(0) \leq \kappa(R,\alpha) \label{eq:THE:IC:CONDITION}
\end{align}
then a smooth solution of \eqref{eq:THE:ODE} satisfies
\begin{align}
 y(t) \leq \frac{\alpha^2}{8 R \bar C} \label{eq:THE:CLOSURE:BOUND}
\end{align}
for all $t\in [0,\tau_R\wedge \sigma]$. For clarity of the presentation we postpone the precise formula for the function $\kappa(R,\alpha)$ and the proof that \eqref{eq:THE:IC:CONDITION} implies \eqref{eq:THE:CLOSURE:BOUND} to Appendix~\ref{sec:ODE} below.

Note that the condition \eqref{eq:THE:IC:CONDITION} imposed on the initial data automatically implies 
\eqref{eq:IC:COND:1}, and hence $\sigma >0$.
Recalling the definition of $y(t)$ and $\rho_\alpha(t)$ in \eqref{eq:yDef3dCase} and \eqref{eq:GBM3Dcase}
 we obtain from \eqref{eq:THE:CLOSURE:BOUND} that for every $t$ in the interval $[0,\sigma\wedge \tau_R]$
\begin{align}
\| u(t)\|_{W^{m,p}} = \gamma^{-1}(t) \Vert v(t) \Vert_{W^{m,p}} = \exp\left( - \frac{\alpha^2t}{8} \right) \rho_\alpha(t) y(t) \leq R \frac{\alpha^{2}}{8 R \bar C } = \frac{\alpha^{2}}{8 \bar C }.
\label{eq:THE:CLOSURE:BOUND:a}
\end{align}
Hence, due to the definition of $\sigma$ (cf.~\eqref{eq:Notes:28}),
the bound \eqref{eq:THE:CLOSURE:BOUND:a} shows that $\sigma \wedge \tau_{R} = \tau_R$. Therefore
\begin{align*}
  \sup_{t \in [0, \tau_{R}]} \| u(t) \|_{W^{1,\infty}} \leq C_{3} \sup_{t \in [0,\tau_{R}]} \Vert u(t) \Vert_{W^{m,p}} \leq \frac{\alpha^{2}}{8},
\end{align*}
which implies that $\xi \geq \tau_{R}$. Therefore, the maximal pathwise solution $(u,\xi)$ of \eqref{eq:ELinM} is global in time on the set $\{\tau_{R} = \infty\}$, 
i.e. on the set where $\rho_{\alpha}(t)$ always stays below $R$ (cf.~\eqref{eq:KRAFTWERK}). We now claim that
\begin{align}
\Prb ( \tau_{R} = \infty ) \geq 1-  \frac{1}{R^{1/4}} \label{eq:tau:R:THE:BOUND}
\end{align}
holds, for any $R>1$.  Note carefully that this lower bound in \eqref{eq:tau:R:THE:BOUND} is independent of $\alpha$. 
Thus if we wish to obtain that the local pathwise solution is global in time with high probability, i.e.
\begin{align*}
  \Prb( \xi = \infty) = 1 - \epsilon,
\end{align*}
for some $\epsilon \in (0,1)$, it is sufficient to choose $R$ so that
\begin{align}
\frac{1}{\epsilon^{4}} \leq R
\label{eq:smallLargeRHypoForEps}
\end{align}
and for this fixed $R$, consider an initial data $u_{0}$ which satisfies $\| u_{0}\|_{W^{m,p}} \leq \kappa(R,\alpha)$.  Alternatively for this $R$ and a \emph{given} (deterministic) 
initial data $\|u_0\|_{W^{m,p}}$ we may choose  $\alpha^2$ sufficiently large so that $\| u_{0}\|_{W^{m,p}} \leq \kappa(R,\alpha)$ to guarantee that the associated $(u,\xi)$
is global with probability $1-\epsilon$.
The proof of Theorem~\ref{thm:GlobalExistence2D3DLinearMultiplicative}, (ii), is now complete, modulo a proof of \eqref{eq:tau:R:THE:BOUND}, which we give next.

In order to estimate $\Prb(\tau_{R} = \infty)$, letting $\mu = \frac{3 \alpha^{2}}{8}$ we observe that 
\begin{align*}
\rho_\alpha(t) = \exp \left( \left(\mu - \frac{\alpha^{2}}{2}\right)t + \alpha W_{t} \right)
\end{align*} is a geometric Brownian motion, the solution of
\begin{align}
d x = \mu x dt + \alpha x  dW, \quad x(0) = 1,
\label{eq:GBM}
\end{align}
where $W$ is a standard $1-D$ Brownian motion. The following lemma, with $\mu = \frac{3 \alpha^{2}}{8}$, proves estimate \eqref{eq:tau:R:THE:BOUND}, and by the above discussion it concludes the proof of Theorem~\ref{thm:GlobalExistence2D3DLinearMultiplicative}.

\begin{lemma}[\bf Estimates for the exit times of geometric Brownian motion] \label{thm:GBMprops}Suppose that $\mu < \frac{\alpha^{2}}{2}$ and $x_{0} > 0$ and is {\em deterministic}. Let $x(t)$ be the solution of \eqref{eq:GBM} and
 for $R > 1$ define $\tau_{R}$ as
\begin{align}
	\tau_{R} = \inf \left\{t \geq 0 \colon x(t) > R \right\}.
	\label{eq:GrowthQuantGBM}
\end{align}
Then we have
\begin{align}
  \Prb( \tau_{R} = \infty) \geq 1- \left( \frac{1}{R} \right)^{1 - \frac{2 \mu}{\alpha^{2}}}
  \label{eq:ProbOfNeverHittingRGBMdisp}.
\end{align}
\end{lemma}
\begin{proof}[Proof of Lemma~\ref{thm:GBMprops}]
For $\lambda>0$ we apply the Ito formula for $f(x) = x^{\lambda}$ 
and obtain that
$$
	d x^{\lambda} = \lambda x^{\lambda-1} dx	+ \frac{\lambda (\lambda-1)}{2} x^{\lambda-2} dx dx
	= \left( \mu \lambda + \frac{\alpha^{2} \lambda (\lambda-1)}{2} \right) x^{\lambda} dt +
	\alpha \lambda x^{\lambda}dW.
$$
Integrating up to any time $t \wedge \tau_{R}$ and taking an expected value we find that
$$
 \E x^{\lambda}(t \wedge \tau_{R}) =  1
 	+ \E \int_{0}^{t \wedge \tau_{R}}  \left( \mu \lambda + \frac{\alpha^{2} \lambda (\lambda-1)}{2} \right)
		x^{\lambda} ds.
$$
Taking $\lambda = \lambda_{c} = 1 - \frac{2 \mu}{\alpha^{2}}$ in the above expression we find that
$$
 \E x^{\lambda_{c}}(t \wedge \tau_{R}) = 1.
$$
Now, using that $\tau_R$ is increasing in $R$ and the continuity of measures we get
\begin{align*}
   \Prb( \tau_{R} = \infty) = \Prb\left( \bigcap_{n} \{\tau_{R} > n\} \right)
   &= \lim_{N \rightarrow \infty} \Prb( \tau_{R} > N) 
   =  \lim_{N \rightarrow \infty} \Prb( x^{\lambda_{c}}(N \wedge \tau_{R}) < R^{\lambda_{c}} )\\
   &\geq  \lim_{N \rightarrow \infty} \left(1 -  \frac{\E x^{\lambda_{c}}(N \wedge \tau_{R}) }{ R^{\lambda_{c}}}\right) = 1 - \frac{ 1 }{ R^{\lambda_{c}}}
\end{align*}
which concludes the proof of the lemma.
\end{proof}

\appendix

\section{The smoothing operator and associated properties}
\setcounter{equation}{0}
\label{sec:SmoothingOperator}

In this appendix we define and review some basic properties of a class of
smoothing operators  $F_{\epsilon}$ as used in \cite{KatoLai1984}.
These mollifiers are used to construct solutions in $W^{m,p}$ in Section~\ref{sec:WmpSolutionsCons} above.   

For every $\epsilon > 0$, let  $\tilde{F}_{\epsilon}$ be a standard mollifier on $\RR^{d}$, 
for instance consider $\tilde{F}_{\epsilon}$ to be the convolution against the inverse Fourier transform of $\exp(-\epsilon |\xi|^2)$.  
Assuming $\partial \DD$ is sufficiently smooth, there exists (see for instance~\cite[Chapter 5]{AdamsFournier}) 
a linear extension operator $E$ from $\DD$ to $\RR^d$,  i.e.
$E u(x) = u(x)$ a.e. in $\DD$, and $\| E u \|_{W^{m,p}(\RR^d)} \leq C \| u \|_{W^{m,p}(\DD)}$ for  $m\geq 0$, and all $2\leq p < \infty$.
We also take $R$ to be a restriction operator, which is bounded from
$W^{m,p}(\RR^{d})$ into $W^{m,p}(\DD)$ for $m\geq 0$
and all $p \geq 2$.  Lastly, we let $P$ be the Leray projection operator
as defined in Section~\ref{sec:prelim}.  We finally define the smoothing operators $F_\epsilon$ by
\begin{align}
  F_{\epsilon} = P\; R\; \tilde{F}_{\epsilon}\; E   \label{eq:MollDef}
\end{align}
for every $\epsilon > 0$.   We have the following basic properties
for $F_{\epsilon}$.
\begin{lemma}[\bf Properties of the smoothing operator]\label{thm:MolProps}
Suppose that $m \geq 0$, and $p \geq 2$. For every $\epsilon > 0$ the operator $F_{\epsilon}$ maps $X_{m,p}$ into $X_{m'}$, where $m'=m+5$.  Moreover the following properties hold:
\begin{itemize}
\item[(i)]  The collection $F_{\epsilon}$ is uniformly bounded on $X_{m,p}$ independently of $\epsilon$
\begin{align}
   \| F_{\epsilon} u \|_{W^{m,p}} \leq C\|u \|_{W^{m,p}}, \quad u \in X_{m,p}
   \label{eq:MollUniformBnd}
\end{align}
where $C = C(m, p, \DD)$ is a universal constant independent of $\epsilon > 0$.
\item[(ii)] For every $\epsilon > 0$, when $m\geq 1$ we have
\begin{align}
   \| F_{\epsilon} u \|_{W^{m,p}} \leq \frac{C}{\epsilon} \|u \|_{W^{m-1,p}}, \quad u \in X_{m,p}
   \label{eq:MollDivEpsScaling}
\end{align}
and
\begin{align} 
 \| F_{\epsilon} u - u \|_{W^{m-1,p}} \leq C  \epsilon \|u \|_{W^{m,p}}, \quad u \in X_{m,p}
\end{align}
where $C =  C(m, p, \DD)$ is a universal constant independent of $\epsilon > 0$.
\item[(iii)] The sequence of mollifications $F_{\epsilon}u$
converge to $u$, for every $u$ in $X_{m,p}$, that is
\begin{align}
  \lim_{\epsilon \rightarrow 0} \| F_{\epsilon} u - u \|_{W^{m,p}}  = 0
 \label{eq:MollConvEspWmp}
\end{align}
and when $m \geq 1$ we also have
\begin{align}
  \lim_{\epsilon \rightarrow 0} \frac{1}{\epsilon} \| F_{\epsilon} u - u \|_{W^{m-1,p}}  = 0.
   \label{eq:MollConvEspWmmin1pGain}
\end{align}
\item[(iv)] The convergence of $F_\epsilon u$ to $u$ is uniform over compact subsets of $X_{m,p}$. In particular if $\{ u^k\}_{k\geq 1}$ is a sequence of functions in $X_{m,p}$ which converge in $X_{m,p}$, then we have
\begin{align} 
 \lim_{\epsilon \rightarrow 0} \sup_{k \geq 1} \| F_{\epsilon} u^k - u^k \|_{W^{m,p}}  = 0
\end{align}
and
\begin{align}
   \lim_{\epsilon \rightarrow 0}  \sup_{k \geq 1}  \frac{1}{\epsilon} \| F_{\epsilon} u^k - u^k \|_{W^{m-1,p}}  = 0,
   \label{eq:MollConvEspWmmin1pGain:Uniform}
\end{align}
when $m\geq 1$.
\end{itemize}
\end{lemma}
The above properties hold for $F_{\epsilon}$, since  they hold for the standard mollifier $\tilde{F}_\epsilon$ on $\RR^d$, 
we have that $R$ and $E$ are bounded maps between the relevant Sobolev spaces, and $RE = {\rm Id}_{\DD}$ a.e.. 
For further details, see for instance~\cite{AdamsFournier,KatoLai1984}.

\section{A technical lemma about ODEs}
\label{sec:ODE}
\setcounter{equation}{0}

In this appendix we give the proof of a technical lemma which was used in proving the $3D$ case of Theorem~\ref{thm:GlobalExistence2D3DLinearMultiplicative}, in Section~\ref{sec:GlobalExistence:3D} above. The {\em reason d'\^{e}tre} of the below lemma is to very carefully keep track of the dependence on $\alpha$ for all constants involved. This enables us to control the quantities involved as the parameter $\alpha$ is sent to either $0$ or $\infty$. 
\begin{lemma}
\label{lemma:ODE:log}
Let $\bar C\geq 1$ be a universal constant. Fix the parameters $R \geq 1,\alpha \neq 0$ and  $T>0$.
For $y_{0} > 0$, let $y(t)$ be a positive smooth function satisfying
\begin{align}
& \frac{dy}{dt}(t)  \leq  \bar{C} R \exp\left( - \frac{\alpha^{2} t}{8} \right) y(t)
	\left(\bar C + \alpha^{2} + z(t) \log y(t)  \right)  \label{eq:app:C:1}\\
&y(0) = y_{0}\label{eq:app:C:2}
\end{align}
where $z(t)$ is a {\em given} continuous function such that $0 < z(t) \leq \alpha^{2}/4$ for all $t\in [0,T]$.
There exits a positive function $K(R ,\alpha) \geq 2$ such that if
\begin{align}
y_0 \leq \frac{\alpha^{2}}{4 \bar C K(R ,\alpha)} \label{eq:app:C:IC}
\end{align}
then we have
\begin{align}
y(t) \leq \frac{K(R ,\alpha)}{2 R} y_{0} \leq \frac{\alpha^{2}}{8R \bar C} \label{eq:app:C:thebound}
\end{align}
for all $t \in [0,T)$.  This function $K(R,\alpha)$ may be chosen explicitly as
\begin{align}
K (R , \alpha)  =  
2 R
\left(  1 + \left(\frac{\alpha^2}{8 \bar C} \right)^{(1- \frac{1}{8(D_R -1)})}\right)
\exp \left( \frac{ 8 \bar C R D_R (\bar C + \alpha^{2})}{\alpha^{2}}\right)
\label{eq:app:C:K:def}
\end{align}
where we have denoted $D_R = \exp(4 \bar C R)$. Additionally, for every {\em fixed} $R \geq 1$ we obtain the asymptotic behavior for the function $$\kappa(R,\alpha) = \frac{\alpha^2}{2 \bar C K(R,\alpha)}$$ to be 
\begin{align}
&\lim_{\alpha^2 \to \infty} \kappa(R,\alpha) = \lim_{\alpha^2 \to \infty} \frac{\alpha^{2}}{K( R, \alpha)} = \infty, \label{eq:app:C:asymptotics} \\
&\lim_{\alpha^2 \to 0} \kappa(R,\alpha) = \lim_{\alpha^2 \to 0} \frac{\alpha^{2}}{K( R, \alpha)} = 0.
\end{align}
\end{lemma}
\begin{proof}[Proof of Lemma~\ref{lemma:ODE:log}]
For ease of notation, let $a(t) = \bar C R \exp(-\alpha^{2}t/8)$.
After letting $Y(t) = \log y(t)$, the inequality \eqref{eq:app:C:1} reads
\begin{align}
\frac{dY(t)}{dt} \leq a(t) \left( (\bar C + \alpha^{2}) + z(t) Y(t) \right) \label{eq:app:C:Y}
\end{align}
with initial condition $Y(0) = \log y_{0}$. The initial value problem associated to \eqref{eq:app:C:Y} leads to the bound
\begin{align}
Y(t) &\leq Y(0) \exp\left( \int_{0}^{t} a(s) z(s) ds \right) + (\bar C + \alpha^{2}) \int_{0}^{t} a(s) \exp\left( \int_{s}^{t} a(s') z(s') ds'\right) ds \notag\\
&\leq Y(0) \exp\left( \int_{0}^{t} a(s) z(s) ds \right) + (\bar C + \alpha^{2}) \exp(2 \bar C R) \int_{0}^{t} a(s) ds \notag\\
&\leq Y(0) \exp\left( \int_{0}^{t} a(s) z(s) ds \right) + \frac{ 8 \bar C R  (\bar C + \alpha^{2}) \exp(2 \bar C R)}{\alpha^{2}}  \label{eq:app:C:Ybound}
\end{align}
where we used the a priori bound $z \leq \alpha^{2}/4$ and the identity $\int_{0}^{\infty} a(t) dt = 8 \bar C R / \alpha^{2}$. 
By exponentiation it follows that
\begin{align}
y(t)
	\leq  y_{0}^{\exp\left( \int_{0}^{t} a(s) z(s) ds \right) }
		\exp \left( \frac{ 8 \bar C R (\bar C + \alpha^{2}) \exp(2 \bar C R)}{\alpha^{2}} \right) \label{eq:app:C:y:bound1}.
\end{align}
We note that if $y_{0} \leq 1$, since $\exp\left( \int_{0}^{t} a(s) z(s) ds \right) \geq 1$, we have
\begin{align}
y_{0}^{\exp\left( \int_{0}^{t} a(s) z(s) ds \right) } \leq y_{0}.\label{eq:app:C:y:bound2}
\end{align}
On the other hand, if $y_{0} > 1$, due to \eqref{eq:app:C:IC} we may bound
\begin{align}
y_{0} \leq \frac{\alpha^{2}}{M}, \label{eq:app:C:IC:bound}
\end{align}
whenever $M \leq 4 \bar C K$. Hence, recalling the a priori bound on $z(t)$ and integrating $a(t)$ from $0$ to $\infty$,
we obtain from \eqref{eq:app:C:y:bound1} and \eqref{eq:app:C:IC:bound} that
\begin{align}
y_{0}^{\exp\left( \int_{0}^{t} a(s) z(s) ds \right) }
\leq y_{0}^{D_{R}}
\leq y_{0} y_{0}^{D_{R}-1}
\leq y_{0} \left( \frac{\alpha^{2}}{M}\right)^{D_{R}-1} \label{eq:app:C:y:bound3}
\end{align}
for $y_{0} > 1$, since $ D_{R}  = \exp(2 \bar C R) \geq 3$.
Hence, we obtain from \eqref{eq:app:C:y:bound1}, \eqref{eq:app:C:y:bound2}, and \eqref{eq:app:C:y:bound3} that
\begin{align}
y(t)
\leq
	y_{0} \frac{1}{2R}
			\left(
				2 R
				\left(1 +\left( \frac{\alpha^{2}}{M}\right)^{D_{R} -1} \right)
				\exp \left( \frac{ 8 \bar C C_{*} D_{R} (\bar C + \alpha^{2})}{\alpha^{2}}\right)
			\right)
=:  y_{0} \frac{1}{2R} \bar K(M). \label{eq:app:K:bar:def}
\end{align}
The proof of \eqref{eq:app:C:thebound} is completed if we show that $\bar K(M) \leq K$ for all $\alpha >0$, for some $M$ is chosen such that $M\leq 4 \bar C K$. We now let
\begin{align}
M=  8 \bar C  \indFn{\alpha^{2} \leq 8 \bar C} + \indFn{\alpha^{2} > 8 \bar C} (8 \bar C)^{\frac{1}{2(D_{R}-1)}} \alpha^{(2 - \frac{1}{D_{R}-1})}
\label{eq:app:C:M:def}
\end{align}
and 
define
\begin{align*}
K (R , \alpha)  = 
2 R
\left(  1 + \left(\frac{\alpha^2}{8 \bar C} \right)^{(1- \frac{1}{8(D_R -1)})}\right)
\exp \left( \frac{ 8 \bar C R D_R (\bar C + \alpha^{2})}{\alpha^{2}}\right).
\end{align*}
Indeed, it is not hard to verify that for $R \geq 1$, and $\bar C \geq 1$,
we have $4 \bar C K \geq M$ for all $\alpha>0$. Lastly, to verify that the above defined $K$ indeed is larger than $\bar K(M)$ (which was defined in \eqref{eq:app:K:bar:def}), it is sufficient to check that
\begin{align}
\left( \frac{\alpha^{2}}{M}\right)^{D_{R} -1} \leq  \left(\frac{\alpha^2}{8 \bar C} \right)^{(1- \frac{1}{8(D_{R}-1)})}
\label{eq:app:C:tocheck}
\end{align}
for all $\alpha >0$. Indeed, \eqref{eq:app:C:tocheck} may be checked by a direct computation using \eqref{eq:app:C:M:def} and $D_{R} \geq 3$.

Lastly, one may directly check that for any fixed $R\geq 1$, as $\alpha \rightarrow \infty$ we have $K(R,\alpha) = {O}(\alpha^{2 - \frac{1}{4(D_{R}-1)}})$,
and therefore $\alpha^{2} / K(R,\alpha) \rightarrow \infty$,  as $\alpha \to \infty$, which concludes the proof of \eqref{eq:app:C:asymptotics}. To conclude, it is clear from the definition of $K(R\alpha)$ that it is larger than $2$, and hence $\alpha^{2}/K(R,\alpha) \to 0$ as $\alpha \to 0$, which concludes the proof of the lemma.
\end{proof}

\section{A non-blowup condition for SDEs with linear-logarithmic growth in the drift}
\label{sec:Gronwal:SODE}
\setcounter{equation}{0}

In this section we state and prove a condition for the non-blow up
of solutions to SODEs via a Logarithmic Gr\"onwall type argument.
See e.g. \cite{FangZhang2005} for related results.
\begin{lemma}\label{thm:NonBlowupLogLipp}
Fix a stochastic basis  $\mathcal{S} := (\Omega, \mathcal{F}, \Prb,$ $
\{\mathcal{F}_t\}_{t \geq 0}, \WW)$.  Suppose that on $\mathcal{S}$
we have defined $Y$ a real valued, predictable process defined up to a blow up time $\xi > 0$,
i.e. for all bounded stopping times $\tau < \xi$,
$
\sup_{t \in [0,\tau]} Y < \infty
$
a.s. and
$$
\sup_{t \in [0, \xi)} Y = \infty
\quad \textrm{ on the set } \{\xi < \infty\}.
$$
Assume
that $Y \geq 1$ and that on $[0, \xi)$, $Y$ satisfies the It\={o} stochastic differential
\begin{align}
	dY = Xdt + Z \dW, \quad Y(0) = Y_{0},
\end{align}
where on $[0, \xi)$, $X$, $Z$ are respectively real valued and
$L_{2}$ valued predictable processes and $Y_{0}$ is $\mathcal{F}_{0}$
and bounded above by a deterministic constant $M > 0$.\footnote{This 
condition is not essential; we may merely assume
that $Y_{0} < \infty$, almost surely.  See Remark~\ref{rmk:InitialDataCrazyness}
below}
Suppose that there exists a stochastic process 
\begin{align}
\eta \in L^1(\Omega; L^1_{loc} [0,\infty))
\label{eq:ExtraProcessToCauseMentalBreakdown}
\end{align}
with $\eta \geq 1$ for almost every $(\omega, t)$ 
and
an increasing collection of stopping times
$\tau_{R}$ with $\tau_{R} \leq \xi$ and such that
\begin{align}
 \Prb \left( \bigcap_{R > 0} \{\tau_{R} < \xi \wedge T \} \right) =0.
    \label{eq:WeHitXiWithFiniteR}
\end{align}
We further assume that for every fixed $R >0$, 
there exists a deterministic constant $K_{R}$ 
depending only on $R$ (independent of $t$),
and a number $r \in [0, 1/2]$ such that,
\begin{align*}
	|X| \leq& K_{R} ((1 + \log Y) Y + \eta),\\
	\|Z\|_{L_{2}} \leq& K_{R}  Y^{1-r} \eta^{r}
\end{align*}
which holds over $[0,\tau_{R}]$.
Then $\xi = \infty$ and in particular, 
$
	\sup_{t \in [0,T]} Y < \infty, \, \textrm{ a.s. }
$
for every $T > 0$.
\end{lemma}
\begin{proof}
As in \cite{FangZhang2005}, we introduce the functions
\begin{align}
\zeta (x) &= (1 + \ln x ) \notag \\ 
\Psi (x) &= \int_{0}^{x} \frac{1}{r \zeta(r) + 1} d r \notag \\ 
\Phi (x) &= \exp (\Psi (x) ).\label{eq:2D:phiR}
\end{align}
By direct computation we find that
\begin{align*}
\Phi'(x) = \frac{\Phi(x)}{x  \zeta(x) + 1}, \quad 
\Phi''(x) = - \frac{\Phi(x) \zeta(x)}{(x \zeta(x) +1)^{2}}.
\end{align*}
Thus, by an application of the It\={o} lemma, we have
\begin{align*}
d \Phi(Y) = \Phi'(Y)dY + \frac{1}{2}\Phi''(Y)dYdY
	=  \frac{\Phi(Y)}{Y  \zeta(Y) + 1} X dt
	- \frac{1}{2} \frac{\Phi(Y) \zeta(Y)}{(Y \zeta(Y) +1)^{2}} \|Z\|_{L_{2}}^{2} dt + \frac{\Phi(Y)}{Y  \zeta(Y) + 1}  Z {\dW}.
\end{align*}
For $S >0$ we define the stopping times 
$$
\zeta_{S} := \inf\{t \geq 0: Y(t) > S\}  \wedge \tau_{R},
\quad \rho_{S} :=  \inf\left\{ t \geq 0: \int_0^t \eta ds > S \right\}.
$$
In view of the definition of $\xi$, we have that $\lim\limits_{S \to \infty} \zeta_{S} = \tau_{R} \wedge \xi$.
Due to \eqref{eq:ExtraProcessToCauseMentalBreakdown} we also have that $\lim\limits_{S \to \infty} \rho_{S} = \infty$.
Fix $T, S_1, S_2 >0$. We estimate and any stopping times $0 \leq \tau_a \leq \tau_b \leq \zeta_{S_1} \wedge \rho_{S_2} \wedge T$
\begin{align*}
  \E \sup_{t \in [\tau_a,\tau_b]} \Phi(Y)
  \leq& \E\Phi(Y(\tau_a))  +
        \E \int_{\tau_a}^{\tau_b} \Phi(Y)  \left(  \frac{|X|}{Y  \zeta(Y) + 1}  +
	 \frac{1}{2}  \left| \frac{ \zeta(Y)\|Z\|_{L_{2}}^{2}  }{(Y \zeta(Y) +1)^{2}}  \right| \right)dt +
	\E \sup_{t \in [\tau_a,\tau_b]} \left|\int_{\tau_a}^{t}\frac{\Phi(Y)}{Y  \zeta(Y) + 1}  Z {\dW} \right|, \notag\\
  \leq& \E\Phi(Y(\tau_a))  +
       C \E \int_{\tau_a}^{\tau_b} \Phi(Y)  \left(  1 + \eta \right)dt +
	C\E \left(\int_{\tau_a}^{\tau_b} \left(\frac{\Phi(Y)}{Y  \zeta(Y) + 1}\right)^{2} \|Z\|_{L_{2}}^{2}  dt\right)^{1/2}, \notag\\	
  \leq& \E\Phi(Y(\tau_a))  +
       C \E \int_{\tau_a}^{\tau_b} \Phi(Y)  \left(  1 + \eta \right)dt +
	C\E \left(\int_{\tau_a}^{\tau_b} \Phi(Y)^{2} \eta  dt \right)^{1/2}, \notag\\	
  \leq& \E\Phi(Y(\tau_a))  +
       C \E \int_{\tau_a}^{\tau_b} \Phi(Y)  \left(  1 + \eta \right)dt +
        \frac{1}{2} \E \sup_{t \in [\tau_a,\tau_b]} \Phi(Y)
\end{align*}
where $C$, depends on $R$ through $K_{R}$ and is
is independent of $T$, $S_{1}$, $\xi$, $\tau_a$ and $\tau_b$.  Rearranging
and applying a stochastic version of the Gr\"{o}nwall Lemma given 
in \cite[Lemma 5.3]{GlattHoltzZiane2} we find
\begin{align*}
  \E \sup_{t \in [0,\sigma_{S_1} \wedge \rho_{S_2} \wedge T]} \Phi(Y) \leq C
\end{align*}
where here $C = C(R, T,S_2, M)$ and is independent of $S_1$ and $\xi$.
Thus, sending $S_1 \to \infty$ and applying the monotone convergence theorem, 
\begin{align}
  \E \sup_{t \in [0, \rho_{S_2} \wedge \tau_{R} \wedge T]} \Phi(Y) \leq C. 
  \label{eq:GronwallLambada}
\end{align}
Thus, by the properties of $\Phi$ (cf. \eqref{eq:2D:phiR}) we infer
$$
\sup_{t \in [0, \rho_{S_2} \wedge \tau_{R} \wedge T]} Y < \infty \; \textrm{ for each } R, S_{2} >0,
$$
on a set of full measure.
Thus, since $\lim\limits_{S_{2} \to \infty} \rho_{S_{2}} = \infty$
we infer that, for each
$R > 0$, $\sup_{t \in [0,\tau_{R} \wedge T]} Y < \infty$, almost surely.  In view of the
condition \eqref{eq:WeHitXiWithFiniteR} imposed on the stopping times $\tau_{R}$ this in turn implies
$\sup_{t \in [0,\xi \wedge T]} Y < \infty$.
Since $T$ was also arbitrary to begin with, we have perforce
$\xi = \infty$, almost surely.  The proof is therefore complete.
\end{proof}
\begin{remark}\label{rmk:InitialDataCrazyness}
In Lemma~\ref{thm:NonBlowupLogLipp} we may actually just
assume that $Y_{0}$ is finite almost surely.  Indeed if we define
the sets $\Omega_{M} := \{ Y_{0} \leq M\}$ we infer, arguing 
similarly to above that
$$
  \E\left( \indFn{\Omega_{M}}\sup_{t \in [0,\zeta_{S_1} \wedge \rho_{S_2} \wedge T]} \Phi(Y) \right)\leq C_{M}.
$$
We thus find that ${\xi = \infty}$ for almost every $\omega$ in $\cap_{M} \Omega_{M}$.  Since
this latter set is clearly of full measure, we may thus establish the proof of Lemma~\ref{thm:NonBlowupLogLipp} 
in this more general situation.
\end{remark}

\subsection*{Acknowledgments}

The work of NGH was partially supported by the National Science Foundation
under the grants NSF-DMS-1004638, NSF-DMS-0906440,
and by the Research Fund of Indiana University. VV is grateful to Peter Constantin for stimulating discussions,  to the Mathematics Department and the Institute for Scientific Computing and Applied Mathematics at Indiana University for their hospitality. VV was in part supported by an AMS-Simmons travel award.

\end{document}